\documentclass[a4paper,reqno,11pt]{amsart}
\usepackage{amssymb}
\usepackage{amstext}
\usepackage{amsmath}
\usepackage{amscd}
\usepackage{amsthm}
\usepackage{amsfonts}
\usepackage{pict2e}
\usepackage{enumerate}
\usepackage{graphicx}
\usepackage{latexsym}
\usepackage{mathrsfs}
\usepackage{slashbox}
\input xy
\xyoption{all}
\usepackage{pstricks}
\usepackage{lscape}
\usepackage{comment}

\makeatletter

\@addtoreset{equation}{section}
\makeatother
\newtheorem{theorem}{Theorem}[section]

\newtheorem{corollary}[theorem]{Corollary}
\newtheorem{lemma}[theorem]{Lemma}
\newtheorem{proposition}[theorem]{Proposition}
\newtheorem{definition-theorem}[theorem]{Definition-Theorem}
\newtheorem{definition-proposition}[theorem]{Definition-Proposition}

\newtheorem{mainthm}{Theorem}

\theoremstyle{definition}
\newtheorem{definition}[theorem]{Definition}
\newtheorem{remark}[theorem]{Remark}
\newtheorem{example}[theorem]{Example}

\newtheorem{notation}[theorem]{Notation}
\newtheorem{algorithm}[theorem]{Algorithm}

\newcommand{\E}{{\mathcal E}}

\newcommand{\G}{{\mathcal G}}

\newcommand{\M}{{\mathcal M}}
\newcommand{\N}{{\mathcal N}}

\renewcommand{\S}{{\mathcal S}}
\newcommand{\T}{{\mathcal T}}

\newcommand{\Y}{{\mathcal Y}}
\newcommand{\Z}{{\mathcal Z}}
\newcommand{\ZZ}{\mathbb{Z}}

\newcommand{\add}{\mathsf{add}\hspace{.01in}}
\newcommand{\Fac}{\mathsf{Fac}\hspace{.01in}}

\renewcommand{\mod}{\mathsf{mod}\hspace{.01in}}

\newcommand{\Arc}{\operatorname{arc}\nolimits}

\newcommand{\Hom}{\operatorname{Hom}\nolimits}

\newcommand{\ind}{\operatorname{ind}\nolimits}
\newcommand{\inn}{\operatorname{in}\nolimits}

\newcommand{\np}{\operatorname{np}\nolimits}
\newcommand{\op}{\operatorname{op}\nolimits}

\newcommand{\pr}{\operatorname{pr}\nolimits}
\newcommand{\rad}{\operatorname{rad}\nolimits}
\newcommand{\s}{\operatorname{s}\nolimits}
\newcommand{\soc}{\operatorname{soc}\nolimits}

\renewcommand{\top}{\operatorname{top}\nolimits}

\newcommand{\psttilt}{\mbox{\rm ps$\tau$-tilt}\hspace{.01in}}

\newcommand{\sttilt}{\mbox{\rm s$\tau$-tilt}\hspace{.01in}}

\newcommand{\tilt}{\mbox{\rm tilt}\hspace{.01in}}
\newcommand{\trigid}{\mbox{\rm $\tau$-rigid}\hspace{.01in}}
\newcommand{\ttilt}{\mbox{\rm $\tau$-tilt}\hspace{.01in}}

\newcommand{\bQ}{\overline{Q}}
\newcommand{\bLambda}{\overline{\Lambda}}
\newcommand{\hasse}{{\rm H}}

\newcommand{\po}{\Omega}

\newcommand{\arr}{\operatorname{a}\nolimits}

\newcommand{\xto}{\xrightarrow}
\begin{document}
\title[$\tau$-tilting modules over Nakayama algebras]{The classification of $\tau$-tilting modules over Nakayama algebras}
\author[Takahide Adachi]{Takahide Adachi}
\address{Graduate School of Mathematics, Nagoya University, Frocho, Chikusaku, Nagoya, 464-8602, Japan}
\email{m09002b@math.nagoya-u.ac.jp}
\begin{abstract}
In this paper, we study $\tau$-tilting modules over Nakayama algebras.
We establish bijections between $\tau$-tilting modules, triangulations of a polygon with a puncture, and certain integer sequences.
Moreover, we give an algorithm to construct the Hasse quiver of support $\tau$-tilting modules by using Drozd-Kirichenko rejection.
\end{abstract}
\maketitle
\section*{Introduction}
In representation theory of algebras, tilting modules are important because they control derived equivalences.
An effective method to construct tilting modules is given by mutation \cite{RS,U}.
It is an operation to replace an indecomposable direct summand of a given tilting module to get a new one.
However, it is known that mutation of tilting modules is often impossible depending on a choice of indecomposable direct summands.
To improve behavior of mutation of tilting modules, 
the authors of \cite{AIR} introduced the notion of (support) $\tau$-tilting modules as a generalization of tilting modules.
In fact, they showed that mutation of support $\tau$-tilting modules is always possible.
In addition, $\tau$-tilting modules satisfy nice properties of tilting modules.
For example, the set of support $\tau$-tilting modules has a natural structure of a partially ordered set (poset, for short),
and the Hasse quiver of this poset coincides with the mutation quiver of support $\tau$-tilting modules.
This is an analog of a result by Happel-Unger \cite{HU} for tilting modules.
It is also known that there are close relationships 
between $\tau$-tilting modules and some important notions in representation theory,
{\it e.g.}, torsion classes, silting complexes (see \cite{AI,KV}) and cluster-tilting objects (see \cite{BMRRT}).
Therefore it is important to give a classification of (support) $\tau$-tilting modules.
For results on $\tau$-tilting modules, we refer to \cite{AIR, BY, J1, Ma, Mi1, Mi2, Z} for instance.

The aim of this paper is to classify $\tau$-tilting modules over any Nakayama algebra $\Lambda$,
and moreover to give an algorithm to construct the Hasse quiver of support $\tau$-tilting $\Lambda$-modules.
The following theorem is our first main result, which gives combinatorial descriptions of support $\tau$-tilting $\Lambda$-modules.
\begin{mainthm}[Theorems \ref{2-11}, \ref{2-6} and \ref{2-8}]\label{1}
Let $\Lambda$ be a Nakayama algebra with $n$ simple modules.
Assume that the Loewy length of each indecomposable projective $\Lambda$-module is at least $n$.
Then there are bijections between
\begin{itemize}
\item[(1)] the set $\ttilt\Lambda$ of isomorphism classes of basic $\tau$-tilting $\Lambda$-modules,
\item[(2)] the set $\psttilt\Lambda$ of isomorphism classes of basic proper support $\tau$-tilting $\Lambda$-modules,
\item[(3)] the set $\T(n)$ of triangulations of an $n$-regular polygon with a puncture,
\item[(4)] the set $\Z(n)$ of sequences $(a_{1},a_{2},\ldots, a_{n})$ of nonnegative integers with $\sum_{i=1}^{n}a_{i}=n$.
\end{itemize}
\end{mainthm}
We give a more general result for an arbitrary Nakayama algebra (see Theorem \ref{2-14}).
Moreover, we give a bijection between the set $\sttilt\Lambda$ of isomorphism classes of basic support $\tau$-tilting $\Lambda$-modules 
and the set of signed triangulations introduced recently by \cite{BS} and \cite{IN} (See Corollary \ref{bij}).
Note that Theorem \ref{1} is analogous to known classification results in representation theory:
tilting modules for cyclic quivers \cite{BK}, torsion pairs in tube categories \cite{BBM}, 
cluster-tilting objects in cluster categories of type $A$ \cite{CCS} and $D$ \cite{S}, 
cluster-tilting modules over self-injective algebras and Gorenstein orders of type $A$ and $D$ \cite{I}.

Next, we study a relationship between support $\tau$-tilting $\Lambda$-modules over an algebra $\Lambda$ (not necessarily Nakayama) and 
those over the factor algebra $\bLambda=\Lambda/\soc Q$, where $Q$ is an indecomposable projective-injective $\Lambda$-module.
The following theorem is the second main result in this paper.
\begin{mainthm}[Theorems \ref{3-4} and \ref{3-5} for details]\label{2}
Let $\Lambda$ be a basic finite dimensional algebra, 
$Q$ an indecomposable projective-injective $\Lambda$-module, and $\bLambda:=\Lambda/\soc Q$.
Then all support $\tau$-tilting modules of $\Lambda$ can be obtained explicitly from those of $\bLambda$.
Moreover, the Hasse quiver of support $\tau$-tilting modules of $\Lambda$ can be constructed explicitly from that of $\bLambda$.
\end{mainthm}
As an application, we have an algorithm to construct the Hasse quivers of Nakayama algebras.
Each Nakayama algebra $\Lambda$ always has an indecomposable projective-injective $\Lambda$-module $Q$, and 
the factor algebra $\bLambda$ is a Nakayama algebra again.
Therefore we can iteratively apply Theorem \ref{2}, and we have the following.
\begin{mainthm}[Algorithm \ref{3-9} for details]\label{3}
Let $\Lambda$ be a Nakayama algebra.
Then there exists an algorithm to construct the Hasse quiver of support $\tau$-tilting $\Lambda$-modules.
\end{mainthm}
\subsection*{Notation}
Throughout this paper, $K$ is an algebraically closed field.
By an algebra we mean basic and finite dimensional $K$-algebra and by a module we mean a finitely generated right module.
We denote by $[i,j]$ the interval $\{ i,i+1,\ldots,j-1,j\}$ of integers.
Let $\{ e_{1}, e_{2}, \ldots, e_{n}\}$ be a complete set of primitive orthogonal idempotents of an algebra $\Lambda$
and $\E_{\Lambda}:=\{ \sum_{i\in I}e_{i}\ |\ I\subset [1,n] \}$.
For each $i\in[1, n]$, we put $P_{i}=e_{i}\Lambda$, $S_{i}=P_{i}/\rad P_{i}$ and $E_{i}=\Hom_{K}(\Lambda e_{i},K)$.
We denote by $\mod \Lambda$ the category of finitely generated right $\Lambda$-modules and 
by $\add M$ the full subcategory of $\mod\Lambda$ consisting of direct summands of finite direct sums of copies of $M\in\mod\Lambda$.
We denote by $\tau_{\Lambda}$ the Auslander-Reiten translation of $\Lambda$ 
and by $\langle e \rangle$ a two-sided ideal of $\Lambda$ generated by $e\in \Lambda$.
For two sets $X$ and $Y$, we denote by $X\coprod Y$ the disjoint union.
\subsection*{Acknowledgements}
The author wishes to express his sincere gratitude to Osamu Iyama for valuable advices and suggestions.
He is grateful to Kota Yamaura and Gustavo Jasso for helpful discussions.
He thanks Takuma Aihara and Yuya Mizuno for useful comments.
He would like to thank Tomoki Nakanishi for suggesting him to consider signed triangulations.
\section{Preliminaries}
Let $\Lambda$ be a basic finite dimensional algebra over an algebraically closed field.
In this section, we collect basic results which are necessary in this paper.

We may regard $(\Lambda/I)$-modules as $\Lambda$-modules by the following lemma.
\begin{lemma}\cite[A.6.1]{ASS}
Let $I$ be a two-sided ideal of $\Lambda$.
Then the natural surjection $\Lambda\rightarrow \Lambda/I$ induces a fully faithful functor $\mod(\Lambda/I)\rightarrow \mod \Lambda$.
\end{lemma}
The following elementary lemma is often used.
\begin{lemma}\label{1-1}
Let $M\in\mod\Lambda$ be indecomposable, $P\in\mod\Lambda$ indecomposable projective, and $E\in\mod\Lambda$ indecomposable injective.
\begin{itemize}
\item[(1)] \cite[II.1]{ARS} The following hold.
\begin{itemize}
\item[(a)] $\Hom_{\Lambda}(P,M)\neq 0$ if and only if $M$ has $\top P $ as a composition factor.
\item[(b)] $\Hom_{\Lambda}(M,E)\neq 0$ if and only if $M$ has $\soc E$ as a composition factor.
\end{itemize}
\item[(2)] \cite[IV.3.5]{ASS} The following hold.
\begin{itemize}
\item[(a)] Assume that $M\not\simeq P$. 
The natural injection $\rad P\rightarrow P$ induces an isomorphism $\Hom_{\Lambda}(M,\rad P)\simeq \Hom_{\Lambda}(M,P)$.
\item[(b)] Assume that $M\not\simeq E$. 
The natural surjection $E\rightarrow E/\soc E$ induces an isomorphism $\Hom_{\Lambda}(E/\soc E,M)\simeq \Hom_{\Lambda}(E,M)$.
\end{itemize}  
\end{itemize}
\end{lemma}
We recall the definition and basic properties of $\tau$-tilting modules.
For more details, we refer to \cite{AIR}.
We denote by $|M|$ the number of pairwise nonisomorphic indecomposable summands of a $\Lambda$-module $M$.
\begin{definition}
\begin{itemize}
\item[(1)] We call $M$ in $\mod\Lambda$ {\it{$\tau$-rigid}} if $\Hom_{\Lambda}(M,\tau M)=0$.
\item[(2)] We call $M$ in $\mod\Lambda$ {\it{$\tau$-tilting}} if it is $\tau$-rigid and $|M|=|\Lambda|$.
\end{itemize}
\end{definition}
It is known that
\begin{itemize}
\item[(i)] If $M$ is $\tau$-rigid, then $|M|\le |\Lambda|$.
\item[(ii)] Each $\tau$-tilting $\Lambda$-module $M$ is {\it sincere} ({\it i.e.}, every simple $\Lambda$-module appears as a composition factor in $M$).
\end{itemize}

For $\tau$-tilting modules, we have an analog of Bongartz Lemma for tilting modules.
\begin{proposition}\cite[Theorem 2.10 and 2.12]{AIR}\label{air2.10}
Let $M$ be a $\tau$-rigid $\Lambda$-module.
Then there exists $N\in\mod\Lambda$ such that $M\oplus N$ is a $\tau$-tilting $\Lambda$-module.
Moreover, $M$ is a $\tau$-tilting $\Lambda$-module if and only if it is a maximal $\tau$-rigid $\Lambda$-module 
(i.e., if $M\oplus L$ is $\tau$-rigid for some $\Lambda$-module $L$, then $L\in\add M$).
\end{proposition}
\begin{definition}\label{air0.1}
We call $M$ in $\mod\Lambda$ {\it{support $\tau$-tilting}} if there exists an idempotent $e_{M}\in \Lambda$ such that 
$M$ is a $\tau$-tilting $(\Lambda/\langle e_{M} \rangle)$-module.
Note that $e_{M}$ can be chosen uniquely from $\E_{\Lambda}$ since $M$ is a sincere $(\Lambda/\langle e_{M} \rangle)$-module by (ii) above.
If moreover $e_{M}\neq 0$, $M$ is called a {\it{proper support $\tau$-tilting}} $\Lambda$-module.
\end{definition}

Throughout this paper, we denote by $\tilt\Lambda$ (respectively, $\ttilt \Lambda$, $\sttilt \Lambda$, $\psttilt\Lambda$)
the set of isomorphism classes of basic tilting (respectively, $\tau$-tilting, support $\tau$-tilting, proper support $\tau$-tilting) $\Lambda$-modules.
The following observations are clear.
\begin{proposition}\label{1-2}
\begin{itemize}
\item[(1)] $\sttilt\Lambda=\ttilt\Lambda \coprod \psttilt\Lambda$.
\item[(2)] $\psttilt\Lambda =\coprod_{e\in \E_{\Lambda}\setminus\{ 0\}}\ttilt (\Lambda/\langle e \rangle)$.
\item[(3)] If $\Lambda$ is hereditary, then $\ttilt\Lambda=\tilt\Lambda$.
\end{itemize}
\end{proposition}

The following lemma is useful.
\begin{lemma}\cite[Lemma 2.1]{AIR}\label{air2.1}
Let $I$ be a two-sided ideal of $\Lambda$, and $M,N\in \mod(\Lambda/I)$.
Then the following hold.
\begin{itemize}
\item[(1)] If $\Hom_{\Lambda}(N,\tau_{\Lambda} M)=0$, then $\Hom_{\Lambda/I}(N,\tau_{\Lambda/I}M)=0$.
\item[(2)] The converse of $(1)$ holds if $I=\langle e\rangle$ for an idempotent $e\in\Lambda$.
\end{itemize}
\end{lemma}
We give a criterion for $\tau$-rigid modules to be support $\tau$-tilting modules.
We denote by $\s(M)$ the number of nonisomorphic simple modules appearing in a composition series of $M\in\mod\Lambda$.
Then $M$ is sincere if and only if $\s(M)=|\Lambda|$.
\begin{proposition}\label{1-3}
Let $M$ be a $\tau$-rigid $\Lambda$-module.
Then the following are equivalent:
\begin{itemize}
\item[(1)] $M$ is a support $\tau$-tilting $\Lambda$-module.
\item[(2)] $|M|=\s(M)$.
\end{itemize}
\end{proposition}
\begin{proof}
Let $e\in\Lambda$ be a maximal idempotent such that $\Hom_{\Lambda}(e\Lambda, M)=0$.
Namely, $M$ does not have $\top e\Lambda$ as a composition factor by Lemma \ref{1-1}(1).
Then we have $\s(M)=|\Lambda|-|e\Lambda|$.
Thus (1)$\Rightarrow$(2) holds clearly.
On the other hand, (2)$\Rightarrow$(1) holds since $M$ is a $\tau$-rigid $(\Lambda/\langle e\rangle)$-module by Lemma \ref{air2.1}(1).
\end{proof}
Let $M$ be a $\Lambda$-module and $P$ a projective $\Lambda$-module.
We call $(M,P)$ an {\it almost complete support $\tau$-tilting pair} if $M$ is $\tau$-rigid, $\Hom_{\Lambda}(P,M)=0$, and $|M|+|P|=|\Lambda|-1$.
\begin{proposition}\cite[Theorem 2.18]{AIR}\label{air2.18}
For any basic almost complete support $\tau$-tilting pair $(M,P)$ for $\Lambda$,  
there exist exactly two basic support $\tau$-tilting $\Lambda$-modules $L$ and $N$ 
such that $M\in \add L\cap \add N$ and $P\in\add e_{L}\Lambda \cap \add e_{N}\Lambda$.
\end{proposition}
Two basic support $\tau$-tilting modules $L$ and $N$ are said to be \emph{mutations} of each other 
if there exists a basic almost complete support $\tau$-tilting pair $(M,P)$ such that 
$M$ (respectively, $P$) is a direct summand of $L$ and $N$ (respectively, $e_{L}\Lambda$ and $e_{N}\Lambda$).

For a $\Lambda$-module $M$, we denote by $\Fac(M)$ the full subcategory of $\mod \Lambda$ consisting of factor modules of direct sums of copies of $M$.
\begin{definition-proposition}\cite[Theorem 2.7]{AIR}
For any $M, N \in\sttilt\Lambda$, we write $M \ge N$ if $\Fac(M)\supseteq\Fac(N)$.
Then $\ge$ gives a partial order on $\sttilt\Lambda$.
\end{definition-proposition}
We have an associated Hasse quiver:
\begin{itemize}
\item The set of vertices is $\sttilt\Lambda$.
\item We draw an arrow from $M$ to $N$ if $M>N$ and there exists no $L\in\sttilt\Lambda$ such that $M>L>N$.
\end{itemize}
By \cite[Corollary 2.34]{AIR}, two support $\tau$-tilting modules are connected by an arrow in the Hasse quiver $\hasse(\Lambda)$ if and only if they are mutations of each other.
In particular, the Hasse quiver $\hasse (\Lambda)$ is a $|\Lambda|$-regular graph.
\section{Classification of $\tau$-tilting modules over Nakayama algebras}
In this section, for Nakayama algebras, 
we study a connection between (1) $\tau$-tilting modules, (2) proper support $\tau$-tilting modules, 
(3) triangulations of a regular polygon with a puncture, and (4) certain integer sequences.
In Subsection \ref{summary}, we give an example of these correspondences.

Recall the definition and basic properties of Nakayama algebras.
A module $M$ is said to be {\it uniserial} if it has a unique composition series.
A finite dimensional algebra is said to be {\it Nakayama} if every indecomposable projective module
and every indecomposable injective module are uniserial.
The following quivers will play a central role in this paper.
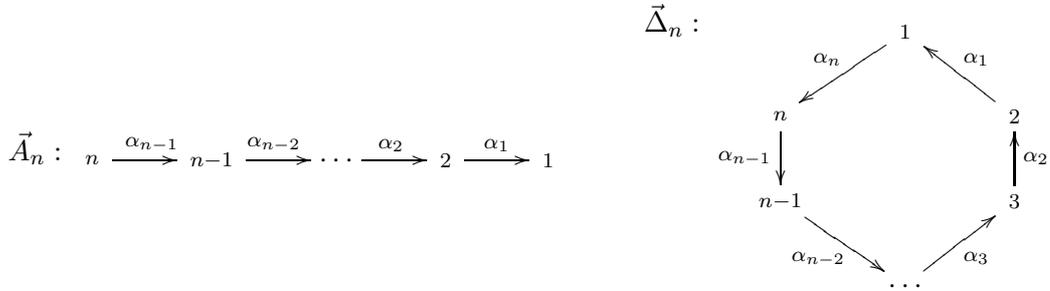
\begin{align}
\begin{minipage}{0.6\hsize}
\ \ \ \ \ $\vec{A}_{n}:$ \xymatrix{\begin{smallmatrix}n\end{smallmatrix}\ar[r]^{\alpha_{n-1}}&\begin{smallmatrix}n-1\end{smallmatrix}\ar[r]^{\alpha_{n-2}}
&\cdots\ar[r]^{\alpha_{2}}&\begin{smallmatrix}2\end{smallmatrix}\ar[r]^{\alpha_{1}}&\begin{smallmatrix}1\end{smallmatrix}}\notag
\end{minipage}
\begin{minipage}{0.4\hsize}
$\vec{\Delta}_{n}:$ \xymatrix @R=7mm{
&\begin{smallmatrix}1\end{smallmatrix}\ar[ld]_{\alpha_{n}}&\\
\begin{smallmatrix}n\end{smallmatrix}\ar[d]_{\alpha_{n-1}}&&\begin{smallmatrix}2\end{smallmatrix}\ar[lu]_{\alpha_{1}}\\
\begin{smallmatrix}n-1\end{smallmatrix}\ar[rd]_{\alpha_{n-2}}&&\begin{smallmatrix}3\end{smallmatrix}\ar[u]_{\alpha_{2}}\\
&\cdots \ar[ur]_{\alpha_{3}}&
}\notag
\end{minipage}
\end{align}
\begin{proposition}\cite[V.3.2]{ASS}\label{ass3.2}
A basic connected algebra is Nakayama if and only if 
its quiver is either $\vec{A}_{n}$ or $\vec{\Delta}_{n}$.
\end{proposition}
Throughout this section, we assume that $\Lambda$ is a basic connected Nakayama algebra with $n$ simple modules.
We give a concrete description of indecomposable modules over Nakayama algebras.
We denote by $\ell(M)$ the Loewy length of $M\in\mod\Lambda$.
\begin{proposition}\cite[V.3.5, V.4.1 and V.4.2]{ASS}\label{ass3.5}
For any indecomposable $\Lambda$-module $M$, there exists $i\in [1,n]$ 
and $t\in [1,\ell(P_{i})]$ such that $M\simeq P_{i}/\rad^{t} P_{i}$ and $t=\ell(M)$.
Moreover, if $M$ is not projective, then we have $\tau M\simeq \rad P_{i}/\rad^{t+1}P_{i}$ and $\ell(\tau M)=\ell(M)$.
\end{proposition}
We let $\Lambda_{n}^{r}:=K\vec{\Delta}_{n}/J^{r}$, where $J$ is the arrow ideal of $K\vec{\Delta}_{n}$.
The Auslander-Reiten quiver of $\Lambda_{n}^{r}$ can be drawn easily \cite[V.4.1]{ASS}.
For example, the Auslander-Reiten quiver of $\Lambda^{5}_{4}$ is given by the following:
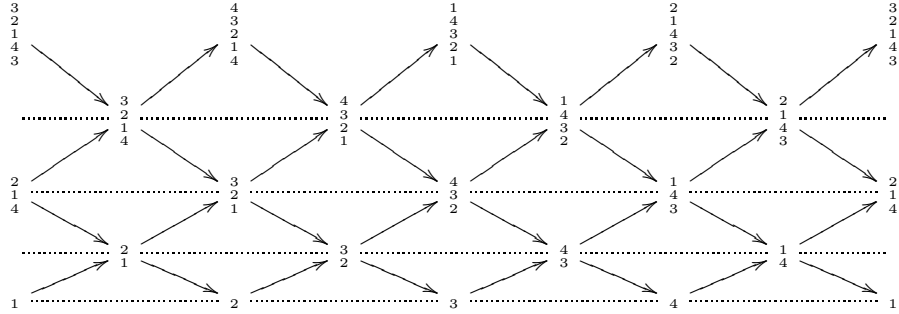
\begin{align}
\tiny{\xymatrix @R=2mm @C=10mm{
{\begin{smallmatrix}3\\2\\1\\4\\3\end{smallmatrix}}\ar[rd]&&
{\begin{smallmatrix}4\\3\\2\\1\\4\end{smallmatrix}}\ar[rd]&&
{\begin{smallmatrix}1\\4\\3\\2\\1\end{smallmatrix}}\ar[rd]&&
{\begin{smallmatrix}2\\1\\4\\3\\2\end{smallmatrix}}\ar[rd]&&
{\begin{smallmatrix}3\\2\\1\\4\\3\end{smallmatrix}}\\
&{\begin{smallmatrix}3\\2\\1\\4\end{smallmatrix}}\ar[ru]\ar[rd]\ar@{.}[l]&&
{\begin{smallmatrix}4\\3\\2\\1\end{smallmatrix}}\ar[ru]\ar[rd]\ar@{.}[ll]&&
{\begin{smallmatrix}1\\4\\3\\2\end{smallmatrix}}\ar[ru]\ar[rd]\ar@{.}[ll]&&
{\begin{smallmatrix}2\\1\\4\\3\end{smallmatrix}}\ar[ru]\ar[rd]\ar@{.}[ll]&\ar@{.}[l]\\
{\begin{smallmatrix}2\\1\\4\end{smallmatrix}}\ar[ru]\ar[rd]&&
{\begin{smallmatrix}3\\2\\1\end{smallmatrix}}\ar[ru]\ar[rd]\ar@{.}[ll]&&
{\begin{smallmatrix}4\\3\\2\end{smallmatrix}}\ar[ru]\ar[rd]\ar@{.}[ll]&&
{\begin{smallmatrix}1\\4\\3\end{smallmatrix}}\ar[ru]\ar[rd]\ar@{.}[ll]&&
{\begin{smallmatrix}2\\1\\4\end{smallmatrix}}\ar@{.}[ll]\\
&{\begin{smallmatrix}2\\1\end{smallmatrix}}\ar[ru]\ar[rd]\ar@{.}[l]&&
{\begin{smallmatrix}3\\2\end{smallmatrix}}\ar[ru]\ar[rd]\ar@{.}[ll]&&
{\begin{smallmatrix}4\\3\end{smallmatrix}}\ar[ru]\ar[rd]\ar@{.}[ll]&&
{\begin{smallmatrix}1\\4\end{smallmatrix}}\ar[ru]\ar[rd]\ar@{.}[ll]&\ar@{.}[l]\\
{\begin{smallmatrix}1\end{smallmatrix}}\ar[ru]&&\textnormal{$\begin{smallmatrix}2\end{smallmatrix}$}\ar[ru]\ar@{.}[ll]&&
{\begin{smallmatrix}3\end{smallmatrix}}\ar[ru]\ar@{.}[ll]&&\textnormal{$\begin{smallmatrix}4\end{smallmatrix}$}\ar[ru]\ar@{.}[ll]&&
{\begin{smallmatrix}1\end{smallmatrix}}\ar@{.}[ll]
}}\notag
\end{align}

We use the following notation in this paper.
\begin{notation}\label{nota}
Fix an integer $n>0$.
For any integer $i$, there exist integers $j$ and $1\le k\le n$ such that $i=nj+k$.
Then we let $(i)_{n}:=k$.
For integers $i<j$, we let 
\begin{align}
[i,j]_{n} :=\{ (i)_{n},(i+1)_{n},\ldots,(j-1)_{n},(j)_{n} \}. \notag
\end{align}
\end{notation}
For example, when $n=4$, we have $[ 3,5 ]_{4}=\{ 3,4,1\}$.

By Proposition \ref{ass3.5}, each indecomposable $\Lambda$-module $M$ is uniquely determined, up to isomorphism, 
by its simple top $S_{j}$ and the Loewy length $l:=\ell(M)$.
In this case, $M$ has a unique composition series with the associated composition factors
\begin{align}
S_{(j)_{n}},S_{(j-1)_{n}},\ldots,S_{(j-l+1)_{n}}. \notag
\end{align}
Thus we can easily characterize when nonzero homomorphisms exist between indecomposable $\Lambda$-modules.
\begin{lemma}\label{2-1}
Let $M=P_{j}/\rad^{l}P_{j}$ and $N=P_{i}/\rad^{k}P_{i}$ for $i,j,k,l\in [1,n]$.
The following conditions are equivalent:
\begin{itemize}
\item[(1)] $\Hom_{\Lambda}(M,N)\neq 0$.
\item[(2)] $j\in [i-k+1,i]_{n}$ and $(i-k+1)_{n}\in [j-l+1,j]_{n}$.
\end{itemize}
Moreover, if $l\ge k$, then the following condition is also equivalent:
\begin{itemize}
\item[(3)] $\Hom_{\Lambda}(P_{j},N)\neq 0$.
\end{itemize}
\end{lemma}
\begin{proof}
(1) $\Rightarrow$ (2): 
If $\Hom_{\Lambda}(M,N)\neq 0$, then $M$ has $\soc N=S_{(i-k+1)_{n}}$ as a composition factor 
and $N$ has $\top M=S_{j}$ as a composition factor by Lemma \ref{1-1}(1).
Hence the assertion follows.

(2) $\Rightarrow$ (1): 
By our assumption, there exists an indecomposable $\Lambda$-module $L$ such that $\top L=S_{j}$, $\soc L=S_{(i-k+1)_{n}}$ and $\ell(L)\le n$.
Then $L$ is a factor module of $M$ and a submodule of $N$.
Thus we have $\Hom_{\Lambda}(M,N)\neq 0$.

(1) $\Rightarrow$ (3):
This is clear since $M$ is a factor module of $P_{j}$.

(3) $\Rightarrow$ (1):
Take a nonzero morphism $f: P_{j}\xto{}N$.
Since $l \ge k$, we have $f(\rad^{l}P_{j})\subset \rad^{l}N=0$.
Thus $f$ induces a nonzero morphism $M=P_{j}/\rad^{l}P_{j}\xto{}N$.
\end{proof}
We give a criterion for indecomposable modules to be $\tau$-rigid.
\begin{proposition}\label{2-2}
Let $M$ be an indecomposable nonprojective $\Lambda$-module.
Then $M$ is $\tau$-rigid if and only if $\ell(M)<n$.
\end{proposition}
\begin{proof}
By Proposition \ref{ass3.5}, 
we can assume that $M=P_{j}/\rad^{l}P_{j}$ and $\tau M=P_{j-1}/\rad^{l}P_{j-1}$.
Then we have
\begin{align}
\Hom_{\Lambda}(M,\tau M)\neq 0
&\overset{\ref{2-1}}{\Longleftrightarrow} 
\begin{cases}
\ j\in [j-l,j-1]_{n}\\
\ (j-l)_{n}\in [j-l+1, j]_{n}
\end{cases}\notag\\
&\Longleftrightarrow \ell(M)\ge n. \notag\qedhere
\end{align}
\end{proof}
In view of \cite[Theorem 3.2]{AIR}, this was independently obtained by Antipov-Zvonareva \cite[Proposition 5.3]{AZ}.
\subsection{$\tau$-tilting modules and proper support $\tau$-tilting modules}
In this subsection, we study a connection between $\tau$-tilting $\Lambda$-modules and proper support $\tau$-tilting $\Lambda$-modules.

Let $\{ e_{1},e_{2},\ldots,e_{n}\}$ be a complete set of primitive orthogonal idempotents of $\Lambda$
and $\E_{\Lambda}:=\{ \sum_{j\in J}e_{j}\ |\ J\subset [1,n] \}$.
We define the bijection
\begin{align}
\phi: \E_{\Lambda}\longrightarrow \E_{\Lambda} \notag
\end{align}
given by $\phi(\sum_{i\in I}e_{i})=\sum_{i\in I}e_{i-1}$, where $e_{0}:=e_{n}$.

We denote by $\mod_{\np} \Lambda$ the full subcategory of $\mod\Lambda$ 
consisting of $\Lambda$-modules which do not have nonzero projective direct summands,
and let $\psttilt_{\np}\Lambda:=\psttilt\Lambda\cap \mod_{\np}\Lambda$.
We decompose $M\in\mod\Lambda$ as $M=M_{\np}\oplus M_{\pr}$, 
where $M_{\np}\in \mod_{\np}\Lambda$ and $M_{\pr}$ is a maximal projective direct summand of $M$. 

We state our main theorem in this subsection, where $e_{M}$ is the idempotent in Definition \ref{air0.1}.
\begin{theorem}\label{2-11}
Let $\Lambda$ be a Nakayama algebra.
Then the following hold.
\begin{itemize}
\item[(1)] There are mutually inverse bijections
\begin{align}
\ttilt\Lambda\longleftrightarrow \psttilt_{\np}\Lambda \notag
\end{align}
given by $\ttilt\Lambda \ni M \mapsto M_{\np}$ and $\psttilt_{\np}\Lambda\ni N\mapsto N\oplus \phi(e_{N})\Lambda$.
\item[(2)] If $\ell(P_{i})\ge n$ for all $i\in [1,n]$, then $\psttilt_{\np}\Lambda=\psttilt\Lambda$.
In particular, we have a bijection
\begin{align}
\ttilt\Lambda\longleftrightarrow \psttilt\Lambda. \notag
\end{align}
\end{itemize}
\end{theorem}
In the rest of this subsection, we will give a proof of Theorem \ref{2-11}.
First, we show that the map $\psttilt_{\np}\Lambda\ni N\mapsto N\oplus \phi(e_{N})\Lambda \in\ttilt\Lambda$ is well-defined.
\begin{proposition}\label{2-12}
If $N \in \psttilt_{\np}\Lambda$, then we have $N\oplus \phi(e_{N})\Lambda\in \ttilt\Lambda$.
\end{proposition}
\begin{proof}
Since $N$ is annihilated by $e_{N}$, $N$ does not have $\top e_{N}\Lambda$ as a composition factor.
By Proposition \ref{ass3.5}, $\tau N$ does not have any direct summand of $\top (\phi(e_{N})\Lambda)$ as a composition factor.
Hence $N\oplus \phi(e_{N})\Lambda$ is a $\tau$-rigid $\Lambda$-module.
Moreover, by $N\in \mod_{\np}\Lambda$, we have 
\begin{align}
|N\oplus \phi(e_{N})\Lambda|=|N|+|\phi(e_{N})\Lambda|=|N|+|e_{N}\Lambda|=|\Lambda|. \notag
\end{align}
Thus $N\oplus \phi(e_{N})\Lambda$ is a $\tau$-tilting $\Lambda$-module.
\end{proof}
Conversely, for a given $\tau$-tilting $\Lambda$-module, we give a construction of a certain proper support $\tau$-tilting $\Lambda$-module.
\begin{proposition}\label{2-13}
The following hold.
\begin{itemize}
\item[(1)] Each $\tau$-tilting $\Lambda$-module has a nonzero projective $\Lambda$-module as a direct summand.
\item[(2)] If $M\in \ttilt\Lambda$, then we have $M_{\np}\in \psttilt_{\np}\Lambda$ and $M= M_{\np}\oplus \phi(e_{M_{\np}})\Lambda$.
\end{itemize}
\end{proposition}
\begin{proof}
(i) We claim that, if $M$ has a nonzero projective direct summand, then we have $M_{\np}\in \psttilt_{\np}\Lambda$ and $M= M_{\np}\oplus \phi(e_{M_{\np}})\Lambda$.
Let $M=M_{\np}\oplus M_{\pr}\in \ttilt\Lambda$.
We may assume that $M_{\pr}=e\Lambda$, where $e\in\Lambda$ is a nonzero idempotent.
Since $M$ is $\tau$-rigid, $\tau M_{\np}$ does not have $\top e\Lambda$ as a composition factor by Lemma \ref{1-1}(1). 
Thus $M_{\np}$ does not have $\top (\phi^{-1}(e)\Lambda)$ as a composition factor by Proposition \ref{ass3.5}.
Hence $M_{\np}$ is a $\tau$-rigid $(\Lambda/\langle \phi^{-1}(e)\rangle)$-module.
Moreover, we have 
\begin{align}
|M_{\np}|=|M|-|e\Lambda|=|\Lambda|-|\phi^{-1}(e)\Lambda|. \notag
\end{align}
Thus $M_{\np}$ is a $\tau$-tilting $(\Lambda/\langle \phi^{-1}(e)\rangle)$-module.
Hence, $M_{\np}\in\psttilt_{\np}\Lambda$ holds, and moreover we have $\phi^{-1}(e)=e_{M_{\np}}$.
Thus $M_{\pr}=e\Lambda=\phi(e_{M_{\np}})\Lambda$.

(ii) We show that each $\tau$-tilting $\Lambda$-module has a nonzero projective $\Lambda$-module as a direct summand.
Assume that $M\in\ttilt\Lambda$ does not have a nonzero projective $\Lambda$-module as a direct summand.
Let $L$ be an indecomposable direct summand of $M$ with the maximal Loewy length.
We write $M=L\oplus N$. 
Then $(N,0)$ is a basic almost complete support $\tau$-tilting pair for $\Lambda$.
Since $M$ is $\tau$-rigid, $\Hom_{\Lambda}(L,\tau N)$ vanishes.
Since the Loewy length of $L$ is at least the Loewy length of any indecomposable direct summand of $N$ by the choice of $L$,
we have $\Hom_{\Lambda}(P_{L},\tau N)=0$ by Lemma \ref{2-1}(3)$\Rightarrow$(1), where $P_{L}$ is a projective cover of $L$.
Thus $P_{L}\oplus N$ is a $\tau$-rigid $\Lambda$-module.
Moreover, since $P_{L}$ is indecomposable, we have 
\begin{align}
|P_{L}\oplus N|=|P_{L}|+|N|=|\Lambda|.\notag
\end{align}
Hence, $P_{L}\oplus N$ is a $\tau$-tilting $\Lambda$-module.
Moreover, $N$ is a support $\tau$-tilting $\Lambda$-module by (i).
This means that basic almost complete support $\tau$-tilting pair $(N,0)$ has
three basic support $\tau$-tilting $\Lambda$-modules $N, L\oplus N$ and $P_{L}\oplus N$.
However, this contradicts Proposition \ref{air2.18}.
Hence the statement (1) follows.

(iii) The statement (2) follows from (i) and (ii).
\end{proof}

Now we are ready to prove Theorem \ref{2-11}.
\begin{proof}[Proof of Theorem \ref{2-11}]
(1) It follows from Proposition \ref{2-12} and \ref{2-13}.

(2) This is clear because $\ell(P_{i})\ge n$ implies that $P_{i}$ is sincere.
\end{proof}
As an immediate consequence of Theorem \ref{2-11}, we have the following statement, 
where $\ttilt_{\np}(\Lambda/\langle e\rangle):=\ttilt(\Lambda/\langle e\rangle)\cap \mod_{\np}\Lambda$.
\begin{corollary}
Let $\Lambda$ be a Nakayama algebra.
\begin{itemize}
\item[(1)] We have 
\begin{align}
\displaystyle\sttilt\Lambda 
=\coprod_{e\in \E_{\Lambda}\setminus \{ 0\}}\left(
\ttilt(\Lambda/\langle e\rangle)\coprod \{ M\oplus \phi(e)\Lambda\ |\ M\in \ttilt_{\np}(\Lambda/\langle e \rangle)\right). \notag
\end{align}
\item[(2)] If $\ell(P_{i})\ge n$ for all $i\in  [1,n]$, we have
\begin{align}
\displaystyle\sttilt\Lambda 
=\coprod_{e\in \E_{\Lambda}\setminus \{ 0\}}
\{ M,\ M\oplus \phi(e)\Lambda\ |\ M\in \tilt(\Lambda/\langle e \rangle)\}.\notag
\end{align}
\end{itemize}
\end{corollary}
\begin{proof}
By Theorem \ref{2-11}(1), we have 
\begin{align}\label{eq1}
\ttilt\Lambda
=\coprod_{e\in\E_{\Lambda}\setminus \{ 0\}}\{ M\oplus \phi(e)\Lambda \mid M\in \ttilt_{\np}(\Lambda/\langle e \rangle) \}. 
\end{align}

(1) The assertion follows from (\ref{eq1}), Proposition \ref{1-2}(1) and (2).

(2) By Theorem \ref{2-11}(2), we can omit ``np'' in (\ref{eq1}), and hence we have 
\begin{align}\label{eq2}
\displaystyle\sttilt\Lambda 
=\coprod_{e\in \E_{\Lambda}\setminus \{ 0\}}
\{ M,\ M\oplus \phi(e)\Lambda\ |\ M\in \ttilt(\Lambda/\langle e \rangle)\}.
\end{align}
By Proposition \ref{1-2}(3), we can replace $\ttilt(\Lambda/\langle e\rangle)$ in (\ref{eq2}) by $\tilt(\Lambda/\langle e \rangle)$.
Thus the assertion follows.
\end{proof}
Finally, we give an example.
\begin{example}
Let $\Lambda:=\Lambda_{3}^{3}=K\vec{\Delta}_{3}/J^{3}$ (see Proposition \ref{ass3.2}).
To obtain $\tau$-tilting $\Lambda$-modules, let us calculate $\Lambda/\langle e \rangle$ for any idempotent $e\in \E_{\Lambda}$.
We have $\Lambda/\langle e_{i}\rangle \simeq K\vec{A}_{2}$, 
$\Lambda/\langle e_{i}+e_{i+1}\rangle \simeq K\vec{A}_{1}$, and 
$\Lambda/\langle e_{1}+e_{2}+e_{3}\rangle = \{ 0\}$ for $i\in [1,3]$.
Thus we have
\begin{align}
&\ttilt(\Lambda/\langle e_{i}\rangle)
=\tilt(K\vec{A}_{2})=\{ \begin{smallmatrix}
i+2\\
i+1
\end{smallmatrix}\oplus\begin{smallmatrix}
i+2\\
\end{smallmatrix}, 
\begin{smallmatrix}
i+2\\
i+1
\end{smallmatrix}\oplus\begin{smallmatrix}
i+1\\
\end{smallmatrix}\}\notag\\
&\ttilt(\Lambda/\langle e_{i}+e_{i+1}\rangle)
=\tilt(K\vec{A}_{1})=\{
\begin{smallmatrix}
i+2\\
\end{smallmatrix}
\}\notag\\
&\ttilt(\Lambda/\langle e_{1}+e_{2}+e_{3}\rangle)=\{ 0 \}. \notag
\end{align}
Applying Theorem \ref{2-11}, we have
\begin{align}
\sttilt\Lambda
&=\{ 
\begin{smallmatrix}
\{0\}
\end{smallmatrix},
\begin{smallmatrix}
1
\end{smallmatrix},
\begin{smallmatrix}
2
\end{smallmatrix},
\begin{smallmatrix}
3
\end{smallmatrix}, 
\begin{smallmatrix}
1\\
3
\end{smallmatrix}\oplus\begin{smallmatrix}
1\\
\end{smallmatrix}, 
\begin{smallmatrix}
1\\
3
\end{smallmatrix}\oplus\begin{smallmatrix}
3\\
\end{smallmatrix},
\begin{smallmatrix}
2\\
1
\end{smallmatrix}
\oplus
\begin{smallmatrix}
2\\
\end{smallmatrix}, 
\begin{smallmatrix}
2\\
1
\end{smallmatrix}\oplus\begin{smallmatrix}
1\\
\end{smallmatrix},
\begin{smallmatrix}
3\\
2
\end{smallmatrix}
\oplus
\begin{smallmatrix}
3\\
\end{smallmatrix}, 
\begin{smallmatrix}
3\\
2
\end{smallmatrix}\oplus\begin{smallmatrix}
2\\
\end{smallmatrix}
\}\notag\\
&\coprod\{
\begin{smallmatrix}
1\\
3\\
2
\end{smallmatrix}\oplus\begin{smallmatrix}
2\\
1\\
3
\end{smallmatrix}\oplus\begin{smallmatrix}
3\\
2\\
1
\end{smallmatrix},
\begin{smallmatrix}
1
\end{smallmatrix}\oplus\begin{smallmatrix}
1\\
3\\
2
\end{smallmatrix}\oplus\begin{smallmatrix}
2\\
1\\
3
\end{smallmatrix},
\begin{smallmatrix}
2
\end{smallmatrix}\oplus\begin{smallmatrix}
2\\
1\\
3
\end{smallmatrix}\oplus\begin{smallmatrix}
3\\
2\\
1
\end{smallmatrix},
\begin{smallmatrix}
3
\end{smallmatrix}\oplus\begin{smallmatrix}
3\\
2\\
1
\end{smallmatrix}\oplus\begin{smallmatrix}
1\\
3\\
2
\end{smallmatrix},\notag\\
&\ \ \ \ \ \ 
\begin{smallmatrix}
1\\
3
\end{smallmatrix}\oplus\begin{smallmatrix}
1
\end{smallmatrix}\oplus\begin{smallmatrix}
1\\
3\\
2
\end{smallmatrix},
\begin{smallmatrix}
1\\
3
\end{smallmatrix}\oplus\begin{smallmatrix}
3
\end{smallmatrix}\oplus\begin{smallmatrix}
1\\
3\\
2
\end{smallmatrix},
\begin{smallmatrix}
2\\
1
\end{smallmatrix}\oplus\begin{smallmatrix}
2
\end{smallmatrix}\oplus\begin{smallmatrix}
2\\
1\\
3
\end{smallmatrix},
\begin{smallmatrix}
2\\
1
\end{smallmatrix}\oplus\begin{smallmatrix}
1
\end{smallmatrix}\oplus\begin{smallmatrix}
2\\
1\\
3
\end{smallmatrix},
\begin{smallmatrix}
3\\
2
\end{smallmatrix}\oplus\begin{smallmatrix}
3
\end{smallmatrix}\oplus\begin{smallmatrix}
3\\
2\\
1
\end{smallmatrix},
\begin{smallmatrix}
3\\
2
\end{smallmatrix}\oplus\begin{smallmatrix}
2
\end{smallmatrix}\oplus\begin{smallmatrix}
3\\
2\\
1
\end{smallmatrix}
\}.\notag
\end{align}
Moreover, the Hasse quiver $\hasse(\Lambda)$ is the following:
\begin{align}
\xymatrix @R=2.5mm @C=0.5mm{
 & & & & &\textnormal{$\tiny\begin{smallmatrix}
1\\
3\\
2
\end{smallmatrix}\begin{smallmatrix}
2\\
1\\
3
\end{smallmatrix}\begin{smallmatrix}
3\\
2\\
1
\end{smallmatrix}$}\ar[dllll]\ar[d]\ar[drrrr]& & & & & \\
 &\textnormal{\tiny$\begin{smallmatrix}
2
\end{smallmatrix}\begin{smallmatrix}
2\\
1\\
3
\end{smallmatrix}\begin{smallmatrix}
3\\
2\\
1
\end{smallmatrix}$}\ar[dl]\ar[dr]& & & &\textnormal{\tiny$\begin{smallmatrix}
1\\
3\\
2
\end{smallmatrix}\begin{smallmatrix}
2\\
1\\
3
\end{smallmatrix}\begin{smallmatrix}
1
\end{smallmatrix}$}\ar[dl]\ar[dr]& & & &\textnormal{\tiny$\begin{smallmatrix}
1\\
3\\
2
\end{smallmatrix}\begin{smallmatrix}
3
\end{smallmatrix}\begin{smallmatrix}
3\\
2\\
1
\end{smallmatrix}$}\ar[dl]\ar[dr]& \\
\textnormal{\tiny$\begin{smallmatrix}
2
\end{smallmatrix}\begin{smallmatrix}
3\\
2
\end{smallmatrix}\begin{smallmatrix}
3\\
2\\
1
\end{smallmatrix}$}\ar[dd]\ar@/^6mm/[rrrrrrrrrr]& &\textnormal{\tiny$\begin{smallmatrix}
2
\end{smallmatrix}\begin{smallmatrix}
2\\
1\\
3
\end{smallmatrix}\begin{smallmatrix}
2\\
1
\end{smallmatrix}$}\ar[dd]& &\textnormal{\tiny$\begin{smallmatrix}
2\\
1
\end{smallmatrix}\begin{smallmatrix}
2\\
1\\
3
\end{smallmatrix}\begin{smallmatrix}
1
\end{smallmatrix}$}\ar[dd]\ar[ll]& &\textnormal{\tiny$\begin{smallmatrix}
1\\
3\\
2
\end{smallmatrix}\begin{smallmatrix}
1\\
3
\end{smallmatrix}\begin{smallmatrix}
1
\end{smallmatrix}$}\ar[dd]& &\textnormal{\tiny$\begin{smallmatrix}
1\\
3\\
2
\end{smallmatrix}\begin{smallmatrix}
3
\end{smallmatrix}\begin{smallmatrix}
1\\
3
\end{smallmatrix}$}\ar[dd]\ar[ll]& &\textnormal{\tiny$\begin{smallmatrix}
3\\
2
\end{smallmatrix}\begin{smallmatrix}
3
\end{smallmatrix}\begin{smallmatrix}
3\\
2\\
1
\end{smallmatrix}$}\ar[dd] \\
 & & & & & & & & & & \\
\textnormal{\tiny$\begin{smallmatrix}
2
\end{smallmatrix}\begin{smallmatrix}
3\\
2
\end{smallmatrix}$}\ar[dr]\ar@/^6mm/[rrrrrrrrrr]& &\textnormal{\tiny$\begin{smallmatrix}
2
\end{smallmatrix}\begin{smallmatrix}
2\\
1
\end{smallmatrix}$}\ar[dl]& &\textnormal{\tiny$\begin{smallmatrix}
2\\
1
\end{smallmatrix}\begin{smallmatrix}
1
\end{smallmatrix}$}\ar[ll]\ar[dr]& &\textnormal{\tiny$\begin{smallmatrix}
1\\
3
\end{smallmatrix}\begin{smallmatrix}
1
\end{smallmatrix}$}\ar[dl]& &\textnormal{\tiny$\begin{smallmatrix}
3
\end{smallmatrix}\begin{smallmatrix}
1\\
3
\end{smallmatrix}$}\ar[ll]\ar[dr]& &\textnormal{\tiny$\begin{smallmatrix}
3\\
2
\end{smallmatrix}\begin{smallmatrix}
3
\end{smallmatrix}$}\ar[dl] \\
 &\textnormal{\tiny$\begin{smallmatrix}
2
\end{smallmatrix}$}\ar[drrrr]& & & &\textnormal{\tiny$\begin{smallmatrix}
1
\end{smallmatrix}$}\ar[d]& & & &\textnormal{\tiny$\begin{smallmatrix}
3
\end{smallmatrix}$}\ar[dllll]& \\
 & & & & &\textnormal{\tiny$\begin{smallmatrix}
\{ 0\}
\end{smallmatrix}$}& & & & & \\
}\notag
\end{align}
\end{example}
\subsection{$\tau$-tilting modules and triangulations}
In this subsection, we give a connection between $\tau$-tilting $\Lambda$-modules and triangulations of an $n$-regular polygon with a puncture.
Recall the definition and basic results of triangulations. 
We refer to \cite{FST} about more details of triangulations of surfaces. 
Let $\G_{n}$ be an $n$-regular polygon with a puncture.
We label the points of $\G_{n}$ counterclockwise around the boundary by $1,2,\ldots,n$.
\begin{definition}
Let $i,j\in [1,n]$.
\begin{itemize}
\item[(1)] An {\it inner arc} $\langle i,j \rangle$ in $\G_{n}$ is a path from the point $i$ to the point $j$
homotopic to the boundary path $(i)_{n},(i+1)_{n},\ldots,(i+t)_{n}=j$, where $t\in [ 2,n ]$.
Then we call $i$ an {\it initial} point, $j$ a {\it terminal} point, and $\ell(\langle i,j\rangle):=t$ the {\it length} of the inner arc.
By definition, $2 \le \ell(\langle i,j\rangle) \le n$ holds for any inner arc in $\G_{n}$.
\item[(2)] A {\it projective arc} $\langle \bullet,j\rangle$ in $\G_{n}$ is a path from the puncture to the point $j$.
Then we call $j$ a {\it terminal} point.
\item[(3)] An {\it admissible arc} is an inner arc or a projective arc.
\end{itemize}
\end{definition}

We denote by $\Arc(n)$ (respectively, $\Arc_{\inn}(n)$, $\Arc_{\pr}(n)$) the set of all admissible (respectively, inner, projective) arcs of $\G_{n}$.
Thus we have  
\begin{align}
\Arc(n)= \Arc_{\inn}(n)\coprod \Arc_{\pr}(n). \notag
\end{align}

Note that, if $i\neq j$, $\langle i,j\rangle$ and $\langle j,i\rangle$ are different arcs as the picture in Figure \ref{fig1} shows.

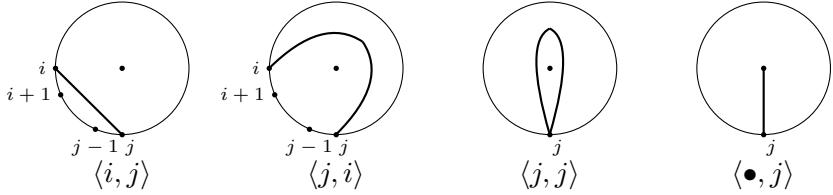
\begin{figure}[htbp]
\begin{picture}(300,65)(0,0)
\put(30,40){\circle{50}}
\put(30,40){\circle*{2}}
\put(30,15){\circle*{2}}
\put(5,40){\circle*{2}}
\put(20,17){\circle*{2}}
\put(7,30){\circle*{2}}
\put(33,10){\makebox(0,0){\tiny$j$}}
\put(20,10){\makebox(0,0){\tiny$j-1$}}
\put(-5,30){\makebox(0,0){\tiny$i+1$}}
\put(0,40){\makebox(0,0){\tiny$i$}}
\put(30,0){\makebox(0,0){$\langle i,j\rangle$}}
\put(110,40){\circle{50}}
\put(110,40){\circle*{2}}
\put(110,15){\circle*{2}}
\put(85,40){\circle*{2}}
\put(100,17){\circle*{2}}
\put(87,30){\circle*{2}}
\put(113,10){\makebox(0,0){\tiny$j$}}
\put(100,10){\makebox(0,0){\tiny$j-1$}}
\put(75,30){\makebox(0,0){\tiny$i+1$}}
\put(80,40){\makebox(0,0){\tiny$i$}}
\put(110,0){\makebox(0,0){$\langle j,i\rangle$}}
\put(190,40){\circle{50}}
\put(190,40){\circle*{2}}
\put(190,15){\circle*{2}}
\put(193,10){\makebox(0,0){\tiny$j$}}
\put(190,0){\makebox(0,0){$\langle j,j\rangle$}}
\put(270,40){\circle{50}}
\put(270,40){\circle*{2}}
\put(270,15){\circle*{2}}
\put(273,10){\makebox(0,0){\tiny$j$}}
\put(270,0){\makebox(0,0){$\langle \bullet,j\rangle$}}
\thicklines
\put(30,15){\line(-1,1){25}}
\qbezier(85,40)(105,60)(120,50)
\qbezier(120,50)(130,35)(110,15)
\qbezier(190,15)(180,50)(190,55)
\qbezier(190,55)(200,50)(190,15)
\put(270,40){\line(0,-1){25}}
\end{picture}
\caption{Admissible arcs in a polygon with a puncture} 
\label{fig1}
\end{figure}
\begin{definition}
\begin{itemize}
\item[(1)] Two admissible arcs in $\G_{n}$ are called {\it compatible} if they do not intersect in $\G_{n}$ 
(except possibly at their initial and terminal points).
\item[(2)] A {\it triangulation} of $\G_{n}$ is a maximal set of distinct pairwise compatible admissible arcs.
We denote by $\T(n)$ the set of triangulations of $\G_{n}$.
\item[(3)] For integers $l_{1},l_{2},\ldots,l_{n}\ge 1$,
we denote by $\T(n;l_{1},l_{2},\ldots,l_{n})$ the subset of $\T(n)$ consisting of triangulations
such that the length of every inner arc with the terminal point $j$ is at most $l_{j}$ for any $j\in [1,n]$.
\end{itemize}
\end{definition}

For example, the set of all projective arcs gives a triangulation of $\G_{n}$.

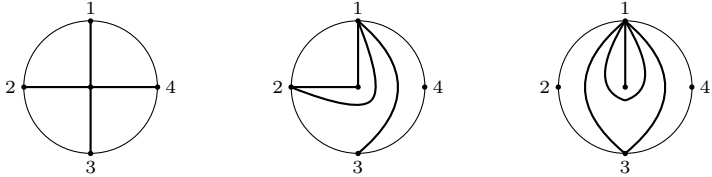
\begin{figure}[hbtp] 
\begin{picture}(300,55)(0,0)
\put(50,25){\circle{50}}
\put(50,25){\circle*{2}}
\put(50,50){\circle*{2}}
\put(50,55){\makebox(0,0){\tiny 1}}
\put(75,25){\circle*{2}}
\put(80,25){\makebox(0,0){\tiny 4}}
\put(50,0){\circle*{2}}
\put(50,-5){\makebox(0,0){\tiny 3}}
\put(25,25){\circle*{2}}
\put(20,25){\makebox(0,0){\tiny 2}}
\put(150,25){\circle{50}}
\put(150,25){\circle*{2}}
\put(150,50){\circle*{2}}
\put(150,55){\makebox(0,0){\tiny 1}}
\put(175,25){\circle*{2}}
\put(180,25){\makebox(0,0){\tiny 4}}
\put(150,0){\circle*{2}}
\put(150,-5){\makebox(0,0){\tiny 3}}
\put(125,25){\circle*{2}}
\put(120,25){\makebox(0,0){\tiny 2}}
\put(250,25){\circle{50}}
\put(250,25){\circle*{2}}
\put(250,50){\circle*{2}}
\put(250,55){\makebox(0,0){\tiny 1}}
\put(275,25){\circle*{2}}
\put(280,25){\makebox(0,0){\tiny 4}}
\put(250,0){\circle*{2}}
\put(250,-5){\makebox(0,0){\tiny 3}}
\put(225,25){\circle*{2}}
\put(220,25){\makebox(0,0){\tiny 2}}
\thicklines
\put(50,25){\line(0,1){25}}
\put(50,25){\line(1,0){25}}
\put(50,25){\line(0,-1){25}}
\put(50,25){\line(-1,0){25}}
\put(150,25){\line(0,1){25}}
\put(150,25){\line(-1,0){25}}
\qbezier(150,50)(160,25)(155,20)
\qbezier(155,20)(150,15)(125,25)
\qbezier(150,50)(180,25)(150,0)
\put(250,25){\line(0,1){25}}
\qbezier(250,50)(265,25)(250,20)
\qbezier(250,20)(235,25)(250,50)
\qbezier(250,50)(280,25)(250,0)
\qbezier(250,50)(220,25)(250,0)
\end{picture}
\caption{Triangulations of $\G_{4}$} 
\end{figure}
For a subset $X$ of $\Arc(n)$, we denote by $X_{i,j}$ the subset of $X$ consisting of all inner arcs contained in the fan 
whose boundary is $\langle \bullet, i\rangle$, $\langle \bullet, j\rangle$ and the edge connecting from $i$ to $j$ in the counterclockwise direction. 
\begin{lemma}\label{2-3}
Let $X\in \T(n)$. For $i,j\in [1,n]$, we have
\begin{align}
|X_{i,j}|\le (j-i)_{n}-1. \notag
\end{align}
Moreover, $|X_{i,j}|=(j-i)_{n}-1$ if and only if $\langle i,j \rangle\in X$.
\end{lemma}
\begin{proof}
Assume that $X$ contains an inner arc $\langle i,j\rangle$.
Then we regard $X_{i,j}$ as a triangulation of the $((j-i)_{n}+1)$-regular polygon (with no puncture) 
by identifying the inner arc $\langle i,j\rangle$ with a side of the polygon.
Thus the cardinality of $X_{i,j}$ is equal to $(j-i)_{n}-1$.
Hence the assertion follows from the maximality of triangulation in a polygon.
\end{proof}

Triangulations of $\G_{n}$ have the following properties.
Let $X\in \T(n)$.
Assume that $1\le j_{1}<j_{2}<\cdots<j_{r}\le n$ are all integers satisfying $\langle \bullet, j_{i}\rangle\in X$, 
and let $j_{i}^{+}:=j_{i+1}$ for any $i\in [1,r]$, where $j^{+}_{r}=j_{r+1}:=j_{1}+n$.
Note that, if $j_{i}^{+}-j_{i}>1$, then $X$ must contain the inner arc $\langle j_{i}, j_{i}^{+} \rangle$.
\begin{proposition}\label{2-4}
For each triangulation $X\in\T(n)$, the following hold.
\begin{itemize}
\item[(1)] $X$ consists of exactly $n$ admissible arcs.
\item[(2)] $X$ contains at least one projective arc.
\end{itemize}
\end{proposition}
This is well-known, but we give a proof here for the convenience of the reader.
\begin{proof}
Let $X\in \T(n)$. 
For a longest inner arc in $X$ with the initial point $i$ and the terminal point $j$, 
there exist projective arcs $\langle \bullet,i\rangle, \langle \bullet,j\rangle\in X$ by the maximality of triangulations.
Thus each triangulation contains a projective arc.
Moreover, by Lemma \ref{2-3}, the cardinality of $X$ is
\[
|X|=\sum_{i=1}^{r}\{1+(j_{i}^{+}-j_{i}-1)\}=j_{r+1}-j_{1}=n.\qedhere
\]
\end{proof}
We give a correspondence between indecomposable $\tau$-rigid modules and admissible arcs.
By Proposition \ref{2-2}, every indecomposable nonprojective $\tau$-rigid $\Lambda$-module $M$ is uniquely determined by 
its simple top $S_{j}$ and its simple socle $S_{k}$.
Such an indecomposable $\tau$-rigid module is denoted by $M_{k-2,j}$.
Moreover, let $M_{\bullet,j}:=P_{j}$.
We denote by $\trigid\Lambda$ the set of isomorphism classes of indecomposable $\tau$-rigid $\Lambda$-modules.
\begin{proposition}\label{2-5}
Let $\Lambda$ be a Nakayama algebra and $\ell_{j}:=\ell(P_{j})$.
The following hold.
\begin{itemize}
\item[(1)] There is a bijection
\begin{align}
\Arc_{\pr}(n)\coprod \{ \langle i,j \rangle \in \Arc_{\inn}(n) \mid i,j\in [1,n],\ \ell(\langle i,j\rangle)\le \ell(P_{j})\} 
\longrightarrow \trigid\Lambda \notag
\end{align}
given by $\langle i,j\rangle\mapsto M_{i,j}$ for $i\in [1,n]\coprod\{\bullet\}$ and $j\in [1,n]$.
\item[(2)] For any $i,k\in [1,n]\coprod\{\bullet\}$ and $j,l\in [1,n]$, 
$\langle i,j\rangle$ and $\langle k,l\rangle$ are compatible if and only if $M_{i,j}\oplus M_{k,l}$ is $\tau$-rigid. 
\end{itemize}
\end{proposition}
\begin{proof}
(1) By Proposition \ref{2-2}, an indecomposable $\Lambda$-module $M$ is $\tau$-rigid if and only if it is projective or $\ell(M)<n$ holds.
Since $\ell(M)\le \ell(P)$ holds for a projective cover $P$ of $M$,
we have the desired bijection.

(2) If $i=k\in\{ \bullet\}$, then the assertion is clear.
We may assume that $i\in [1,n]\coprod\{ \bullet\}$ and $j,k,l\in [1,n]$.
Since we have $\tau M_{k,l}=M_{k-1,l-1}$ by Proposition \ref{ass3.5}, thus, by Lemma \ref{2-1}, 
$\Hom_{\Lambda}(M_{i,j},\tau M_{k,l})\neq 0$ if and only if $j\in [k+1,l-1]_{n}$ and $(k+1)_{n}\in [j-\ell(M_{i,j})+1,j]_{n}$.
This means  that $\langle i,j \rangle$ and $\langle k,l\rangle$ are compatible.
\end{proof}
As a conclusion, we obtain the following theorem.
A similar observation was given independently in \cite[Proposition 5.4]{AZ}.
\begin{theorem}\label{2-6}
Let $\Lambda$ be a Nakayama algebra with $n$ simple modules and $\ell_{j}:=\ell(P_{j})$ for any $j\in [1,n]$.
Then the map in Proposition \ref{2-5} induces a bijection
\begin{align}
\ttilt\Lambda \longrightarrow \T(n;\ell_{1},\ell_{2},\ldots,\ell_{n}). \notag
\end{align}
\end{theorem}
\begin{proof}
$\tau$-tilting $\Lambda$-modules are precisely maximal $\tau$-rigid $\Lambda$-modules by Proposition \ref{air2.10}.
On the other hand, triangulations of $\G_{n}$ are precisely maximal sets of pairwise compatible admissible arcs with the cardinality $n$ 
by Proposition \ref{2-4}.
Thus the assertion follows from Proposition \ref{2-5}(2).
\end{proof}

In the rest of this subsection, 
we give a connection between support $\tau$-tilting $\Lambda$-modules and signed triangulations of an $n$-gon with a puncture.

\begin{definition}[See \cite{LF, BS, IN}]
\begin{itemize}
\item[(1)] A triangulation is called a {\it signed triangulation} if the puncture is assigned the signature either $+$ or $-$.
We can identify the set of signed triangulations with $\T(n;\ell_{1},\ldots,\ell_{n})\times \ZZ/2\ZZ$.
\item[(2)] A {\it pop} of a signed triangulation is the operation of changing the sign of the puncture.
By {\it flips} of signed triangulations, 
we mean the operations which are flips of (ordinary) triangulations and pops for the projective arcs in self-folded triangles.
\end{itemize}
Then we always play a flip for any admissible arc of signed triangulations.
\end{definition}

The end of this subsection, we have the following bijection.
We denote by $(M,e_{M})$ a support $\tau$-tilting $\Lambda$-module $M$.
\begin{corollary}\label{bij}
Assume that $\ell(P_{i})\ge n$ for all $i\in [1,n]$.
Then there is a bijection 
\begin{align}
\T(n;\ell_{1},\ell_{2},\ldots,\ell_{n})\times \ZZ/2\ZZ \rightarrow \sttilt\Lambda \notag
\end{align}
given by $(\langle i,j \rangle,\pm)\mapsto (M_{i,j},0)$, 
$(\langle \bullet,j\rangle,+)\mapsto (P_{j},0)$ 
and $(\langle \bullet,j\rangle,-)\mapsto (0,e_{j+1})$.
Moreover, the bijection are compatible with mutations and flips.
\end{corollary}
\begin{proof}
The assertion follows from that, by Theorem \ref{2-11}(2) and \ref{2-6}, we have the following commutative diagram: 

\begin{picture}(400,80)(0,0)
\put(100,65){$\xymatrix{
\T(n;\ell_{1},\ell_{2},\ldots,\ell_{n})\times \{ +\}\ar[r]\ar@{<->}[d]&\ttilt\Lambda\ar@{<->}[d]\\
\T(n;\ell_{1},\ell_{2},\ldots,\ell_{n})\times \{ -\}\ar[r]&\psttilt\Lambda
}$}
\put(395.4,18){\qedhere}
\end{picture}
\end{proof}
We give an example of Corollary \ref{bij}.

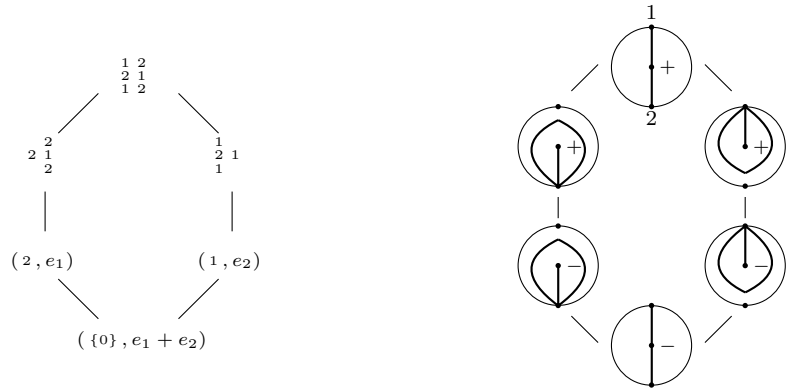
\begin{picture}(400,150)(0,0)
\put(130,110){\tiny{$\begin{smallmatrix}
1\\
2\\
1
\end{smallmatrix}\begin{smallmatrix}
2\\
1\\
2
\end{smallmatrix}$}}

\put(95,80){\tiny{$\begin{smallmatrix}
2
\end{smallmatrix}\begin{smallmatrix}
2\\
1\\
2
\end{smallmatrix}$}}

\put(165,80){\tiny{$\begin{smallmatrix}
1\\
2\\
1
\end{smallmatrix}\begin{smallmatrix}
1
\end{smallmatrix}$}}

\put(90,40){\tiny{$\left(\begin{smallmatrix}
2
\end{smallmatrix},e_{1}\right)$}}

\put(160,40){\tiny{$\left(\begin{smallmatrix}
1
\end{smallmatrix},e_{2}\right)$}}

\put(115,10){{\tiny$\left(\begin{smallmatrix}
\{ 0\}
\end{smallmatrix},e_{1}+e_{2}\right)$}}

\put(108,90){\line(1,1){15}}
\put(103,68){\line(0,-1){15}}
\put(108,35){\line(1,-1){15}}

\put(168,90){\line(-1,1){15}}
\put(173,68){\line(0,-1){15}}
\put(168,35){\line(-1,-1){15}}

\put(330,115){\circle{30}}
\put(330,115){\circle*{2}}
\put(330,130){\circle*{2}}
\put(330,100){\circle*{2}}
\put(333,113){{\tiny$+$}}

\put(326,133){$\begin{smallmatrix}1\end{smallmatrix}$}
\put(326,93){$\begin{smallmatrix}2\end{smallmatrix}$}

\put(295,85){\circle{30}}
\put(295,85){\circle*{2}}
\put(295,100){\circle*{2}}
\put(295,70){\circle*{2}}
\put(298,83){{\tiny$+$}}

\put(365,85){\circle{30}}
\put(365,85){\circle*{2}}
\put(365,100){\circle*{2}}
\put(365,70){\circle*{2}}
\put(368,83){{\tiny$+$}}

\put(295,40){\circle{30}}
\put(295,40){\circle*{2}}
\put(295,55){\circle*{2}}
\put(295,25){\circle*{2}}
\put(298,38){{\tiny$-$}}

\put(365,40){\circle{30}}
\put(365,40){\circle*{2}}
\put(365,55){\circle*{2}}
\put(365,25){\circle*{2}}
\put(368,38){{\tiny$-$}}

\put(330,10){\circle{30}}
\put(330,10){\circle*{2}}
\put(330,25){\circle*{2}}
\put(330,-5){\circle*{2}}
\put(333,8){{\tiny$-$}}

\put(300,106){\line(1,1){10}}
\put(295,66){\line(0,-1){8}}
\put(300,21){\line(1,-1){10}}

\put(360,106){\line(-1,1){10}}
\put(365,66){\line(0,-1){8}}
\put(360,21){\line(-1,-1){10}}

\thicklines
\qbezier(330,115)(330,122.5)(330,130)
\qbezier(330,115)(330,107.5)(330,100)

\qbezier(295,85)(295,77.5)(295,70)
\qbezier(295,70)(275,85)(295,95)
\qbezier(295,70)(315,85)(295,95)

\qbezier(365,85)(365,92.5)(365,100)
\qbezier(365,100)(345,85)(365,75)
\qbezier(365,100)(385,85)(365,75)

\qbezier(295,40)(295,32.5)(295,25)
\qbezier(295,25)(275,40)(295,50)
\qbezier(295,25)(315,40)(295,50)

\qbezier(365,40)(365,47.5)(365,55)
\qbezier(365,55)(345,40)(365,30)
\qbezier(365,55)(385,40)(365,30)

\qbezier(330,10)(330,17.5)(330,25)
\qbezier(330,10)(330,2.5)(330,-5)

\end{picture}
\subsection{Triangulations and integer sequences}
In this subsection, we give a simple description of $\ttilt\Lambda$ and $\T(n)$ in terms of certain nonnegative integer sequences.
Let 
\begin{align}
\Z(n):=\left\{ (a_{1},a_{2},\ldots,a_{n})\in \ZZ_{\ge 0}^{n}\ \biggl|\  \sum_{i=1}^{n}a_{i}=n \right\}. \notag
\end{align}

Let $X$ be a subset of the set $\Arc(n)$ of all admissible arcs in $\G_{n}$.
In view of Proposition \ref{2-5}, we define 
\begin{align}
\top(X)=(a_{1},a_{2},\ldots,a_{n}), \notag
\end{align}
where $a_{p}$ is the number of admissible arcs in $X$ with a terminal point $p$.
If $X\in\T(n)$, then $\top(X)\in\Z(n)$ since all triangulations of $\G_{n}$ contain exactly $n$ admissible arcs by Proposition \ref{2-4}.

\begin{theorem}\label{2-8}
The map $X\mapsto \top(X)$ gives a bijection 
\begin{align}
\top:\T(n) \longrightarrow \Z(n). \notag
\end{align}
\end{theorem}

We illustrate Theorem \ref{2-8} with an example.
\begin{example}
Let $n=3$.

\begin{picture}(400,200)(0,55)
\put(40,215){\circle{50}}
\put(40,215){\circle*{2}}
\put(40,240){\circle*{2}}
\put(40,245){\makebox(0,0){\tiny 1}}
\put(18.5,202.5){\circle*{2}}
\put(13.5,202.5){\makebox(0,0){\tiny 2}}
\put(61.5,202.5){\circle*{2}}
\put(66.5,202.5){\makebox(0,0){\tiny 3}}

\thicklines
\qbezier(40,215)(40,227.5)(40,240)
\qbezier(40,215)(29.2,208.7)(18.5,202.5)
\qbezier(40,215)(50.5,208.7)(61.5,202.5)

\put(25,170){$(1,1,1)$}
\put(120,215){\circle{50}}
\put(120,215){\circle*{2}}
\put(120,240){\circle*{2}}
\put(120,245){\makebox(0,0){\tiny 1}}
\put(98.5,202.5){\circle*{2}}
\put(93.5,202.5){\makebox(0,0){\tiny 2}}
\put(141.5,202.5){\circle*{2}}
\put(146.5,202.5){\makebox(0,0){\tiny 3}}

\thicklines
\qbezier(120,215)(120,227.5)(120,240)
\qbezier(120,215)(109.2,208.7)(98.5,202.5)
\qbezier(98.5,202.5)(150,195)(120,240)

\put(105,170){$(2,1,0)$}
\put(200,215){\circle{50}}
\put(200,215){\circle*{2}}
\put(200,240){\circle*{2}}
\put(200,245){\makebox(0,0){\tiny 1}}
\put(178.5,202.5){\circle*{2}}
\put(173.5,202.5){\makebox(0,0){\tiny 2}}
\put(221.5,202.5){\circle*{2}}
\put(226.5,202.5){\makebox(0,0){\tiny 3}}

\thicklines
\qbezier(178.5,202.5)(190,225)(205,220)
\qbezier(205,220)(210,205)(178.5,202.5)
\qbezier(200,215)(189.2,208.7)(178.5,202.5)
\qbezier(178.5,202.5)(230,195)(200,240)

\put(185,170){$(1,2,0)$}
\put(280,215){\circle{50}}
\put(280,215){\circle*{2}}
\put(280,240){\circle*{2}}
\put(280,245){\makebox(0,0){\tiny 1}}
\put(258.5,202.5){\circle*{2}}
\put(253.5,202.5){\makebox(0,0){\tiny 2}}
\put(301.5,202.5){\circle*{2}}
\put(306.5,202.5){\makebox(0,0){\tiny 3}}

\thicklines
\qbezier(258.5,202.5)(270,225)(285,220)
\qbezier(285,220)(290,205)(258.5,202.5)
\qbezier(280,215)(269.2,208.7)(258.5,202.5)
\qbezier(301.5,202.5)(280,250)(258.5,202.5)

\put(265,170){$(0,3,0)$}
\put(360,215){\circle{50}}
\put(360,215){\circle*{2}}
\put(360,240){\circle*{2}}
\put(360,245){\makebox(0,0){\tiny 1}}
\put(338.5,202.5){\circle*{2}}
\put(333.5,202.5){\makebox(0,0){\tiny 2}}
\put(381.5,202.5){\circle*{2}}
\put(386.5,202.5){\makebox(0,0){\tiny 3}}

\thicklines
\qbezier(360,215)(371,208.5)(381.5,202.5)
\qbezier(360,215)(349.2,208.7)(338.5,202.5)
\qbezier(381.5,202.5)(360,250)(338.5,202.5)

\put(345,170){$(0,2,1)$}
\put(40,115){\circle{50}}
\put(40,115){\circle*{2}}
\put(40,140){\circle*{2}}
\put(40,145){\makebox(0,0){\tiny 1}}
\put(18.5,102.5){\circle*{2}}
\put(13.5,102.5){\makebox(0,0){\tiny 2}}
\put(61.5,102.5){\circle*{2}}
\put(66.5,102.5){\makebox(0,0){\tiny 3}}

\thicklines
\qbezier(40,115)(51,108.5)(61.5,102.5)
\qbezier(61.5,102.5)(45,125)(35,120)
\qbezier(35,120)(35,105)(61.5,102.5)
\qbezier(61.5,102.5)(40,150)(18.5,102.5)

\put(25,70){$(0,1,2)$}
\put(120,115){\circle{50}}
\put(120,115){\circle*{2}}
\put(120,140){\circle*{2}}
\put(120,145){\makebox(0,0){\tiny 1}}
\put(98.5,102.5){\circle*{2}}
\put(93.5,102.5){\makebox(0,0){\tiny 2}}
\put(141.5,102.5){\circle*{2}}
\put(146.5,102.5){\makebox(0,0){\tiny 3}}

\thicklines
\qbezier(120,115)(131,108.5)(141.5,102.5)
\qbezier(141.5,102.5)(125,125)(115,120)
\qbezier(115,120)(115,105)(141.5,102.5)
\qbezier(120,140)(90,100)(141.5,102.5)

\put(105,70){$(0,0,3)$}
\put(200,115){\circle{50}}
\put(200,115){\circle*{2}}
\put(200,140){\circle*{2}}
\put(200,145){\makebox(0,0){\tiny 1}}
\put(178.5,102.5){\circle*{2}}
\put(173.5,102.5){\makebox(0,0){\tiny 2}}
\put(221.5,102.5){\circle*{2}}
\put(226.5,102.5){\makebox(0,0){\tiny 3}}

\thicklines
\qbezier(200,115)(211,108.5)(221.5,102.5)
\qbezier(200,115)(200,127.5)(200,140)
\qbezier(200,140)(170,100)(221.5,102.5)

\put(185,70){$(1,0,2)$}
\put(280,115){\circle{50}}
\put(280,115){\circle*{2}}
\put(280,140){\circle*{2}}
\put(280,145){\makebox(0,0){\tiny 1}}
\put(258.5,102.5){\circle*{2}}
\put(253.5,102.5){\makebox(0,0){\tiny 2}}
\put(301.5,102.5){\circle*{2}}
\put(306.5,102.5){\makebox(0,0){\tiny 3}}

\thicklines
\qbezier(280,140)(270,115)(280,110)
\qbezier(280,110)(290,115)(280,140)
\qbezier(280,115)(280,127.5)(280,140)
\qbezier(280,140)(250,100)(301.5,102.5)

\put(265,70){$(2,0,1)$}
\put(360,115){\circle{50}}
\put(360,115){\circle*{2}}
\put(360,140){\circle*{2}}
\put(360,145){\makebox(0,0){\tiny 1}}
\put(338.5,102.5){\circle*{2}}
\put(333.5,102.5){\makebox(0,0){\tiny 2}}
\put(381.5,102.5){\circle*{2}}
\put(386.5,102.5){\makebox(0,0){\tiny 3}}

\thicklines
\qbezier(360,140)(350,115)(360,110)
\qbezier(360,110)(370,115)(360,140)
\qbezier(360,115)(360,127.5)(360,140)
\qbezier(338.5,102.5)(390,95)(360,140)

\put(345,70){$(3,0,0)$}
\end{picture}
\end{example}

To prove Theorem \ref{2-8}, we construct the inverse map $\Z(n)\rightarrow \T(n)$.
To $a=(a_{1}, a_{2}, \ldots, a_{n})\in \Z(n)$, we associate a sequence $a'=(a'_{1},a'_{2},\ldots, a'_{n})$ given by 
\begin{align}
a'_{i}=\displaystyle \sum_{j=1}^{i}(a_{j}-1). \notag
\end{align}
In the rest of this subsection, we extend integer sequences $a$ and $a'$ to maps
$a:\ZZ \rightarrow \ZZ_{\ge 0}$ and $a': \ZZ \rightarrow \ZZ$ by 
\begin{align}
a_{p}:=a_{(p)_{n}}\ \textnormal{and}\ a'_{p}:=a'_{(p)_{n}} \notag
\end{align}
for any integer $p$.
Since $a'$ is periodic ({\it i.e.}, $a'_{p}=a'_{p+n}$ for any integer $p$) and $a'_{p}\ge a'_{p-1}-1$ holds for any integer $p$, we have that $\{ a'_{1},a'_{2},\ldots,a'_{n}\}$ is interval, that is,
\begin{align}
\{ a'_{1},a'_{2},\ldots,a'_{n}\} =\left[\min\{ a'_{1},a'_{2},\ldots,a'_{n}\},\max\{ a'_{1},a'_{2},\ldots,a'_{n}\}\right]. \label{eq4}
\end{align}
Let
\begin{align}
||a'||=\max\{ a'_{1}, a'_{2}, \ldots, a'_{n} \}. \notag
\end{align}
For an integer $p$, let 
\begin{align}
\delta_{p}:=
\begin{cases}
1 &(a'_{p}=||a'||)\\
0 &(\textnormal{otherwise}).
\end{cases}\notag
\end{align}

We give an interpretation of admissible arcs of a triangulation $X$ in $\top(X)$.
\begin{lemma}\label{newlem1}
Let $X\in \T(n)$ and $a:=\top(X)=(a_{1},a_{2},\ldots, a_{n})$. 
\begin{itemize}
\item[(1)] Let $\langle \bullet, j_{1}\rangle, \langle \bullet, j_{2}\rangle, \ldots, \langle \bullet, j_{r}\rangle$ be all projective arcs in $X$. Then, for $i\in[1,n]$, the following hold.
\begin{itemize}
\item[(a)] $a'_{i}=||a'||$ if $i\in \{ j_{1},j_{2},\ldots,j_{r} \}$.
\item[(b)] $a'_{i}<||a'||$ if $i\notin \{ j_{1},j_{2},\ldots,j_{r} \}$.
\end{itemize}
In particular, $\langle \bullet, j \rangle \in X$ if and only if $a'_{j}=||a'||$.
\item[(2)] Assume that $j_{i}< k+1\le l-1<j^{+}_{i}$. Then $X$ contains the inner arc $\langle k, l\rangle$ if and only if the following conditions are satisfied:
\begin{itemize}
\item[(a)] $a'_{m}<a'_{k}$ for any $m\in[k+1,l-1]$.
\item[(b)] There exists $s\in[1,a_{l}-\delta_{l}]$ such that $a'_{k}=a'_{l-1}+s$.
\end{itemize}
\end{itemize}
\end{lemma}
\begin{proof}
(1) Since each triangulation contains at least one projective arc by Proposition \ref{2-4}, 
it suffices to show the following assertions for each projective arc $\langle \bullet, j_{i}\rangle$, where $j_{r+1}:=j_{r}+n$:
(i) $a'_{j_{i}}>a'_{k}$ for each $k\in[j_{i}+1,j_{i+1}-1]$, and
(ii) $a'_{j_{i}}=a'_{j_{i+1}}$.

(i) Fix $k\in[j_{i}+1,j_{i+1}-1]$. 
Since the set $X_{j_{i},k}$ consists of exactly $a_{l}$ inner arcs with a terminal point $l$ for all $l\in[j_{i}+1,k]$, we have 
\begin{align}
|X_{j_{i},k}|=a_{j_{i}+1}+a_{j_{i}+2}+\cdots+a_{k-1}+a_{k}. \notag
\end{align}
Moreover, by Lemma \ref{2-3}, we have
\begin{align}
|X_{j_{i},k}| \le k-j_{i}-1 \notag
\end{align}
and hence $a'_{j_{i}}>a'_{k}$.

(ii) Since $X_{j_{i},j_{i+1}}$ contains the inner arc $\langle j_{i}, j_{i+1}\rangle$, we have
\begin{align}
&a_{j_{i}+1}+a_{j_{i}+2}+\cdots+a_{j_{i+1}-1}+a_{j_{i+1}}=|X_{j_{i},j_{i+1}}|+|\{ \langle \bullet, j_{i+1}\rangle\}|=j_{i+1}-j_{i}, \notag
\end{align}
by Lemma \ref{2-3}, and hence $a'_{j_{i}}=a'_{j_{i+1}}$.

(2) First we show `only if' part. We assume that $\langle k,l\rangle\in X$.
Then the set $X_{k,l}$ consists of precisely $a_{m}$ inner arcs with the terminal point for each $m\in [k+1,l-1]$, and $s$ inner arcs with the terminal point $l$ for some $s\in[1,a_{l}-\delta_{l}]$. 
Thus we have 
\begin{align}
|X_{k,l}|=a_{k+1}+\cdots+a_{l-1}+s. \notag
\end{align}
Moreover, by Lemma \ref{2-3}, we have 
\begin{align}
|X_{k,l}|=(l-k)_{n}-1 \notag
\end{align}
and hence $a'_{k}=a'_{l-1}+s$. By the similar argument, we have 
\begin{align}
a_{k+1}+\cdots+a_{m-1}+a_{m}=|X_{k,(m)_{n}}|\le (m-k)_{n}-1 \notag
\end{align}
for each $m\in [k+1,l-1]$.
Hence we have $a'_{m}<a'_{k}$.

Next, we show `if' part. We assume that $\langle k,l\rangle\in X$.
Assume that $\langle k,l\rangle\notin X$. Then there exists $\langle g,h\rangle\in X$ satisfying 
the condition either (i) $g\in [k+1,l-1], h\in [l+1,j_{i}^{+}]$ or (ii) $g\in [j_{i},k-1], h\in [k+1,l-1]$. By $\langle g, h\rangle\in X$, we have 
\begin{align}
&a'_{l}<a'_{g}<a'_{k}\le a'_{l}\ \ \ \textnormal{(if (i))},\notag\\
&a'_{h}<a'_{k}<a'_{g}\le a'_{h}\ \ \ \textnormal{(if (ii))}.\notag
\end{align}
This is contradiction.
\if The set $X_{k,l}$ contains $a_{m}$ inner arcs with the terminal point $m$ for each $m\in [k+1,l-1]$.
Moreover, there exists $s\in [1,a_{l}-\delta_{l}]$ such that $X_{k,l}$ contains $s$ inner arcs with the terminal point $l$. Thus, by Lemma \ref{2-3}, we have 
\begin{align}
&a_{k+1}+\cdots+a_{l-1}+s=|X_{k,l}|=(l-j)_{n}-1. \notag\\
&a_{k+1}+\cdots+a_{m-1}+a_{m}=|X_{k,(m)_{n}}|\le (m-k)_{n}-1 \notag
\end{align}
for each $m\in [k+1,l-1]$.
Hence we have 
\begin{align}
a'_{m}<a'_{k}=a'_{l-1}+s \notag
\end{align}
for each $m\in [k+1,l-1]$.

Assume that $\langle k,l\rangle\notin X$. Then there exists $\langle g,h\rangle\in X$ satisfying 
the condition either (i) $g\in [k+1,l-1], h\in [l+1,j_{i}^{+}]$ or (ii) $g\in [j_{i},k-1], h\in [k+1,l-1]$. By $\langle g, h\rangle\in X$, we have 
\begin{align}
&a'_{l}<a'_{g}<a'_{k}\le a'_{l}\ \ \ \textnormal{(if (i))},\notag\\
&a'_{h}<a'_{k}<a'_{g}\le a'_{h}\ \ \ \textnormal{(if (ii))}.\notag
\end{align}
This is a contradiction.\fi
\end{proof}

As a result of Lemma \ref{newlem1}, we have the following proposition.
\begin{proposition}\label{newprop1}
The map $\top:\T(n)\to\Z(n)$ is an injection.
\end{proposition}

Let $a\in \Z(n)$. 
We will construct a triangulation $X_{a}$ which contains precisely $a_{l}$ admissible arcs with a terminal point $l$ for any $l\in [1,n]$.
\begin{definition}
For $a\in \Z(n)$, we define a subset $X_{a}$ of $\Arc(n)$ as follows:
\begin{itemize}
\item[(i)] Let 
\begin{align}
\{ j_{1}<j_{2}<\cdots<j_{r} \}:=\{ j\mid 1\le j \le n,\ a'_{j}=||a'||\}. \notag
\end{align}
We write $j_{i}^{+}:=j_{i+1}$ for each $i\in [1,r]$, where $j^{+}_{r}=j_{r+1}:=j_{1}+n$.
\item[(ii)] We define $(a_{l}-\delta_{l})$ inner arcs with the terminal point $l$ as follows.
For each $s\in [1,a_{l}-\delta_{l}]$, there exists $k<l-1$ satisfying $a'_{k}=a'_{l-1}+s$ by (\ref{eq4}).
Let \begin{align}
k_{s}^{l}:=\mathrm{max}\{ k \mid k<l-1,\ a'_{k}=a'_{l-1}+s \}. \notag
\end{align}
\item[(iii)] Now we define $X_{a}$ by
\begin{align}
\displaystyle X_{a}:= \{ \langle \bullet, j\rangle \mid j\in [1,n],\ a'_{j}=||a'|| \}\coprod
\{ \langle (k_{s}^{l})_{n}, l \rangle \mid l\in [1,n],\ s\in [1,a_{l}-\delta_{l}] \}. \notag
\end{align}
\end{itemize}
\end{definition}

We illustrate the construction of $X_{a}$ with the following example.
\begin{example}
Let $n=8$ and $a:=(0,4,1,0,1,0,2,0)\in\Z(8)$.
Then we have $a'=(-1,2,2,1,1,0,1,0)$.
We display the sequence $a\in \Z(n)$ as the union $L_{a}$ of the line segments connecting $(p,a'_{p})$ and $(p+1, a'_{p+1})$ for each integer $p$.

{\tiny
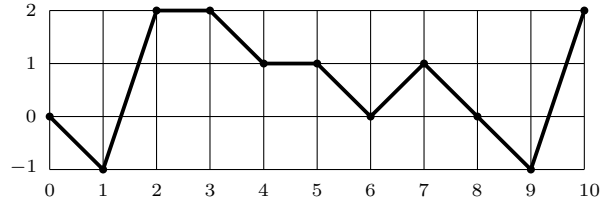
\begin{picture}(450,75)(0,0)
\multiput(125,10)(0,20){4}{\line(1,0){200}}
\multiput(125,10)(20,0){11}{\line(0,1){60}}
\put(123,0){0}
\put(143,0){1}
\put(163,0){2}
\put(183,0){3}
\put(203,0){4}
\put(223,0){5}
\put(243,0){6}
\put(263,0){7}
\put(283,0){8}
\put(303,0){9}
\put(323,0){10}
\put(110,9){$-1$}
\put(116,28){$0$}
\put(116,48){$1$}
\put(116,68){$2$}
\put(125,30){\circle*{3}}
\put(145,10){\circle*{3}}
\put(165,70){\circle*{3}}
\put(185,70){\circle*{3}}
\put(205,50){\circle*{3}}
\put(225,50){\circle*{3}}
\put(245,30){\circle*{3}}
\put(265,50){\circle*{3}}
\put(285,30){\circle*{3}}
\put(305,10){\circle*{3}}
\put(325,70){\circle*{3}}
\thicklines
\linethickness{0.5mm}
\put(125,30){\line(1,-1){20}}
\put(145,10){\line(1,3){20}}
\put(165,70){\line(1,0){20}}
\put(185,70){\line(1,-1){20}}
\put(205,50){\line(1,0){20}}
\put(225,50){\line(1,-1){20}}
\put(245,30){\line(1,1){20}}
\put(265,50){\line(1,-1){20}}
\put(285,30){\line(1,-1){20}}
\put(305,10){\line(1,3){20}}
\end{picture}}

By $||a'||=a'_{2}=a'_{3}$, the projective arcs $\langle \bullet, 2\rangle$ and $\langle \bullet, 3 \rangle$ are elements in $X_{a}$.
For an integer $l$ with $a_{l}-\delta_{l}\neq 0$ and $s\in [1, a_{l}-\delta_{l}]$, we can easily detect $k_{s}^{l}$.
For example, let $l=10$ and $s=1$. Note that $(l)_{8}=(10)_{8}=2$ and $a'_{9}+s=-1+1=0$.
Then we have $k_{1}^{2}=8$, and hence $\langle 8,2\rangle\in X_{a}$.

We draw the line segments connecting $(l,a'_{l})$ and $(k_{s}^{l}, a'_{k_{s}^{l}})$ for all $l$ and $s$ as follows:

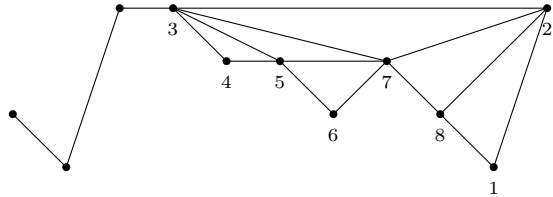
\begin{align}
{\tiny
\begin{picture}(450,75)(0,0)
\put(183,60){3}
\put(203,40){4}
\put(223,40){5}
\put(243,20){6}
\put(263,40){7}
\put(283,20){8}
\put(303,0){1}
\put(323,60){2}
\put(125,30){\circle*{3}}
\put(145,10){\circle*{3}}
\put(165,70){\circle*{3}}
\put(185,70){\circle*{3}}
\put(205,50){\circle*{3}}
\put(225,50){\circle*{3}}
\put(245,30){\circle*{3}}
\put(265,50){\circle*{3}}
\put(285,30){\circle*{3}}
\put(305,10){\circle*{3}}
\put(325,70){\circle*{3}}
\put(125,30){\line(1,-1){20}}
\put(145,10){\line(1,3){20}}
\put(165,70){\line(1,0){20}}
\put(185,70){\line(1,-1){20}}
\put(205,50){\line(1,0){20}}
\put(225,50){\line(1,-1){20}}
\put(245,30){\line(1,1){20}}
\put(265,50){\line(1,-1){20}}
\put(285,30){\line(1,-1){20}}
\put(305,10){\line(1,3){20}}
\put(325,70){\line(-1,0){140}}
\put(325,70){\line(-2,-2){40}}
\put(325,70){\line(-3,-1){60}}
\put(265,50){\line(-1,0){40}}
\put(265,50){\line(-4,1){80}}
\put(225,50){\line(-2,1){40}}
\end{picture}}\notag
\end{align}

Then the line segments do not intersect to each other except the endpoints.
By identifying the regions over $L_{a}$ with $\G_{8}$, the set $X_{a}$ is a triangulation.

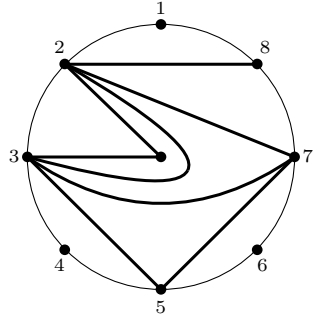
\begin{align}
{\tiny
\begin{picture}(450,100)(0,0)
\put(225,55){\circle{100}}
\put(225,55){\circle*{4}}
\linethickness{0.5mm}
\put(225,105){\circle*{4}}
\put(223,109){1}
\put(189,90){\circle*{4}}
\put(185,94){2}
\put(175,55){\circle*{4}}
\put(168,53){3}
\put(189,20){\circle*{4}}
\put(185,12){4}
\put(225,5){\circle*{4}}
\put(223,-4){5}
\put(261,20){\circle*{4}}
\put(261,12){6}
\put(275,55){\circle*{4}}
\put(278,53){7}
\put(261,90){\circle*{4}}
\put(262,94){8}
\linethickness{0.4mm}
\qbezier(225,55)(207,72.5)(189,90)
\qbezier(225,55)(200,55)(175,55)
\qbezier(175,55)(200,30)(225,5)
\qbezier(275,55)(250,30)(225,5)
\qbezier(275,55)(232,72.5)(189,90)
\qbezier(261,90)(225,90)(189,90)
\qbezier(175,55)(225,20)(275,55)
\qbezier(189,90)(240,60)(235,50)
\qbezier(175,55)(230,40)(235,50)
\end{picture}}\notag
\end{align}
\end{example}

We note that $X_{a}$ in the example above is a triangulation of $\G_{8}$. 
This is always the case as the following result shows.
\begin{proposition}\label{2-9}
We have $X_{a}\in \T(n)$ and $\top(X_{a})=a$.
\end{proposition}
\begin{proof}
We have only to show that $X_{a}\in \T(n)$ since $\top(X_{a})=a$ is clear from the construction.
Indeed, by maximality of $k_{s}^{l}$, for any integer $j\in [k_{s}^{l}+1, l-1]$ we have $a'_{j}<a'_{k_{s}^{l}}$. 
Therefore the line segments connecting $(l,a'_{l})$ and $(k_{s^{l}},a'_{k_{s}^{l}})$ do not intersect to others.
\end{proof}

Now, we are ready to prove Theorem \ref{2-8}.
\begin{proof}[Proof of Theorem \ref{2-8}]
We show that the map $a\mapsto X_{a}$ gives an inverse of the map $X\mapsto \top(X)$.
It follows from that Proposition \ref{newprop1} and \ref{2-9}.
\end{proof}
The next result gives a generalization of Theorem \ref{2-8}, 
For $l_{1},l_{2},\ldots,l_{n}\ge 1$, 
we denote by $\Z(n;l_{1},\ldots,l_{n})$ the subset of $\Z(n)$ consisting of the integer sequences $a=(a_{1},\ldots, a_{n})$
such that $\ell_{j}(a)\le l_{j}$ for each $j\in [1,n]$, where $\ell_{j}(a)$ is given by
\begin{align}
\ell_{j}(a):=
\begin{cases}
0\ \ \textnormal{if $a_{j}-\delta_{j}=0$}.\\
(j-k_{a_{j}-\delta_{j}}^{j})_{n}\ \ \textnormal{if $a_{j}-\delta_{j}>0$}.
\end{cases}\notag
\end{align}
Note that $\ell_{j}(a)$ is at most $n$ and equals to the maximal length of inner arcs in $X_{a}$ with a terminal point $j$
if $a_{j}-\delta_{j}>0$.
\begin{theorem}\label{2-10}
Let $l_{1},\ldots,l_{n}\ge 1$.
There are mutually inverse bijections
\begin{align}
\T(n;l_{1},\ldots,l_{n})\longleftrightarrow \Z(n;l_{1},\ldots,l_{n}) \notag
\end{align}
given by $X\mapsto\top(X)$ and $a\mapsto X_{a}$.
\end{theorem}
\begin{proof}
This is clear from Theorem \ref{2-8} and the definition of $X_{a}$.
\end{proof}
\subsection{Summary and applications}\label{summary}
Summarizing Theorem \ref{2-11}, \ref{2-6}, and \ref{2-10}, we have the following result.
\begin{theorem}\label{2-14}
Let $\Lambda$ be a Nakayama algebra with $n$ simple modules and $\ell_{i}:=\ell(P_{i})$ for each $i\in [1,n]$.
Then there are bijections between
\begin{itemize}
\item[(1)] $\ttilt\Lambda$,
\item[(2)] $\psttilt_{\np}\Lambda$,
\item[(3)] $\T(n;\ell_{1},\ell_{2},\ldots,\ell_{n})$,
\item[(4)] $\Z(n;\ell_{1},\ell_{2},\ldots,\ell_{n})$.
\end{itemize}
\end{theorem}

The following corollary is an immediate consequence of Theorem \ref{2-14}.
We identify $\top M=S_{1}^{a_{1}}\oplus S_{2}^{a_{2}}\oplus \cdots \oplus S_{n}^{a_{n}}$ with a sequence $(a_{1},a_{2},\ldots,a_{n})$.
\begin{corollary}
Assume that $\ell(P_{i})\ge n$ for all $i\in [1,n]$.
Then the map $M\mapsto \top M$ gives a bijection 
\begin{align}
\ttilt\Lambda \longrightarrow \Z(n). \notag
\end{align}
In particular, the cardinality of $\sttilt\Lambda$ is
\begin{align}
|\sttilt\Lambda|=\binom{2n}{n}.\notag
\end{align}
\end{corollary}
\begin{proof}
For each $M\in \ttilt\Lambda$, 
we have 
\begin{align}
\top M =\top X, \notag
\end{align}
where $X$ is the triangulation corresponding to $M$ by Theorem \ref{2-8}.
Thus the first assertion follows from Theorem \ref{2-14}.
Next, we have 
\begin{align}
|\sttilt\Lambda|=|\ttilt\Lambda|+|\psttilt\Lambda|\overset{\ref{2-11}(2)}{=}2|\ttilt\Lambda|\overset{\ref{2-14}}{=}2|\Z(n)|. \notag
\end{align}
Now, it is elementary that $2|\Z(n)|=2\binom{2n-1}{n-1}=\binom{2n}{n}$ holds.
\end{proof}

As an application of Theorem \ref{2-14}, we give a proof of the following well-known result ({\it e.g.} \cite{BK}).
\begin{corollary}\label{2-15}
Let $\Lambda:=K\vec{A}_{n}$ be a path algebra.
Then there are bijections between
\begin{itemize}
\item[(1)] $\tilt\Lambda$,
\item[(2)] $\S(n):=\{ X\in \T(n)\mid \langle \bullet,n\rangle\in X\}$,
\item[(3)] $\Y(n):=\{ a\in \Z(n)\mid ||a'||= 0 \}$.
\end{itemize}
In particular, we have
\begin{align}
|\tilt\Lambda|=C_{n}, \notag
\end{align}
where $C_{n}$ is the $n$-th Catalan number $\frac{1}{n+1}\binom{2n}{n}$.
\end{corollary}
Note that $\S(n)$ can identify the set of triangulations of $(n+2)$-regular polygon (with no puncture).
Therefore the cardinality of $\S(n)$ is equal to the $n$-th Catalan number $C_{n}$
\begin{proof}
Since $\Lambda$ is hereditary, we have $\ttilt\Lambda=\tilt\Lambda$ by Proposition \ref{1-2}(3).
Moreover, there are bijections between $\tilt\Lambda$, $\T(n;1,2,\ldots,n)$, and $\Z(n;1,2,\ldots,n)$ by Theorem \ref{2-14}.

First, we show that
\begin{align}\label{eq3}
\S(n)=\T(n;1,2,\ldots,n).
\end{align}
Indeed, assume that $X\in\S(n)$.
Since $X$ contains the projective arc $\langle \bullet,n\rangle$, 
we have $\ell(\langle i,j\rangle)\le j$ for each inner arc $\langle i,j\rangle\in X$.
Thus, we have $X\in \T(n;1,2,\ldots,n)$.
Conversely, assume that $X\in \T(n;1,2,\ldots,n)$.
Clearly, the projective arc $\langle \bullet, n\rangle$ is compatible with all admissible arc in $X$.
Thus, we have $\langle \bullet ,n\rangle\in X$, and hence $X\in \S(n)$.

Next, we show that
\begin{align}
\Y(n)=\Z(n;1,2,\ldots,n). \notag
\end{align}
Indeed, if $a\in\Z(n;1,2,\ldots,n)$, then $X_{a}$ contains the projective arc $\langle \bullet,n\rangle$ by (\ref{eq3}).
Thus we have $||a'||=a'_{n}(=0)$, and hence $a\in \Y(n)$.
Conversely, if $a\in\Y(n)$, then $a'_{n}=0=||a'||$.
Thus we have $0\le k_{a_{i}-\delta_{i}}^{i}<i\le n$, and hence $\ell_{i}(a)=i-k^{i}_{a_{i}-\delta_{i}}\le i$.
Therefore $a\in\Z(n;1,2,\ldots,n)$.
\end{proof}

\begin{picture}(420,600)(0,0)
\put(125,-10){Example of Theorem \ref{2-14}}
\thicklines 
\put(0,595){\line(1,0){420}}
\put(0,575){\line(1,0){420}}
\put(0,507){\line(1,0){420}}
\put(0,452){\line(1,0){420}}
\put(0,397){\line(1,0){420}}
\put(0,342){\line(1,0){420}}
\put(0,287){\line(1,0){420}}
\put(0,232){\line(1,0){420}}
\put(0,177){\line(1,0){420}}
\put(0,122){\line(1,0){420}}
\put(0,67){\line(1,0){420}}
\put(0,12){\line(1,0){420}}
\put(0,12){\line(0,1){583}}
\put(288,12){\line(0,1){583}}
\put(290,12){\line(0,1){583}}
\put(420,12){\line(0,1){583}}
\put(20,580){$\T(4)$}
\put(105,580){$\Z(4)$}
\put(165,580){$\ttilt\Lambda_{4}^{4}$}
\put(225,580){$\psttilt\Lambda_{4}^{4}$}
\put(300,580){$\ttilt\Lambda_{4}^{3}$}
\put(360,580){$\ttilt\Lambda_{4}^{2}$}
\put(40,540){\circle{46}}
\put(40,540){\circle*{3}}
\put(40,563){\circle*{3}}
\put(38,566){{\tiny{1}}}
\put(63,540){\circle*{3}}
\put(65,540){{\tiny{4}}}
\put(40,517){\circle*{3}}
\put(38,509){{\tiny{3}}}
\put(17,540){\circle*{3}}
\put(11,540){{\tiny{2}}}
\put(40,540){\line(0,1){23}}
\put(40,540){\line(1,0){23}}
\put(40,540){\line(0,-1){23}}
\put(40,540){\line(-1,0){23}}
\put(40,480){\circle{46}}
\put(40,480){\circle*{3}}
\put(40,503){\circle*{3}}
\put(63,480){\circle*{3}}
\put(40,457){\circle*{3}}
\put(17,480){\circle*{3}}
\put(40,480){\line(0,1){23}}
\put(40,480){\line(1,0){23}}
\put(40,480){\line(0,-1){23}}
\qbezier(40,457)(5,480)(40,503)
\put(40,425){\circle{46}}
\put(40,425){\circle*{3}}
\put(40,448){\circle*{3}}
\put(63,425){\circle*{3}}
\put(40,402){\circle*{3}}
\put(17,425){\circle*{3}}
\put(40,425){\line(0,1){23}}
\put(40,425){\line(0,-1){23}}
\qbezier(40,402)(5,425)(40,448)
\qbezier(40,448)(75,425)(40,402)
\put(40,370){\circle{46}}
\put(40,370){\circle*{3}}
\put(40,393){\circle*{3}}
\put(63,370){\circle*{3}}
\put(40,347){\circle*{3}}
\put(17,370){\circle*{3}}
\put(40,370){\line(0,1){23}}
\put(40,370){\line(1,0){23}}
\qbezier(40,347)(5,370)(40,393)
\qbezier(63,370)(40,350)(30,360)
\qbezier(30,360)(20,370)(40,393)
\put(40,315){\circle{46}}
\put(40,315){\circle*{3}}
\put(40,338){\circle*{3}}
\put(63,315){\circle*{3}}
\put(40,292){\circle*{3}}
\put(17,315){\circle*{3}}
\put(40,315){\line(0,1){23}}
\put(40,315){\line(1,0){23}}
\qbezier(63,315)(40,295)(30,305)
\qbezier(30,305)(20,315)(40,338)
\qbezier(63,315)(40,280)(17,315)
\put(40,260){\circle{46}}
\put(40,260){\circle*{3}}
\put(40,283){\circle*{3}}
\put(63,260){\circle*{3}}
\put(40,237){\circle*{3}}
\put(17,260){\circle*{3}}
\put(40,260){\line(0,1){23}}
\qbezier(40,237)(5,260)(40,283)
\qbezier(63,260)(40,240)(30,250)
\qbezier(30,250)(20,260)(40,283)
\qbezier(40,283)(50,250)(40,250)
\qbezier(40,250)(30,250)(40,283)
\put(40,205){\circle{46}}
\put(40,205){\circle*{3}}
\put(40,228){\circle*{3}}
\put(63,205){\circle*{3}}
\put(40,182){\circle*{3}}
\put(17,205){\circle*{3}}
\put(40,205){\line(0,1){23}}
\qbezier(40,182)(5,205)(40,228)
\qbezier(40,228)(75,205)(40,182)
\qbezier(40,228)(50,195)(40,195)
\qbezier(40,195)(30,195)(40,228)
\put(40,150){\circle{46}}
\put(40,150){\circle*{3}}
\put(40,173){\circle*{3}}
\put(63,150){\circle*{3}}
\put(40,127){\circle*{3}}
\put(17,150){\circle*{3}}
\put(40,150){\line(0,1){23}}
\qbezier(40,173)(75,150)(40,127)
\qbezier(40,173)(50,140)(40,140)
\qbezier(40,140)(30,140)(40,173)
\qbezier(40,173)(60,150)(50,140)
\qbezier(50,140)(40,130)(17,150)
\put(40,95){\circle{46}}
\put(40,95){\circle*{3}}
\put(40,118){\circle*{3}}
\put(63,95){\circle*{3}}
\put(40,72){\circle*{3}}
\put(17,95){\circle*{3}}
\put(40,95){\line(0,1){23}}
\qbezier(63,95)(40,60)(17,95)
\qbezier(40,118)(50,85)(40,85)
\qbezier(40,85)(30,85)(40,118)
\qbezier(40,118)(60,95)(50,85)
\qbezier(50,85)(40,75)(17,95)
\put(40,40){\circle{46}}
\put(40,40){\circle*{3}}
\put(40,63){\circle*{3}}
\put(63,40){\circle*{3}}
\put(40,17){\circle*{3}}
\put(17,40){\circle*{3}}
\put(40,40){\line(0,1){23}}
\qbezier(63,40)(40,20)(30,30)
\qbezier(30,30)(20,40)(40,63)
\qbezier(63,40)(40,5)(17,40)
\qbezier(40,63)(50,30)(40,30)
\qbezier(40,30)(30,30)(40,63)
\put(90,540){$(1,1,1,1)$}
\put(90,480){$(1,0,2,1)$}
\put(90,425){$(2,0,2,0)$}
\put(90,370){$(1,0,1,2)$}
\put(90,315){$(1,0,0,3)$}
\put(90,260){$(2,0,1,1)$}
\put(90,205){$(3,0,1,0)$}
\put(90,150){$(4,0,0,0)$}
\put(90,95){$(3,0,0,1)$}
\put(90,40){$(2,0,0,2)$}
\put(168,540){$\begin{smallmatrix}
1\\
4\\
3\\
2
\end{smallmatrix}\begin{smallmatrix}
2\\
1\\
4\\
3
\end{smallmatrix}\begin{smallmatrix}
3\\
2\\
1\\
4
\end{smallmatrix}\begin{smallmatrix}
4\\
3\\
2\\
1
\end{smallmatrix}$}
\put(168,480){$\begin{smallmatrix}
1\\
4\\
3\\
2
\end{smallmatrix}\begin{smallmatrix}
3\\
2\\
1\\
4
\end{smallmatrix}\begin{smallmatrix}
4\\
3\\
2\\
1
\end{smallmatrix}\begin{smallmatrix}
3
\end{smallmatrix}$}
\put(168,425){$\begin{smallmatrix}
1\\
4\\
3\\
2
\end{smallmatrix}\begin{smallmatrix}
3\\
2\\
1\\
4
\end{smallmatrix}\begin{smallmatrix}
1
\end{smallmatrix}\begin{smallmatrix}
3
\end{smallmatrix}$}\put(168,370){$\begin{smallmatrix}
1\\
4\\
3\\
2
\end{smallmatrix}\begin{smallmatrix}
4\\
3\\
2\\
1
\end{smallmatrix}\begin{smallmatrix}
4\\
3
\end{smallmatrix}\begin{smallmatrix}
3
\end{smallmatrix}$}\put(168,315){$\begin{smallmatrix}
1\\
4\\
3\\
2
\end{smallmatrix}\begin{smallmatrix}
4\\
3\\
2\\
1
\end{smallmatrix}\begin{smallmatrix}
4\\
3
\end{smallmatrix}\begin{smallmatrix}
4\\
\end{smallmatrix}$}
\put(168,260){$\begin{smallmatrix}
1\\
4\\
3\\
2
\end{smallmatrix}\begin{smallmatrix}
1\\
4\\
3
\end{smallmatrix}\begin{smallmatrix}
4\\
3
\end{smallmatrix}\begin{smallmatrix}
3
\end{smallmatrix}$}
\put(168,205){$\begin{smallmatrix}
1\\
4\\
3\\
2
\end{smallmatrix}\begin{smallmatrix}
1\\
4\\
3
\end{smallmatrix}\begin{smallmatrix}
1
\end{smallmatrix}\begin{smallmatrix}
3
\end{smallmatrix}$}
\put(168,150){$\begin{smallmatrix}
1\\
4\\
3\\
2
\end{smallmatrix}\begin{smallmatrix}
1\\
4\\
3
\end{smallmatrix}\begin{smallmatrix}
1\\
4
\end{smallmatrix}\begin{smallmatrix}
1
\end{smallmatrix}$}
\put(168,95){$\begin{smallmatrix}
1\\
4\\
3\\
2
\end{smallmatrix}\begin{smallmatrix}
1\\
4\\
3
\end{smallmatrix}\begin{smallmatrix}
1\\
4
\end{smallmatrix}\begin{smallmatrix}
4
\end{smallmatrix}$}
\put(168,40){$\begin{smallmatrix}
1\\
4\\
3\\
2
\end{smallmatrix}\begin{smallmatrix}
1\\
4\\
3
\end{smallmatrix}\begin{smallmatrix}
4\\
3
\end{smallmatrix}\begin{smallmatrix}
4
\end{smallmatrix}$}
\put(240,540){$\begin{smallmatrix}
\{ 0\}
\end{smallmatrix}$}
\put(244,480){$\begin{smallmatrix}
3
\end{smallmatrix}$}
\put(240,425){$\begin{smallmatrix}
1
\end{smallmatrix}\begin{smallmatrix}
3
\end{smallmatrix}$}
\put(240,370){$\begin{smallmatrix}
4\\
3
\end{smallmatrix}\begin{smallmatrix}
3
\end{smallmatrix}$}
\put(240,315){$\begin{smallmatrix}
4\\
3
\end{smallmatrix}\begin{smallmatrix}
4
\end{smallmatrix}$}
\put(236,260){$\begin{smallmatrix}
1\\
4\\
3
\end{smallmatrix}\begin{smallmatrix}
4\\
3
\end{smallmatrix}\begin{smallmatrix}
3
\end{smallmatrix}$}
\put(236,205){$\begin{smallmatrix}
1\\
4\\
3
\end{smallmatrix}\begin{smallmatrix}
1
\end{smallmatrix}\begin{smallmatrix}
3
\end{smallmatrix}$}
\put(236,150){$\begin{smallmatrix}
1\\
4\\
3
\end{smallmatrix}\begin{smallmatrix}
1\\
4
\end{smallmatrix}\begin{smallmatrix}
1
\end{smallmatrix}$}
\put(236,95){$\begin{smallmatrix}
1\\
4\\
3
\end{smallmatrix}\begin{smallmatrix}
1\\
4
\end{smallmatrix}\begin{smallmatrix}
4
\end{smallmatrix}$}
\put(236,40){$\begin{smallmatrix}
1\\
4\\
3
\end{smallmatrix}\begin{smallmatrix}
4\\
3
\end{smallmatrix}\begin{smallmatrix}
4
\end{smallmatrix}$}
\put(305,540){$\begin{smallmatrix}
1\\
4\\
3
\end{smallmatrix}\begin{smallmatrix}
2\\
1\\
4
\end{smallmatrix}\begin{smallmatrix}
3\\
2\\
1
\end{smallmatrix}\begin{smallmatrix}
4\\
3\\
2
\end{smallmatrix}$}
\put(305,480){$\begin{smallmatrix}
1\\
4\\
3
\end{smallmatrix}\begin{smallmatrix}
3\\
2\\
1
\end{smallmatrix}\begin{smallmatrix}
4\\
3\\
2
\end{smallmatrix}\begin{smallmatrix}
3
\end{smallmatrix}$}
\put(305,425){$\begin{smallmatrix}
1\\
4\\
3
\end{smallmatrix}\begin{smallmatrix}
3\\
2\\
1
\end{smallmatrix}\begin{smallmatrix}
1
\end{smallmatrix}\begin{smallmatrix}
3
\end{smallmatrix}$}
\put(305,370){$\begin{smallmatrix}
1\\
4\\
3
\end{smallmatrix}\begin{smallmatrix}
4\\
3\\
2
\end{smallmatrix}\begin{smallmatrix}
4\\
3
\end{smallmatrix}\begin{smallmatrix}
3
\end{smallmatrix}$}
\put(305,315){$\begin{smallmatrix}
1\\
4\\
3
\end{smallmatrix}\begin{smallmatrix}
4\\
3\\
2
\end{smallmatrix}\begin{smallmatrix}
4\\
3
\end{smallmatrix}\begin{smallmatrix}
4
\end{smallmatrix}$}
\put(365,540){$\begin{smallmatrix}
1\\
4
\end{smallmatrix}\begin{smallmatrix}
2\\
1
\end{smallmatrix}\begin{smallmatrix}
3\\
2
\end{smallmatrix}\begin{smallmatrix}
4\\
3
\end{smallmatrix}$}
\put(365,480){$\begin{smallmatrix}
1\\
4
\end{smallmatrix}\begin{smallmatrix}
3\\
2
\end{smallmatrix}\begin{smallmatrix}
4\\
3
\end{smallmatrix}\begin{smallmatrix}
3
\end{smallmatrix}$}
\put(365,425){$\begin{smallmatrix}
1\\
4
\end{smallmatrix}\begin{smallmatrix}
3\\
2
\end{smallmatrix}\begin{smallmatrix}
1
\end{smallmatrix}\begin{smallmatrix}
3
\end{smallmatrix}$}
\end{picture}
\subsection{Miscellaneous results on Nakayama algebras of type $A$}
In this subsection, we give another classification of $\tau$-tilting modules over Nakayama algebras of type $A$.
The calculation of proper support $\tau$-tilting modules over Nakayama algebras 
can be reduced to that of $\tau$-tilting modules over smaller Nakayama algebras of type $A$.
Moreover, as an application, we give the number of $\tau$-tilting modules as a recurrence relation.

Throughout this subsection, we assume that $\Lambda$ is a Nakayama algebra of type $\vec{A}^{\op}_{n}$.
Here we use the quiver $\vec{A}_{n}^{\op}$ instead of $\vec{A}_{n}$ because the opposite index is more convenient for our setting.
Namely, its quiver is isomorphic to $\vec{A}_{n}^{\op}$, that is, 
\begin{align}
\vec{A}_{n}^{\op}: \xymatrix{\begin{smallmatrix}n\end{smallmatrix}&\begin{smallmatrix}n-1\end{smallmatrix}\ar[l]_{\alpha_{n-1}}
&\ldots\ar[l]_{\alpha_{n-2}}&\begin{smallmatrix}2\end{smallmatrix}\ar[l]_{\alpha_{2}}&\begin{smallmatrix}1\end{smallmatrix}\ar[l]_{\alpha_{1}}}\notag
\end{align}
Thus $\tau M$ has the following property for any $M\in \mod\Lambda$.
\begin{lemma}\label{2-16}
If $M$ is a $\tau$-rigid $\Lambda$-module, then $M\oplus P_{1}$ is also a $\tau$-rigid $\Lambda$-module.
\end{lemma}
\begin{proof}
Since the vertex $1$ in $\vec{A}_{n}^{\op}$ is a source,
$\tau M$ does not have $S_{1}$ as a composition factor by Proposition \ref{ass3.5}.
Hence $\Hom_{\Lambda}(P_{1},\tau M)=0$ by Lemma \ref{1-1}(1).
Thus $M\oplus P_{1}$ is $\tau$-rigid.
\end{proof}
By Lemma \ref{2-16}, we have the following result for $\tau$-tilting modules.
\begin{proposition}\label{2-17}
Each $\tau$-tilting $\Lambda$-module has $P_{1}$ as a direct summand.
\end{proposition}
\begin{proof}
Let $M$ be a basic $\tau$-tilting $\Lambda$-module.
Then $M\oplus P_{1}$ is a $\tau$-rigid $\Lambda$-module by Lemma \ref{2-16}.
Thus we have $P_{1}\in \add M$ by Proposition \ref{air2.10}.
\end{proof}
Our aim in this subsection is to show the following theorem.
\begin{theorem}\label{2-18}
Let $\Lambda$ be a Nakayama algebra of type $\vec{A}_{n}^{\op}$.
Then there are mutually inverse bijections
\begin{align}
\displaystyle \ttilt \Lambda \longleftrightarrow 
\coprod_{i=1}^{\ell(P_{1})} \ttilt \left( \Lambda/ \langle e_{i}\rangle \right) \notag
\end{align}
given by $\ttilt\Lambda\ni M \mapsto M/P_{1}$ and $N\mapsto N\oplus P_{1}\in\ttilt\Lambda$.
\end{theorem}
\begin{proof}
Let $N\in \ttilt(\Lambda/\langle e_{i}\rangle)$ for $i\in [1,\ell(P_{1})]$.
Then $N\oplus P_{1}$ is a $\tau$-tilting $\Lambda$-module 
because it is $\tau$-rigid by Lemma \ref{2-16} and $|N\oplus P_{1}|=|\Lambda|$ by $P_{1}\notin \mod(\Lambda/\langle e_{i}\rangle)$.

Conversely, let $M$ be a basic $\tau$-tilting $\Lambda$-module.
Then, by Proposition \ref{2-17}, we decompose $M$ as $M=P_{1}\oplus N_{1}\oplus N_{2}$, 
where $N_{1}$ is a maximal direct summand of $M$ consisting of $(\Lambda/\langle e_{1}\rangle)$-modules
and $N_{2}$ is a direct summand of $M$ with $\top N_{2}\in\add S_{1}$.
If $N_{2}=0$, then $N_{1}$ is a $\tau$-tilting $(\Lambda/\langle e_{1}\rangle)$-module clearly.
Assume that $N_{2}\neq 0$.
Then we have $N_{2}\in \mod(\Lambda/\langle e_{j+1}\rangle)$ for $j:=\ell(N_{2})$.
Note that $1\le j<\ell(P_{1})$.
Now, we claim that $N_{1}\in \mod(\Lambda/\langle e_{j+1}\rangle)$.
Indeed, if it does not hold, then there exists an indecomposable $\Lambda$-module $Y\in\add N_{1}$ which has $S_{j+1}$ as a composition factor.
Let $X:=P_{1}/\rad^{j}P_{1}\in\add N_{2}$. 
Since $\soc(\tau X)=S_{j+1}$ and $\top(\tau X)=S_{2}$ hold by Proposition \ref{ass3.5}, 
we have $\Hom_{\Lambda}(Y,\tau X)\neq 0$ by Lemma \ref{2-1}(2)$\Rightarrow$(1).
However, this contradicts that $M$ is $\tau$-rigid.
Thus $N_{1}\oplus N_{2}$ is a $\tau$-tilting $(\Lambda/\langle e_{j+1}\rangle)$-module.
\end{proof}
As an application of Theorem \ref{2-18}, we have a recurrence relation for the cardinality of $\ttilt\Lambda$.
\begin{corollary}\label{2-19}
\begin{itemize}
\item[(1)] Let $\Lambda$ be a Nakayama algebra of type $A_{n}^{\op}$.
Then we have \begin{align}
\displaystyle |\ttilt \Lambda|=\sum_{i=1}^{\ell(P_{1})}C_{i-1}\cdot |\ttilt (\Lambda/\langle e_{\le i}\rangle)| \notag
\end{align}
where $e_{\le i}:=e_{1}+e_{2}+\cdots +e_{i}$.
\item[(2)] Let $\Lambda=\Gamma_{n}^{r}:=K\vec{A}^{\op}_{n}/\rad^{r}K\vec{A}^{\op}_{n}$.
Then we have 
\begin{align}
\displaystyle |\ttilt \Gamma_{n}^{r}|=\sum_{i=1}^{r}C_{i-1}\cdot |\ttilt \Gamma_{n-i}^{r}|. \notag
\end{align}
\end{itemize}
\end{corollary}
\begin{proof}
(1) By Theorem \ref{2-18}, we have
\begin{align}
|\ttilt\Lambda|=\sum_{i=1}^{\ell(P_{1})}|\ttilt(\Lambda/\langle e_{i}\rangle)|. \notag
\end{align}
Since the quiver of $\Lambda$ is a tree,
we have $\Lambda/\langle e_{i}\rangle \simeq (\Lambda/\langle e_{\ge i}\rangle) \times (\Lambda/\langle e_{\le i}\rangle)$,
where $e_{\ge i}:=e_{i}+e_{i+1}+\cdots +e_{n}$ and $e_{\le i}:=e_{1}+e_{2}+\cdots +e_{i}$ for $i\in [1,n]$.
Thus there is a bijection
\begin{align}
\ttilt(\Lambda/\langle e_{\ge i}\rangle)\times\ttilt(\Lambda/\langle e_{\le i}\rangle) \longrightarrow \ttilt(\Lambda/\langle e_{i}\rangle)\notag 
\end{align}
given by $(N,L)\mapsto N\oplus L$.
Hence we have 
\begin{align}
|\ttilt(\Lambda/\langle e_{i}\rangle)|=|\ttilt(\Lambda/\langle e_{\ge i}\rangle)|\cdot |\ttilt (\Lambda/\langle e_{\le i}\rangle)|. \notag
\end{align}
If $i \le\ell(P_{1})$, then $\Lambda/\langle e_{\ge i} \rangle$ is isomorphic to $K\vec{A}^{\op}_{i-1}$.
Thus, by Corollary \ref{2-15}, we have 
\begin{align}
|\ttilt(\Lambda/\langle e_{\ge i} \rangle)|=|\tilt (K\vec{A}_{i-1})|=C_{i-1}. \notag
\end{align}
Therefore the assertion follows.

(2) If $\Lambda=\Gamma_{n}^{r}$, then we have $\Lambda/\langle e_{\le i}\rangle\simeq\Gamma_{n-i}^{r}$.
Hence it is clear from (1). 
\end{proof}
\begin{example}
We calculate $\ttilt \Gamma_{n}^{2}$ inductively. First, we already know that 
\begin{align}
\tilt(K\vec{A}_{1})=\{ \begin{smallmatrix}
1
\end{smallmatrix}\},\ \  \tilt(K\vec{A}_{2})=\{ 
\begin{smallmatrix}
2
\end{smallmatrix}\oplus\begin{smallmatrix}
1\\
2
\end{smallmatrix},
\begin{smallmatrix}
1
\end{smallmatrix}\oplus\begin{smallmatrix}
1\\
2
\end{smallmatrix}
\}.\notag
\end{align}
By Theorem \ref{2-18}, 
\begin{align}
\ttilt \Gamma_{3}^{2}\longleftrightarrow 
\ttilt (\Gamma_{3}^{2}/\langle e_{2}\rangle) \coprod \ttilt (\Gamma_{3}^{2}/\langle e_{1}\rangle)
=\{ \begin{smallmatrix}
1
\end{smallmatrix}\oplus\begin{smallmatrix}
3
\end{smallmatrix} \}\coprod \{ 
\begin{smallmatrix}
3
\end{smallmatrix}\oplus\begin{smallmatrix}
2\\
3
\end{smallmatrix},
\begin{smallmatrix}
2
\end{smallmatrix}\oplus\begin{smallmatrix}
2\\
3
\end{smallmatrix}\}.\notag
\end{align}
Thus we have
\begin{align}
\ttilt \Gamma_{3}^{2}=
\{ 
\begin{smallmatrix}
3
\end{smallmatrix}\oplus\begin{smallmatrix}
2\\
3
\end{smallmatrix}\oplus\begin{smallmatrix}
1\\
2
\end{smallmatrix},
\begin{smallmatrix}
2
\end{smallmatrix}\oplus\begin{smallmatrix}
2\\
3
\end{smallmatrix}\oplus\begin{smallmatrix}
1\\
2
\end{smallmatrix},
\begin{smallmatrix}
1
\end{smallmatrix}\oplus\begin{smallmatrix}
3
\end{smallmatrix}\oplus\begin{smallmatrix}
1\\
2
\end{smallmatrix} \}.\notag
\end{align}
Similarly, we have
\begin{align}
\ttilt \Gamma_{4}^{2}=
&\{ 
\begin{smallmatrix}
4
\end{smallmatrix}\oplus\begin{smallmatrix}
3\\
4
\end{smallmatrix}\oplus\begin{smallmatrix}
1
\end{smallmatrix}\oplus\begin{smallmatrix}
1\\
2
\end{smallmatrix},
\begin{smallmatrix}
3
\end{smallmatrix}\oplus\begin{smallmatrix}
3\\
4
\end{smallmatrix}\oplus\begin{smallmatrix}
1
\end{smallmatrix}\oplus\begin{smallmatrix}
1\\
2
\end{smallmatrix} \}
\notag \\
&\coprod
\{ 
\begin{smallmatrix}
4
\end{smallmatrix}\oplus\begin{smallmatrix}
3\\
4
\end{smallmatrix}\oplus\begin{smallmatrix}
2\\
3
\end{smallmatrix}\oplus\begin{smallmatrix}
1\\
2
\end{smallmatrix},
\begin{smallmatrix}
3
\end{smallmatrix}\oplus\begin{smallmatrix}
3\\
4
\end{smallmatrix}\oplus\begin{smallmatrix}
2\\
3
\end{smallmatrix}\oplus\begin{smallmatrix}
1\\
2
\end{smallmatrix},
\begin{smallmatrix}
2
\end{smallmatrix}\oplus\begin{smallmatrix}
4
\end{smallmatrix}\oplus\begin{smallmatrix}
2\\
3
\end{smallmatrix}\oplus\begin{smallmatrix}
1\\
2
\end{smallmatrix} \}.\notag
\end{align}
\end{example}
As a similar result of Corollary \ref{2-19}, Jasso \cite{J2} showed a recurrence relation
\begin{align}
|\sttilt\Gamma_{n}^{2}|=2|\sttilt\Gamma_{n-1}^{2}|+|\sttilt\Gamma_{n-2}^{2}|. \notag
\end{align}

We give examples of the number of support $\tau$-tilting modules over some Nakayama algebras. 
\begin{table}[htb]
\begin{center}
\begin{tabular}{c}
\begin{minipage}{0.5\hsize}
\begin{center}
\caption{$|\ttilt \Gamma_{n}^{r}|$}
\begin{tabular}{|c||c|c|c|c|c|} 
\hline
\backslashbox{\ \ \ $r$}{$n$\ \ \ } & 1 & 2 & 3 & 4 & 5 \\ \hline \hline
$1$ & 1 & 1 & 1 & 1 & 1  \\ \hline
$2$ & 1 & 2 & 3 & 5 & 8  \\ \hline
$3$ & 1 & 2 & 5 & 9 & 18 \\ \hline
$4$ & 1 & 2 & 5 & 14& 28 \\ \hline
$5$ & 1 & 2 & 5 & 14& 42 \\ \hline
\end{tabular}
\end{center}
\end{minipage}
\begin{minipage}{0.6\hsize}
\begin{center}
\caption{$|\sttilt \Gamma_{n}^{r}|$}
\begin{tabular}{|c||c|c|c|c|c|} 
\hline
\backslashbox{\ \ \ $r$}{$n$\ \ \ } & 1 & 2 & 3 & 4 & 5  \\ \hline \hline
$1$ & 2 & 4 & 8 & 16& 32  \\ \hline
$2$ & 2 & 5 & 12& 29& 70  \\ \hline
$3$ & 2 & 5 & 14& 37& 98  \\ \hline
$4$ & 2 & 5 & 14& 42&118  \\ \hline
$5$ & 2 & 5 & 14& 42&132  \\ \hline
\end{tabular}
\end{center}
\end{minipage}
\end{tabular}
\begin{tabular}{c}
\begin{minipage}{0.5\hsize}
\begin{center}
\caption{$|\ttilt\Lambda_{n}^{r}|$}
\begin{tabular}{|c||c|c|c|c|c|} 
\hline
\backslashbox{\ \ \ $r$}{$n$\ \ \ } & 1 & 2 & 3 & 4 & 5 \\ \hline \hline
$1$ & 1 & 1 & 1 & 1 & 1  \\ \hline
$2$ & 1 & 3 & 4 & 7 & 11  \\ \hline
$3$ & 1 & 3 & 10& 15& 31 \\ \hline
$4$ & 1 & 3 & 10& 35& 56 \\ \hline
$5$ & 1 & 3 & 10& 35&126 \\ \hline
\end{tabular}
\end{center}
\end{minipage}
\begin{minipage}{0.6\hsize}
\begin{center}
\caption{$|\sttilt\Lambda_{n}^{r}|$}
\begin{tabular}{|c||c|c|c|c|c|} 
\hline
\backslashbox{\ \ \ $r$}{$n$\ \ \ } & 1 & 2 & 3 & 4 & 5  \\ \hline \hline
$1$ & 2 & 4 & 8 & 16& 32 \\ \hline
$2$ & 2 & 6 & 14& 34& 82 \\ \hline
$3$ & 2 & 6 & 20& 50&132 \\ \hline
$4$ & 2 & 6 & 20& 70&182 \\ \hline
$5$ & 2 & 6 & 20& 70&252 \\ \hline
\end{tabular}
\end{center}
\end{minipage}
\end{tabular}
\end{center}
\end{table}
\section{$\tau$-tilting modules and Drozd-Kirichenko rejection}
Our aim in this section is to study a connection of support $\tau$-tilting modules 
between two algebras related by Drozd-Kirichenko rejection.
As an application, we give an algorithm to construct the Hasse quivers of support $\tau$-tilting modules over Nakayama algebras. 

The rejection Lemma of Drozd-Kirichenko, stated below as Proposition \ref{3-1}, provides an explicit relation 
between representation theory of $\Lambda$
and that of factor algebra $\Lambda/S$, where $S=\soc Q$ for an indecomposable projective-injective direct summand of $\Lambda$.
Recall that the natural map $\Lambda \to \Lambda/S$ induces a fully faithful functor $\mod(\Lambda/S)\to\mod\Lambda$, 
which allows us to regard $\mod(\Lambda/S)$ as a full subcategory of $\mod\Lambda$.
We denote by $\ind\Lambda$ the set of isomorphism classes of indecomposable $\Lambda$-modules.
\begin{proposition}\cite{DK}\label{3-1}
Let $\Lambda$ be a basic finite dimensional algebra and $Q$ an indecomposable projective-injective summand of $\Lambda$ as a $\Lambda$-module.
Then the following hold.
\begin{itemize}
\item[(1)] $S:=\soc Q$ is a two-sided ideal of $\Lambda$.
\item[(2)] $\ind(\Lambda/S)=\ind(\Lambda)\setminus \{ Q \}$.
\item[(3)] $\tau_{\Lambda}U\simeq\tau_{\Lambda/S}U$ for any $U\in \mod(\Lambda/S)$ with $U\not\simeq Q/S$, and $\tau_{\Lambda}(Q/S)\simeq \rad Q$.
\item[(4)] Any almost split sequence in $\mod(\Lambda/S)$ is an almost split sequence in $\mod \Lambda$.
\item[(5)] All almost split sequences in $\mod \Lambda$ are either almost split sequences in $\mod(\Lambda/S)$ or 
\begin{align}
\xymatrix{
0\ar[r]&\rad Q\ar[r]&Q\oplus \rad Q/S\ar[r]&Q/S\ar[r]&0.
}\notag
\end{align}
\item[(6)] $Q/S$ is an indecomposable projective $(\Lambda/S)$-module and $\rad Q$ is an indecomposable injective $(\Lambda/S)$-module.
\end{itemize}
\end{proposition}
\subsection{Main results}
Let $\Lambda$ be an arbitrary basic finite dimensional algebra (we do not assume that $\Lambda$ is Nakayama).
We always assume that $\Lambda$ has an indecomposable projective-injective summand $Q$ as a $\Lambda$-module.
Moreover, let $S:=\soc Q$ and $\overline{\Lambda}:=\Lambda/S$. 
We consider the functor
\begin{align}
\overline{(-)}:=-\otimes_{\Lambda}\bLambda:\mod\Lambda\rightarrow \mod\bLambda.\notag
\end{align}
Then $\overline{Q}=Q/S$. 
Note that, for every indecomposable $\Lambda$-module $M\not\simeq Q$, 
we have an isomorphism $\overline{M}\simeq M$ as $\bLambda$-module.

In this subsection, we show that the poset $\sttilt\Lambda$ can be constructed from the poset $\sttilt\bLambda$.
The following construction is crucial.
\begin{definition}\label{3-2}
Let $\po=(\po,\ge)$ be a poset and $\mathsf{N}$ a subposet of $\po$.
\begin{itemize}
\item[(1)] We define a new poset $\po^{\mathsf{N}}=(\po^{\mathsf{N}},\ge_{\mathsf{N}})$ as follows,
where $\mathsf{N}^{+}:=\{ n^{+}\mid n\in \mathsf{N}\}$ is a copy of $\mathsf{N}$, 
and $\omega_{1},\omega_{2}\in\po$ and $n_{1},n_{2}\in \mathsf{N}$ are arbitrary elements:
\begin{align}
&\po^{\mathsf{N}}:=\po\coprod \mathsf{N}^{+},\notag\\
&\omega_{1}\ge_{\mathsf{N}} \omega_{2} :\Leftrightarrow \omega_{1}\ge \omega_{2},\notag\\
&\omega_{1}\ge_{\mathsf{N}}n_{1}^{+} :\Leftrightarrow \omega_{1}\ge n_{1},\notag\\
&n_{2}^{+}\ge_{\mathsf{N}}\omega_{2} :\Leftrightarrow n_{2}\ge\omega_{2},\notag\\
&n_{1}^{+}\ge_{\mathsf{N}}n_{2}^{+} :\Leftrightarrow n_{1}\ge n_{2}.\notag
\end{align}
In particular, $n_{1}\ge_{\mathsf{N}} n_{2}^{+}$ never holds.
It is easily checked that $(\po^{\mathsf{N}},\ge_{\mathsf{N}})$ forms a poset.

\item[(2)] Let $\hasse(\po):=(\po,\hasse_{\arr})$ be the Hasse quiver of $\po$.
We define a new quiver $\hasse(\po)^{\mathsf{N}}:=(\po^{\mathsf{N}},\hasse^{\mathsf{N}}_{\arr})$ as follows, 
where $\omega_{1},\omega_{2}$ are arbitrary elements in $\po\setminus \mathsf{N}$ and $n_{1},n_{2}$ are arbitrary elements in $\mathsf{N}$:
\begin{align}
\hasse^{\mathsf{N}}_{\arr}
&=\{ \omega_{1}\rightarrow \omega_{2}\mid \omega_{1}\rightarrow \omega_{2}\ \textnormal{in}\ \hasse_{\arr}\}
\coprod\{ n_{2}\rightarrow \omega_{2}\mid n_{2}\rightarrow \omega_{2}\ \textnormal{in}\ \hasse_{\arr}\}\notag\\
&\ \ \ \coprod\{ n_{1}\rightarrow n_{2},\ n_{1}^{+}\rightarrow n_{2}^{+}\mid n_{1}\rightarrow n_{2}\ \textnormal{in}\ \hasse_{\arr}\}\notag\\
&\ \ \ \coprod\{ \omega_{1}\rightarrow n_{1}^{+}\mid \omega_{1}\rightarrow n_{1}\ \textnormal{in}\ \hasse_{\arr}\}
\coprod\{ n_{1}^{+}\rightarrow n_{1}\mid n_{1}\in \po \}. \notag
\end{align}
It is easy to check that
\begin{align}
\hasse(\po^{\mathsf{N}})=\hasse(\po)^{\mathsf{N}} \notag
\end{align}
holds.
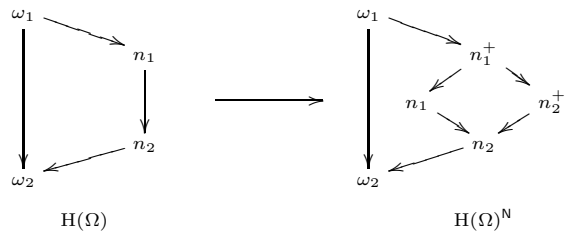
\begin{align}
\tiny{
\xymatrix@C=1mm@R=1mm{
\omega_{1}\ar[dddd]\ar[drr]&&&&&&&&&&\omega_{1}\ar[dddd]\ar[drr]&&&\\
&&n_{1}\ar[dd]&&&&&&&&&&n_{1}^{+}\ar[dl]\ar[dr]&\\
&&&&\ar[rrrrr]&&&&&&&n_{1}\ar[dr]&&n_{2}^{+}\ar[dl]\\
&&n_{2}\ar[dll]&&&&&&&&&&n_{2}\ar[dll]&\\
\omega_{2}&&&&&&&&&&\omega_{2}&&&\\
&\hasse(\po)&&&&&&&&&&&\hasse(\po)^{\mathsf{N}}&
}}\notag
\end{align}
\end{itemize}
\end{definition}
The following theorem is our main result.
Let 
\begin{align}
\N:=\{ N\in\sttilt\bLambda\ |\ \bQ\in\add N\ \textnormal{and}\ \Hom_{\Lambda}(N,Q)=0 \}.\notag
\end{align}
Applying Definition \ref{3-2}, we have a poset $(\sttilt\bLambda)^{\N}$.
For any $\Lambda$-module $M$, we denote by $\alpha(M)$ a basic $\Lambda$-module satisfying $\add(\alpha(M))=\add\overline{M}$.
\begin{theorem}\label{3-4}
Let $\Lambda$ be a basic finite dimensional algebra and $Q$ an indecomposable projective-injective summand of $\Lambda$ as a $\Lambda$-module.
Then the following hold.
\begin{itemize}
\item[(1)] There is an isomorphism of posets
\begin{align}
\sttilt\Lambda\longrightarrow (\sttilt\bLambda)^{\N} \notag
\end{align}
given by $M\mapsto\alpha(M)$.
In particular, we have an isomorphism of Hasse quivers
\begin{align}
\hasse(\Lambda)\simeq \hasse(\bLambda)^{\N}.\notag
\end{align}
\item[(2)] Assume that $\bQ$ has $S$ as a composition factor. 
Then $\N=\emptyset$ holds, and we have an isomorphism of posets
\begin{align}
\sttilt\Lambda \simeq \sttilt\bLambda \notag
\end{align}
and an isomorphism of Hasse quivers
\begin{align}
\hasse(\Lambda)\simeq \hasse(\bLambda). \notag
\end{align}
\end{itemize}
\end{theorem}

We illustrate Theorem \ref{3-4} with the following example.
\begin{example}
Let $\Lambda$ be the preprojective algebra of Dynkin type $A_{3}$.
Then we have $\Lambda=\begin{smallmatrix}
1\\
2\\
3
\end{smallmatrix}\oplus\begin{smallmatrix}
2\\
1\ 3\\
2
\end{smallmatrix}\oplus\begin{smallmatrix}
3\\
2\\
1
\end{smallmatrix}$.
Let $Q:=\begin{smallmatrix}
1\\
2\\
3
\end{smallmatrix}$.
Then we have $\bLambda=\begin{smallmatrix}
1\\
2
\end{smallmatrix}\oplus\begin{smallmatrix}
2\\
1\ 3\\
2
\end{smallmatrix}\oplus\begin{smallmatrix}
3\\
2\\
1
\end{smallmatrix}$.
The Hasse quiver $\hasse(\bLambda)$ is given by
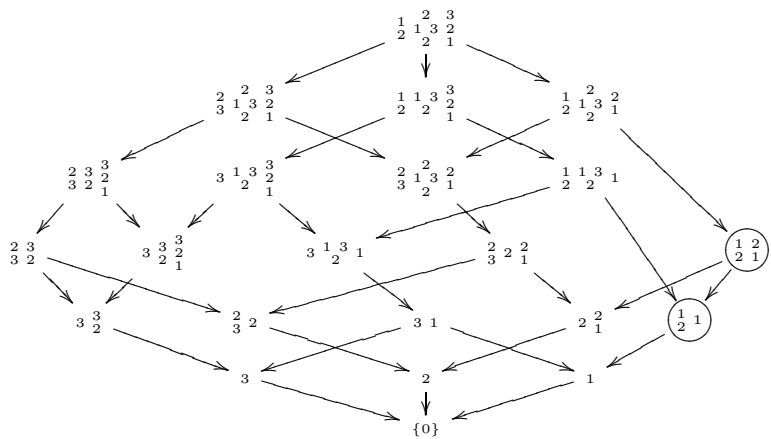
\begin{align}
\tiny{
\xymatrix@C=1mm@R=3mm{
&&&&&{\begin{smallmatrix}
1\\
2
\end{smallmatrix}\begin{smallmatrix}
2\\
1\ 3\\
2
\end{smallmatrix}\begin{smallmatrix}
3\\
2\\
1
\end{smallmatrix}}\ar[dll]\ar[d]\ar[drr]&&&&&\\
&&&{\begin{smallmatrix}
2\\
3
\end{smallmatrix}\begin{smallmatrix}
2\\
1\ 3\\
2
\end{smallmatrix}\begin{smallmatrix}
3\\
2\\
1
\end{smallmatrix}}\ar[dll]\ar[drr]&&{\begin{smallmatrix}
1\\
2\\
\end{smallmatrix}\begin{smallmatrix}
1\ 3\\
2
\end{smallmatrix}\begin{smallmatrix}
3\\
2\\
1
\end{smallmatrix}}\ar[dll]\ar[drr]&&{\begin{smallmatrix}
1\\
2\\
\end{smallmatrix}\begin{smallmatrix}
2\\
1\ 3\\
2
\end{smallmatrix}\begin{smallmatrix}
2\\
1
\end{smallmatrix}}\ar[dll]\ar[ddrrr]&&&\\
&{\begin{smallmatrix}
2\\
3
\end{smallmatrix}\begin{smallmatrix}
3\\
2
\end{smallmatrix}\begin{smallmatrix}
3\\
2\\
1
\end{smallmatrix}}\ar[dl]\ar[dr]&&{\begin{smallmatrix}
3
\end{smallmatrix}\begin{smallmatrix}
1\ 3\\
2
\end{smallmatrix}\begin{smallmatrix}
3\\
2\\
1
\end{smallmatrix}}\ar[dl]\ar[dr]&&{\begin{smallmatrix}
2\\
3
\end{smallmatrix}\begin{smallmatrix}
2\\
1\ 3\\
2
\end{smallmatrix}\begin{smallmatrix}
2\\
1
\end{smallmatrix}}\ar[dr]&&{\begin{smallmatrix}
1\\
2\\
\end{smallmatrix}\begin{smallmatrix}
1\ 3\\
2
\end{smallmatrix}\begin{smallmatrix}
1
\end{smallmatrix}}\ar[dlll]\ar[ddrr]&&&\\
{\begin{smallmatrix}
2\\
3
\end{smallmatrix}\begin{smallmatrix}
3\\
2
\end{smallmatrix}}\ar[dr]\ar[drrr]&&{\begin{smallmatrix}
3
\end{smallmatrix}\begin{smallmatrix}
3\\
2
\end{smallmatrix}\begin{smallmatrix}
3\\
2\\
1
\end{smallmatrix}}\ar[dl]&&{\begin{smallmatrix}
3
\end{smallmatrix}\begin{smallmatrix}
1\ 3\\
2
\end{smallmatrix}\begin{smallmatrix}
1
\end{smallmatrix}}\ar[dr]&&{\begin{smallmatrix}
2\\
3
\end{smallmatrix}\begin{smallmatrix}
2
\end{smallmatrix}\begin{smallmatrix}
2\\
1
\end{smallmatrix}}\ar[dlll]\ar[dr]&&&&*+[Fo]{\begin{smallmatrix}
1\\
2
\end{smallmatrix}\begin{smallmatrix}
2\\
1
\end{smallmatrix}}\ar[dlll]\ar[dl]\\
&{\begin{smallmatrix}
3
\end{smallmatrix}\begin{smallmatrix}
3\\
2
\end{smallmatrix}}\ar[drr]&&{\begin{smallmatrix}
2\\
3
\end{smallmatrix}\begin{smallmatrix}
2
\end{smallmatrix}}\ar[drr]&&{\begin{smallmatrix}
3
\end{smallmatrix}\begin{smallmatrix}
1
\end{smallmatrix}}\ar[dll]\ar[drr]&&{\begin{smallmatrix}
2
\end{smallmatrix}\begin{smallmatrix}
2\\
1
\end{smallmatrix}}\ar[dll]&&*+[Fo]{\begin{smallmatrix}
1\\
2
\end{smallmatrix}\begin{smallmatrix}
1
\end{smallmatrix}}\ar[dll]&\\
&&&{\begin{smallmatrix}
3
\end{smallmatrix}}\ar[drr]&&{\begin{smallmatrix}
2
\end{smallmatrix}}\ar[d]&&{\begin{smallmatrix}
1
\end{smallmatrix}}\ar[dll]&&&\\
&&&&&{\begin{smallmatrix}
\{0\}
\end{smallmatrix}}&&&&&
}}\notag
\end{align}
where the elements in $\N$ are marked by circles.

On the other hand, the Hasse quiver $\hasse(\Lambda)$ is given by
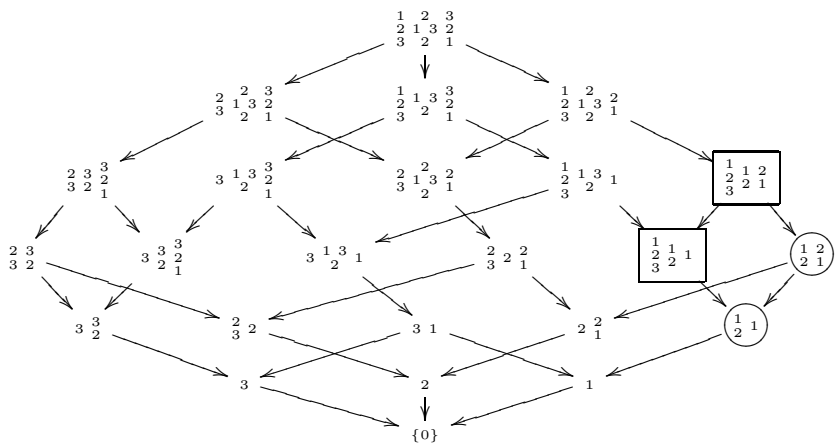
\begin{align}
\tiny{
\xymatrix@C=1mm@R=3mm{
&&&&&{\begin{smallmatrix}
1\\
2\\
3
\end{smallmatrix}\begin{smallmatrix}
2\\
1\ 3\\
2
\end{smallmatrix}\begin{smallmatrix}
3\\
2\\
1
\end{smallmatrix}}\ar[dll]\ar[d]\ar[drr]&&&&&\\
&&&{\begin{smallmatrix}
2\\
3
\end{smallmatrix}\begin{smallmatrix}
2\\
1\ 3\\
2
\end{smallmatrix}\begin{smallmatrix}
3\\
2\\
1
\end{smallmatrix}}\ar[dll]\ar[drr]&&{\begin{smallmatrix}
1\\
2\\
3
\end{smallmatrix}\begin{smallmatrix}
1\ 3\\
2
\end{smallmatrix}\begin{smallmatrix}
3\\
2\\
1
\end{smallmatrix}}\ar[dll]\ar[drr]&&{\begin{smallmatrix}
1\\
2\\
3
\end{smallmatrix}\begin{smallmatrix}
2\\
1\ 3\\
2
\end{smallmatrix}\begin{smallmatrix}
2\\
1
\end{smallmatrix}}\ar[dll]\ar[drr]&&&\\
&{\begin{smallmatrix}
2\\
3
\end{smallmatrix}\begin{smallmatrix}
3\\
2
\end{smallmatrix}\begin{smallmatrix}
3\\
2\\
1
\end{smallmatrix}}\ar[dl]\ar[dr]&&{\begin{smallmatrix}
3
\end{smallmatrix}\begin{smallmatrix}
1\ 3\\
2
\end{smallmatrix}\begin{smallmatrix}
3\\
2\\
1
\end{smallmatrix}}\ar[dl]\ar[dr]&&{\begin{smallmatrix}
2\\
3
\end{smallmatrix}\begin{smallmatrix}
2\\
1\ 3\\
2
\end{smallmatrix}\begin{smallmatrix}
2\\
1
\end{smallmatrix}}\ar[dr]&&{\begin{smallmatrix}
1\\
2\\
3
\end{smallmatrix}\begin{smallmatrix}
1\ 3\\
2
\end{smallmatrix}\begin{smallmatrix}
1
\end{smallmatrix}}\ar[dlll]\ar[dr]&&*+[F]{\begin{smallmatrix}
1\\
2\\
3
\end{smallmatrix}\begin{smallmatrix}
1\\
2
\end{smallmatrix}\begin{smallmatrix}
2\\
1
\end{smallmatrix}}\ar[dl]\ar[dr]&\\
{\begin{smallmatrix}
2\\
3
\end{smallmatrix}\begin{smallmatrix}
3\\
2
\end{smallmatrix}}\ar[dr]\ar[drrr]&&{\begin{smallmatrix}
3
\end{smallmatrix}\begin{smallmatrix}
3\\
2
\end{smallmatrix}\begin{smallmatrix}
3\\
2\\
1
\end{smallmatrix}}\ar[dl]&&{\begin{smallmatrix}
3
\end{smallmatrix}\begin{smallmatrix}
1\ 3\\
2
\end{smallmatrix}\begin{smallmatrix}
1
\end{smallmatrix}}\ar[dr]&&{\begin{smallmatrix}
2\\
3
\end{smallmatrix}\begin{smallmatrix}
2
\end{smallmatrix}\begin{smallmatrix}
2\\
1
\end{smallmatrix}}\ar[dlll]\ar[dr]&&*+[F]{\begin{smallmatrix}
1\\
2\\
3
\end{smallmatrix}\begin{smallmatrix}
1\\
2
\end{smallmatrix}\begin{smallmatrix}
1
\end{smallmatrix}}\ar[dr]&&*+[Fo]{\begin{smallmatrix}
1\\
2
\end{smallmatrix}\begin{smallmatrix}
2\\
1
\end{smallmatrix}}\ar[dlll]\ar[dl]\\
&{\begin{smallmatrix}
3
\end{smallmatrix}\begin{smallmatrix}
3\\
2
\end{smallmatrix}}\ar[drr]&&{\begin{smallmatrix}
2\\
3
\end{smallmatrix}\begin{smallmatrix}
2
\end{smallmatrix}}\ar[drr]&&{\begin{smallmatrix}
3
\end{smallmatrix}\begin{smallmatrix}
1
\end{smallmatrix}}\ar[dll]\ar[drr]&&{\begin{smallmatrix}
2
\end{smallmatrix}\begin{smallmatrix}
2\\
1
\end{smallmatrix}}\ar[dll]&&*+[Fo]{\begin{smallmatrix}
1\\
2
\end{smallmatrix}\begin{smallmatrix}
1
\end{smallmatrix}}\ar[dll]&\\
&&&{\begin{smallmatrix}
3
\end{smallmatrix}}\ar[drr]&&{\begin{smallmatrix}
2
\end{smallmatrix}}\ar[d]&&{\begin{smallmatrix}
1
\end{smallmatrix}}\ar[dll]&&&\\
&&&&&{\begin{smallmatrix}
\{0\}
\end{smallmatrix}}&&&&&
}}\notag
\end{align}
where the elements in $\N^{+}$ are marked by rectangles.
One can easily check that the relationship between $\sttilt\Lambda$ and $\sttilt\bLambda$ is given by Theorem \ref{3-4}.

We refer to \cite{Mi2} for more information on $\hasse(\Lambda)$ for a preprojective algebra $\Lambda$ of Dynkin type.
\end{example}
In the rest of this subsection, we give a proof of Theorem \ref{3-4}.
We will describe the relationship between $\sttilt\Lambda$ and $\sttilt\bLambda$ more explicitly.
We decompose $\sttilt\Lambda$ as $\sttilt\Lambda=\M_{1}\coprod\M_{2}^{-}\coprod\M_{2}^{+}\coprod\M_{3}$, where
\begin{align}
&\M_{1}=\M_{1}(Q):=\{ M\in\sttilt\Lambda\mid Q,\bQ\notin\add M\}, \notag\\
&\M_{2}^{-}=\M_{2}^{-}(Q):=\{ M\in\sttilt\Lambda\mid Q\notin\add M\ \textnormal{and}\ \bQ\in\add M\}, \notag\\
&\M_{2}^{+}=\M_{2}^{+}(Q):=\{ M\in\sttilt\Lambda\mid Q\oplus \bQ \in \add M\}, \notag \\
&\M_{3}=\M_{3}(Q):=\{ M\in\sttilt\Lambda\mid Q\in\add M\ \textnormal{and}\ \bQ\notin \add M\}. \notag
\end{align}
\begin{theorem}\label{3-5}
Let $\Lambda$ be a basic finite dimensional algebra and $Q$ an indecomposable projective-injective summand of $\Lambda$ as a $\Lambda$-module.
Then the following hold.
\begin{itemize}
\item[(1)] The map $M\mapsto \alpha(M)$ gives bijections
\begin{align}
\M_{1}\rightarrow \N_{1},\ \M_{2}^{-}\rightarrow \N_{2},\ \M_{2}^{+}\rightarrow\N_{2}^{+},\ \M_{3}\rightarrow \N_{3}, \notag
\end{align}
where 
\begin{align}
&\N_{1}=\N_{1}(Q):=\{ N\in\sttilt\bLambda\ |\ \bQ\notin\add N \}, \notag\\
&\N_{2}=\N_{2}(Q):=\N=\{ N\in\sttilt\bLambda\ |\ \bQ\in\add N\ \textnormal{and}\ \Hom_{\Lambda}(N,Q)=0 \}, \notag\\
&\N_{3}=\N_{3}(Q):=\{ N\in\sttilt\bLambda\ |\ \bQ\in\add N\ \textnormal{and}\ \Hom_{\Lambda}(N,Q)\neq 0 \}. \notag
\end{align}
and $\N_{2}^{+}$ is a copy of $\N_{2}$.
In particular, there is a bijection
\begin{align}
\alpha: \sttilt\Lambda\rightarrow (\sttilt\bLambda)^{\N_{2}} \notag
\end{align}
given by the bijection above.
\item[(2)] We have
\begin{align}
\sttilt\Lambda=\{ N\mid N\in \N_{1}\coprod \N_{2} \} \coprod \{ Q\oplus N \mid N\in\N_{2} \} \coprod \{ Q\oplus (N/\bQ)\mid N\in\N_{3}\}. \notag
\end{align}
\end{itemize}
\end{theorem}
\begin{proof}
We only have to give a proof of (1) because (2) follows from (1) immediately.
We start with an easy lemma.
\begin{lemma}\label{3-6}
Assume $\bQ\neq 0$ and that $U\in\mod\bLambda$ does not have $\bQ$ as a direct summand.
The following are equivalent.
\begin{itemize}
\item[(a)] $\bQ\oplus U$ is a $\tau$-rigid $\Lambda$-module.
\item[(b)] $\bQ\oplus U$ is a $\tau$-rigid $\bLambda$-module with $\Hom_{\Lambda}(\bQ\oplus U,Q)=0$.
\end{itemize}
\end{lemma}
\begin{proof}
We have isomorphisms
\begin{align}
\Hom_{\Lambda}(\bQ\oplus U,\tau_{\Lambda}(\bQ\oplus U))
&\overset{\ref{3-1}(3)}{\simeq} \Hom_{\Lambda}(\bQ\oplus U,\rad Q) \oplus \Hom_{\bLambda}(\bQ\oplus U,\tau_{\bLambda}U)\notag\\
&\overset{\ref{1-1}(2)}{\simeq} \Hom_{\Lambda}(\bQ\oplus U,Q) \oplus \Hom_{\bLambda}(\bQ\oplus U,\tau_{\bLambda}(\bQ\oplus U)).\notag
\end{align}
Thus the assertion follows.
\end{proof}
The following proposition plays an important role.
\begin{proposition}\label{3-7}
Assume $\bQ\neq 0$ and that $U\in\mod\bLambda$ does not have $\bQ$ as a direct summand.
\begin{itemize}
\item[(1)] The following are equivalent:
\begin{itemize}
\item[(a)] $U$ is a support $\tau$-tilting $\Lambda$-module.
\item[(b)] $U$ is a support $\tau$-tilting $\bLambda$-module.
\end{itemize}
\item[(2)] The following are equivalent:
\begin{itemize}
\item[(a)] $Q\oplus \bQ\oplus U$ is a support $\tau$-tilting $\Lambda$-module.
\item[(b)] $\bQ\oplus U$ is a support $\tau$-tilting $\Lambda$-module.
\item[(c)] $\bQ\oplus U$ is a support $\tau$-tilting $\bLambda$-module with $\Hom_{\Lambda}(\bQ\oplus U,Q)=0$.
\end{itemize}
\item[(3)] The following are equivalent:
\begin{itemize}
\item[(a)] $Q\oplus U$ is a support $\tau$-tilting $\Lambda$-module.
\item[(b)] $\bQ\oplus U$ is a support $\tau$-tilting $\bLambda$-module with $\Hom_{\Lambda}(\bQ\oplus U,Q)\neq 0$.
\end{itemize}
\end{itemize}
\end{proposition}
\begin{proof}
(1) Since we have $\tau_{\Lambda}U\simeq\tau_{\bLambda}U$ by Proposition \ref{3-1}(3), the assertion follows.

(2) (a) $\Leftrightarrow$ (b): We claim that $M:=Q\oplus \bQ\oplus U$ is a $\tau$-rigid $\Lambda$-module if and only if 
$N:=\bQ\oplus U$ is a $\tau$-rigid $\Lambda$-module.
Indeed, this follows from isomorphisms
\begin{align}
\Hom_{\Lambda}(M,\tau_{\Lambda}M)
&\overset{\ \ \ \ \ }{\simeq}\Hom_{\Lambda}(Q,\tau_{\Lambda}N)\oplus \Hom_{\Lambda}(N,\tau_{\Lambda}N)\notag\\
&\overset{\ref{1-1}(2)}{\simeq}\Hom_{\Lambda}(\bQ, \tau_{\Lambda}N)\oplus \Hom_{\Lambda}(N,\tau_{\Lambda}N), \notag
\end{align}
and $\bQ\in \add N$.

By Lemma \ref{3-6}, both (a) and (b) imply $\Hom_{\Lambda}(N,Q)=0$, 
or equivalently $N$ does not have $S$ as a composition factor by Lemma \ref{1-1}(1). 
Hence $\s(M)=\s(N)+1$ holds, where $\s(M)$ is the number of nonisomorphic simple modules appearing in a composition series of $M$..
Thus the assertion follows from Proposition \ref{1-3} and $|M|=|N|+1$.

(b) $\Leftrightarrow$ (c): This is immediate from Lemma \ref{3-6}.

(3) First, we claim that $Q\oplus U$ is a $\tau$-rigid $\Lambda$-module if and only if $\bQ\oplus U$ is a $\tau$-rigid $\bLambda$-module.
Indeed, this follows from isomorphisms
\begin{align}
\Hom_{\Lambda}(Q\oplus U, \tau_{\Lambda}U)
\overset{\ref{1-1}(2)}{\simeq}\Hom_{\Lambda}(\bQ\oplus U,\tau_{\Lambda} U)
\overset{\ref{3-1}(3)}{\simeq}\Hom_{\bLambda}(\bQ\oplus U,\tau_{\bLambda}U).\notag
\end{align}
Next, we claim that, if $Q\oplus U\in\sttilt\Lambda$, then $\Hom_{\Lambda}(\bQ\oplus U,Q)\neq 0$.
Indeed, assume that $\Hom_{\Lambda}(\bQ\oplus U,Q) =0$.
Since $Q\oplus U$ is $\tau$-rigid, $\bQ\oplus U$ is a $\tau$-rigid $\bLambda$-module by Lemma \ref{1-1}(2) and Proposition \ref{3-1}(3), 
and hence $\bQ\oplus U$ is a $\tau$-rigid $\Lambda$-module by Lemma \ref{3-6}.
Thus $Q\oplus \bQ\oplus U$ is a $\tau$-rigid $(\Lambda/\langle e_{Q\oplus U})\rangle)$-module by Lemma \ref{air2.1}.
This contradicts that $\tau$-tilting modules are maximal $\tau$-rigid by Proposition \ref{air2.10}.
Since $\Hom_{\Lambda}(\bQ\oplus U,Q)\neq 0$ holds by our assumption, 
$\bQ\oplus U$ has $S$ as a composition factor by Lemma \ref{1-1}(1).
Thus the assertion follows from Proposition \ref{1-3} since we have $\s(Q\oplus U)=\s(\bQ\oplus U)$.
\end{proof}
\begin{remark}
If $Q$ is a simple projective-injective $\Lambda$-module, or equivalently $\bQ=0$, then we have $\sttilt\bLambda=\N_{2}$.
In this case, Proposition \ref{3-7}(2) holds because we have $\bLambda\simeq \Lambda/\langle e_{i}\rangle$, where $\soc(Q)=S_{i}=\top(e_{i}\Lambda)$.
\end{remark}

We now finish the proof of Theorem \ref{3-5}(1).

By Proposition \ref{3-7}, the map $M\mapsto\alpha(M)$ gives bijections
\begin{align}
\M_{1}\rightarrow \N_{1},\ \M_{2}^{-}\rightarrow \N_{2},\ \M_{2}^{+}\rightarrow \N_{2}^{+},\ \M_{3}\rightarrow \N_{3}. \notag
\end{align}
Hence $\alpha:\sttilt\Lambda\rightarrow (\sttilt\bLambda)^{\N}$ is a bijection.
\end{proof}
The following lemma is useful to understand the structure of the poset $\sttilt\Lambda$.
\begin{lemma}\label{3-8}
\begin{itemize}
\item[(1)] Let $N,N'\in\sttilt\bLambda$. If $N>N'$ holds, then we have 
\begin{align}
(N,N')\notin (\N_{1}\times \N_{2})\coprod(\N_{1}\times \N_{3})\coprod(\N_{2}\times\N_{3}). \notag
\end{align}
\item[(2)] Let $M,M'\in\sttilt\Lambda$. If $M>M'$ holds, then we have 
\begin{align}
(M,M')\notin (\M_{1}\times \M_{2}^{\pm})\coprod(\M_{1}\times \M_{3})\coprod(\M_{2}^{\pm}\times\M_{3})\coprod(\M_{2}^{-}\times\M_{2}^{+}) \notag.
\end{align}
\end{itemize}
\end{lemma}
\begin{proof}
We only prove (1); the proof of (2) is similar.
If $N\ge N'$ holds and $P\in\add N'$ is projective, then $P\in \add N$.
Thus the assertion except $(N,N')\notin\N_{2}\times \N_{3}$ follows.
Assume that $(N,N')\in \N_{2}\times \N_{3}$.
Then we have $N'\in\Fac(N)$.
This is a contradiction since $N$ does not have $S$ as a composition factor and $N'$ has $S$ as a composition factor.
\end{proof}
Now, we are ready to prove Theorem \ref{3-4}.
\begin{proof}[Proof of Theorem \ref{3-4}]
(1) Since the map $\alpha:\sttilt\Lambda\rightarrow (\sttilt\bLambda)^{\N}$ is a bijection by Theorem \ref{3-5},
we have only to show that, for any $M,L\in\sttilt\Lambda$, $M\ge L$ holds if and only if $\alpha(M)\ge \alpha(L)$ holds.
Indeed, if $M\ge L$, then 
$L\in \Fac M$, and hence $\overline{L}\in \Fac \overline{M}$ hold, which implies $\alpha(M)\ge \alpha(L)$.
Conversely, assume that $\alpha(M)\ge \alpha(L)$.
If both $M$ and $L$ are in either $\M_{1}$, $\M^{\pm}_{2}$ or $\M_{3}$, then $M\ge L$ holds clearly.
Otherwise, by Lemma \ref{3-8}, we have
\begin{align}
(M,L)\in (\M_{3}\times \M_{2}^{\pm})\coprod (\M_{3}\times\M_{1})\coprod (\M_{2}^{\pm}\times \M_{1})\coprod (\M_{2}^{+},\M_{2}^{-}). \notag
\end{align}
By the definition of the functor $\overline{(-)}$, we have $\overline{M}\in \Fac M$ in $\mod \Lambda$.
Hence $M\ge \alpha(M)\ge \alpha(L)$ holds in $\mod\Lambda$.
If we have $L\in \M_{1}\coprod\M_{2}^{-}$, then $\alpha(L)=\overline{L}=L$ and hence $M\ge L$ hold.
On the other hand, if we have $L\in \M_{2}^{+}$, then $L=Q\oplus \overline{L}\in \Fac M$ holds because $M\in \M_{3}$ has $Q$ as a direct summand.
Hence the assertion follows.

(2) Since $\bQ$ has $S$ as a composition factor, we have $\Hom_{\Lambda}(\bQ, Q)\neq 0$ by Lemma \ref{1-1}(1).
Thus $\N=\emptyset$ holds.
The assertion follows from (1) since $(\sttilt\bLambda)^{\emptyset}=\sttilt\bLambda$.
\end{proof}
\subsection{Applications to Nakayama algebras}
In this subsection, we apply the results in the previous subsection to a Nakayama algebra $\Lambda$.
We give a combinatorial method to construct the Hasse quiver of support $\tau$-tilting $\Lambda$-modules.

It is basic that any Nakayama algebra $\Lambda$ has an at least one indecomposable projective-injective module $Q$
and its factor algebra $\Lambda/\soc Q$ is again Nakayama (see \cite[V.3.3 and V.3.4]{ASS}). 
Thus we can iteratively apply  Drozd-Kirichenko rejection to Nakayama algebras. 
We have the following algorithm for a construction of the Hasse quiver.
\begin{algorithm}\label{3-9}
Let $\Lambda$ be a Nakayama algebra.
\begin{itemize}
\item[(0)] Take a (non-unique) sequence of Nakayama algebras
\begin{align}\label{seq}
\xymatrix{\Lambda=:\Lambda_{1}\ar@{->>}[r]&\Lambda_{2}\ar@{->>}[r]&\cdots\ar@{->>}[r]&\Lambda_{m-1}\ar[r]&\Lambda_{m}:=0}\notag
\end{align}
such that $\Lambda_{i+1}:=\Lambda_{i}/\soc Q_{i}$, 
where $Q_{i}$ is an indecomposable projective-injective $\Lambda_{i}$-module, and $m> 0$ is an integer.
\item[(1)] First, the Hasse quiver $\hasse(\Lambda_{m})$ consists of one vertex and no arrows.
\item[(2)] Secondly, the Hasse quiver $\hasse(\Lambda_{m-1})$ is given by $\hasse(\Lambda_{m})^{\N_{2}(Q_{m-1})}$ (see Definition \ref{3-2}).
\item[(3)] Thirdly, the Hasse quiver $\hasse(\Lambda_{m-2})$ is given by $\hasse(\Lambda_{m-1})^{\N_{2}(Q_{m-2})}$.
\item[$\vdots$]
\item[(l)] The Hasse quiver $\hasse(\Lambda_{m-l+1})$ is given by $\hasse(\Lambda_{m-l+2})^{\N_{2}(Q_{m-l+1})}$.
\item[$\vdots$]
\item[(m)] Finally, the Hasse quiver $\hasse(\Lambda_{1})$ is given by $\hasse(\Lambda_{2})^{\N_{2}(Q_{1})}$.
\end{itemize}
\end{algorithm}

As a direct consequence of Theorem \ref{3-4}, we have the following result.
\begin{theorem}\label{3-10}
Let $\Lambda$ be a Nakayama algebra with $n$ simple modules. 
Assume that $\ell(P_{i})\ge n$ for all $i\in [1,n]$.
Then we have an isomorphism of posets
\begin{align}
\sttilt \Lambda\simeq \sttilt \Lambda_{n}^{n}, \notag
\end{align}
where $\Lambda_{n}^{n}$ is the self-injective Nakayama algebra with $n$ simple modules and the Loewy length $n$.
\end{theorem}
We give an example of calculation of the Hasse quiver using Algorithm \ref{3-9}.
\begin{example}\label{3-3-2}
Let $\Lambda_{1}:=\Lambda_{3}^{4}=K\vec{\Delta_{3}}/J^{4}$.
Then we have  sequence 
\begin{align}
\xymatrix{\Lambda_{1}\ar@{->>}[r]&\Lambda_{2}\ar@{->>}[r]&\Lambda_{3}\ar@{->>}[r]&\cdots \ar@{->>}[r]
&\Lambda_{8}\ar@{->>}[r]&\Lambda_{9}\ar@{->>}[r]&\Lambda_{10}=:K^{3}}.\notag
\end{align}
of Nakayama algebras, where each $\Lambda_{i}$ is explicitly given bellow.
We will describe $\hasse(\Lambda_{i})$ inductively,
where the elements in $\N_{2}(Q_{i})$ are marked by rectangles and those in $\M_{2}(Q_{i})$ are marked by circles.
Note that $\hasse(K^{m})$ is isomorphic to the Hasse quiver of the set of all subsets of an $m$-element set ordered by inclusion.
Thus we may begin with a semisimple algebra.

\begin{picture}(410,158)(0,0)
\put(0,145){(1) $\hasse(\Lambda_{10})$ is given below, where $\Lambda_{10}=K^{3}$ and $Q_{9}=\tiny\begin{smallmatrix}
3\\
2
\end{smallmatrix}$:}
\put(0,130){(2) $\hasse(\Lambda_{9})$ is given below, where $\Lambda_{9}=K(\xymatrix{3\ar[r]&2&1})$ and $Q_{8}=\tiny\begin{smallmatrix}
2\\
1
\end{smallmatrix}$:}
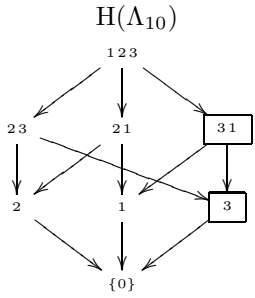
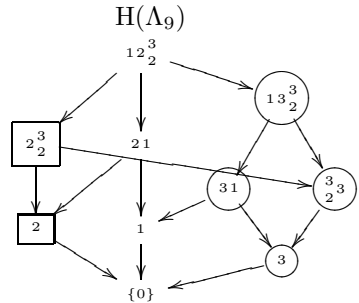
\put(45,113){{\small$\hasse(\Lambda_{10})$}}
\put(265,113){{\small$\hasse(\Lambda_{9})$}}
\put(0,100){$
\xymatrix @R=2mm @C=0mm{
 & & & & &{\tiny\begin{smallmatrix}
1
\end{smallmatrix}\hspace{-1mm}\begin{smallmatrix}
2
\end{smallmatrix}\hspace{-1mm}\begin{smallmatrix}
3
\end{smallmatrix}}\ar[ddllll]\ar[dd]\ar[ddrrrr]& & & & & \\
 & & & & & & & & & & \\
 &{\tiny\begin{smallmatrix}
2
\end{smallmatrix}\hspace{-1mm}\begin{smallmatrix}
3
\end{smallmatrix}}\ar[dd]\ar[ddrrrrrrrr]& & & &{\tiny\begin{smallmatrix}
2
\end{smallmatrix}\hspace{-1mm}\begin{smallmatrix}
1
\end{smallmatrix}}\ar[dd]\ar[ddllll]& & & &*+[F]{\tiny\begin{smallmatrix}
3
\end{smallmatrix}\hspace{-1mm}\begin{smallmatrix}
1
\end{smallmatrix}}\ar[dd]\ar[ddllll]& \\
 & & & & & & & & & & \\
 &{\tiny\begin{smallmatrix}
2
\end{smallmatrix}}\ar[ddrrrr]& & & &{\tiny\begin{smallmatrix}
1
\end{smallmatrix}}\ar[dd]& & & &*+[F]{\tiny\begin{smallmatrix}
3
\end{smallmatrix}}\ar[ddllll]& \\
 & & & & & & & & & & \\
 & & & & &{\tiny\begin{smallmatrix}
\{0\}
\end{smallmatrix}}& & & & & \\
}
$}
\put(220,100){$
\xymatrix @R=0.4mm @C=0mm{
 & & & & &{\tiny\begin{smallmatrix}
1
\end{smallmatrix}\hspace{-1mm}\begin{smallmatrix}
2
\end{smallmatrix}\hspace{-1mm}\begin{smallmatrix}
3\\
2
\end{smallmatrix}}\ar[ddllll]\ar[dd]\ar[drrrr]& & & & & \\
 & & & & & & & & &*+[Fo]{\tiny\begin{smallmatrix}
1
\end{smallmatrix}\hspace{-1mm}\begin{smallmatrix}
3
\end{smallmatrix}\hspace{-1mm}\begin{smallmatrix}
3\\
2
\end{smallmatrix}}\ar[ddl]\ar[ddr]& \\
 &*+[F]{\tiny\begin{smallmatrix}
2
\end{smallmatrix}\hspace{-1mm}\begin{smallmatrix}
3\\
2
\end{smallmatrix}}\ar[dd]\ar[drrrrrrrrr]& & & &{\tiny\begin{smallmatrix}
2
\end{smallmatrix}\hspace{-1mm}\begin{smallmatrix}
1
\end{smallmatrix}}\ar[dd]\ar[ddllll]& & & & & \\
 & & & & & & & &*+[Fo]{\tiny\begin{smallmatrix}
3
\end{smallmatrix}\hspace{-1mm}\begin{smallmatrix}
1
\end{smallmatrix}}\ar[ddr]\ar[dlll]& &*+[Fo]{\tiny\begin{smallmatrix}
3\\
2
\end{smallmatrix}\hspace{-1mm}\begin{smallmatrix}
3
\end{smallmatrix}}\ar[ddl]\\
 &*+[F]{\tiny\begin{smallmatrix}
2
\end{smallmatrix}}\ar[ddrrrr]& & & &{\tiny\begin{smallmatrix}
1
\end{smallmatrix}}\ar[dd]& & & & & \\
 & & & & & & & & &*+[Fo]{\tiny\begin{smallmatrix}
3
\end{smallmatrix}}\ar[dllll]& \\
 & & & & &{\tiny\begin{smallmatrix}
\{0\}
\end{smallmatrix}}& & & & & \\
}
$}
\end{picture}

\begin{picture}(410,202)(0,0)
\put(0,189){(3) $\hasse(\Lambda_{8})$ is given below, 
where $\Lambda_{8}=K\vec{A}_{3}/\langle \alpha_{2}\alpha_{1}\rangle$ and $Q_{7}=\tiny \begin{smallmatrix}
3\\
2\\
1
\end{smallmatrix}$:}
\put(0,174){(4) $\hasse(\Lambda_{7})$ is given below, where $\Lambda_{7}=K\vec{A}_{3}$ and $Q_{6}=\tiny \begin{smallmatrix}
1\\
3
\end{smallmatrix}$:}

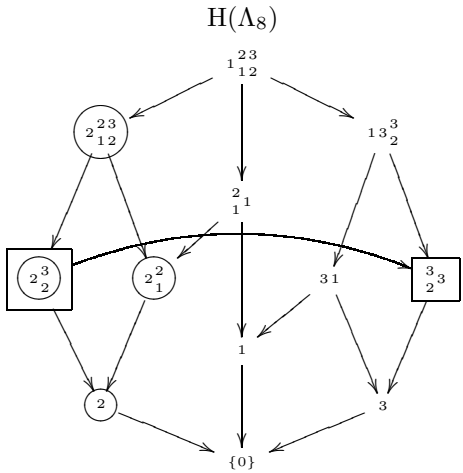
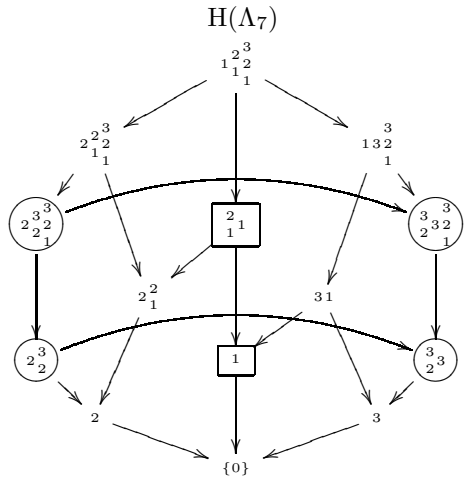
\put(75,157){{\small$\hasse(\Lambda_{8})$}}
\put(285,157){{\small$\hasse(\Lambda_{7})$}}
\put(0,140){$
\xymatrix @R=3.5mm @C=0mm{
 & & & & &{\tiny\begin{smallmatrix}
1
\end{smallmatrix}\hspace{-1mm}\begin{smallmatrix}
2\\
1
\end{smallmatrix}\hspace{-1mm}\begin{smallmatrix}
3\\
2
\end{smallmatrix}}\ar[dllll]\ar[dd]\ar[drrrr]& & & & & \\
 &*+[Fo]{\tiny\begin{smallmatrix}
2
\end{smallmatrix}\hspace{-1mm}\begin{smallmatrix}
2\\
1
\end{smallmatrix}\hspace{-1mm}\begin{smallmatrix}
3\\
2
\end{smallmatrix}}\ar[ddl]\ar[ddr]& & & & & & & &{\tiny\begin{smallmatrix}
1
\end{smallmatrix}\hspace{-1mm}\begin{smallmatrix}
3
\end{smallmatrix}\hspace{-1mm}\begin{smallmatrix}
3\\
2
\end{smallmatrix}}\ar[ddl]\ar[ddr]& \\
 & & & & &{\tiny\begin{smallmatrix}
2\\
1
\end{smallmatrix}\hspace{-1mm}\begin{smallmatrix}
1
\end{smallmatrix}}\ar[dlll]\ar[dd]& & & & & \\
*+[Fo]+[F]{\tiny\begin{smallmatrix}
2
\end{smallmatrix}\hspace{-1mm}\begin{smallmatrix}
3\\
2
\end{smallmatrix}}\ar[ddr]\ar@/^6mm/[rrrrrrrrrr]& &*+[Fo]{\tiny\begin{smallmatrix}
2
\end{smallmatrix}\hspace{-1mm}\begin{smallmatrix}
2\\
1
\end{smallmatrix}}\ar[ddl]& & & & & &{\tiny\begin{smallmatrix}
3
\end{smallmatrix}\hspace{-1mm}\begin{smallmatrix}
1
\end{smallmatrix}}\ar[ddr]\ar[dlll]& &*+[F]{\tiny\begin{smallmatrix}
3\\
2
\end{smallmatrix}\hspace{-1mm}\begin{smallmatrix}
3
\end{smallmatrix}}\ar[ddl]\\
 & & & & &{\tiny\begin{smallmatrix}
1
\end{smallmatrix}}\ar[dd]& & & & & \\
 &*+[Fo]{\tiny\begin{smallmatrix}
2
\end{smallmatrix}}\ar[drrrr]& & & & & & & &{\tiny\begin{smallmatrix}
3
\end{smallmatrix}}\ar[dllll]& \\
 & & & & &{\tiny\begin{smallmatrix}
\{0\}
\end{smallmatrix}}& & & & & \\
}
$}
\put(210,140){$
\xymatrix @R=2.8mm @C=0mm{
 & & & & &{\tiny\begin{smallmatrix}
1
\end{smallmatrix}\hspace{-1mm}\begin{smallmatrix}
2\\
1
\end{smallmatrix}\hspace{-1mm}\begin{smallmatrix}
3\\
2\\
1
\end{smallmatrix}}\ar[dllll]\ar[dd]\ar[drrrr]& & & & & \\
 &{\tiny\begin{smallmatrix}
2
\end{smallmatrix}\hspace{-1mm}\begin{smallmatrix}
2\\
1
\end{smallmatrix}\hspace{-1mm}\begin{smallmatrix}
3\\
2\\
1
\end{smallmatrix}}\ar[dl]\ar[ddr]& & & & & & & &{\tiny\begin{smallmatrix}
1
\end{smallmatrix}\hspace{-1mm}\begin{smallmatrix}
3
\end{smallmatrix}\hspace{-1mm}\begin{smallmatrix}
3\\
2\\
1
\end{smallmatrix}}\ar[ddl]\ar[dr]& \\
*+[Fo]{\tiny\begin{smallmatrix}
2
\end{smallmatrix}\hspace{-1mm}\begin{smallmatrix}
3\\
2
\end{smallmatrix}\hspace{-1mm}\begin{smallmatrix}
3\\
2\\
1
\end{smallmatrix}}\ar[dd]\ar@/^6mm/[rrrrrrrrrr]& & & & &*+[F]{\tiny\begin{smallmatrix}
2\\
1
\end{smallmatrix}\hspace{-1mm}\begin{smallmatrix}
1
\end{smallmatrix}}\ar[dlll]\ar[dd]& & & & &*+[Fo]{\tiny\begin{smallmatrix}
3\\
2
\end{smallmatrix}\hspace{-1mm}\begin{smallmatrix}
3
\end{smallmatrix}\hspace{-1mm}\begin{smallmatrix}
3\\
2\\
1
\end{smallmatrix}}\ar[dd] \\
 & &{\tiny\begin{smallmatrix}
2
\end{smallmatrix}\hspace{-1mm}\begin{smallmatrix}
2\\
1
\end{smallmatrix}}\ar[ddl]& & & & & &{\tiny\begin{smallmatrix}
3
\end{smallmatrix}\hspace{-1mm}\begin{smallmatrix}
1
\end{smallmatrix}}\ar[dlll]\ar[ddr]& & \\
*+[Fo]{\tiny\begin{smallmatrix}
2
\end{smallmatrix}\hspace{-1mm}\begin{smallmatrix}
3\\
2
\end{smallmatrix}}\ar[dr]\ar@/^6mm/[rrrrrrrrrr]& & & & &*+[F]{\tiny\begin{smallmatrix}
1
\end{smallmatrix}}\ar[dd]& & & & &*+[Fo]{\tiny\begin{smallmatrix}
3\\
2
\end{smallmatrix}\hspace{-1mm}\begin{smallmatrix}
3
\end{smallmatrix}}\ar[dl]\\
 &{\tiny\begin{smallmatrix}
2
\end{smallmatrix}}\ar[drrrr]& & & & & & & &{\tiny\begin{smallmatrix}
3
\end{smallmatrix}}\ar[dllll]& \\
 & & & & &{\tiny\begin{smallmatrix}
\{0\}
\end{smallmatrix}}& & & & & \\
}\notag
$}
\end{picture}

\begin{picture}(430,625)(0,0)
\put(0,620){(5) $\hasse(\Lambda_{6})$ is given below, 
where $\Lambda_{6}=K\vec{\Delta}_{3}/\langle \alpha_{3}\alpha_{2},\ \alpha_{1}\alpha_{3}\rangle$ and $Q_{5}=\tiny\begin{smallmatrix}
2\\
1\\
3
\end{smallmatrix}$:}
\put(0,605){(6) $\hasse(\Lambda_{5})$ is given below, 
where $\Lambda_{5}=K\vec{\Delta}_{3}/\langle \alpha_{3}\alpha_{2}, \alpha_{2}\alpha_{1}\alpha_{3}\rangle$ and $Q_{4}=\tiny \begin{smallmatrix}
1\\
3\\
2
\end{smallmatrix}$:}
\put(95,591){{\small$\hasse(\Lambda_{6})$}}
\put(305,591){{\small$\hasse(\Lambda_{5})$}}
\put(0,573){$
\xymatrix @R=2.7mm @C=0mm{
 & & & & &{\tiny\begin{smallmatrix}
1\\
3
\end{smallmatrix}\hspace{-1mm}\begin{smallmatrix}
2\\
1
\end{smallmatrix}\hspace{-1mm}\begin{smallmatrix}
3\\
2\\
1
\end{smallmatrix}}\ar[dllll]\ar[d]\ar[drrrr]& & & & & \\
 &{\tiny\begin{smallmatrix}
2
\end{smallmatrix}\hspace{-1mm}\begin{smallmatrix}
2\\
1
\end{smallmatrix}\hspace{-1mm}\begin{smallmatrix}
3\\
2\\
1
\end{smallmatrix}}\ar[dl]\ar[ddr]& & & &*+[Fo]{\tiny\begin{smallmatrix}
1\\
3
\end{smallmatrix}\hspace{-1mm}\begin{smallmatrix}
2\\
1
\end{smallmatrix}\hspace{-1mm}\begin{smallmatrix}
1
\end{smallmatrix}}\ar[ddl]\ar[ddr]& & & &{\tiny\begin{smallmatrix}
1\\
3
\end{smallmatrix}\hspace{-1mm}\begin{smallmatrix}
3
\end{smallmatrix}\hspace{-1mm}\begin{smallmatrix}
3\\
2\\
1
\end{smallmatrix}}\ar[ddl]\ar[dr]& \\
{\tiny\begin{smallmatrix}
2
\end{smallmatrix}\hspace{-1mm}\begin{smallmatrix}
3\\
2
\end{smallmatrix}\hspace{-1mm}\begin{smallmatrix}
3\\
2\\
1
\end{smallmatrix}}\ar[dd]\ar@/^6mm/[rrrrrrrrrr]& & & & & & & & & &{\tiny\begin{smallmatrix}
3\\
2
\end{smallmatrix}\hspace{-1mm}\begin{smallmatrix}
3
\end{smallmatrix}\hspace{-1mm}\begin{smallmatrix}
3\\
2\\
1
\end{smallmatrix}}\ar[dd] \\
 & &*+[F]{\tiny\begin{smallmatrix}
2
\end{smallmatrix}\hspace{-1mm}\begin{smallmatrix}
2\\
1
\end{smallmatrix}}\ar[ddl]& &*+[Fo]+[F]{\tiny\begin{smallmatrix}
2\\
1
\end{smallmatrix}\hspace{-1mm}\begin{smallmatrix}
1
\end{smallmatrix}}\ar[ll]\ar[ddr]& &*+[Fo]{\tiny\begin{smallmatrix}
1\\
3
\end{smallmatrix}\hspace{-1mm}\begin{smallmatrix}
1
\end{smallmatrix}}\ar[ddl]& &{\tiny\begin{smallmatrix}
3
\end{smallmatrix}\hspace{-1mm}\begin{smallmatrix}
1\\
3
\end{smallmatrix}}\ar[ll]\ar[ddr]& & \\
{\tiny\begin{smallmatrix}
2
\end{smallmatrix}\hspace{-1mm}\begin{smallmatrix}
3\\
2
\end{smallmatrix}}\ar[dr]\ar@/^6mm/[rrrrrrrrrr]& & & & & & & & & &{\tiny\begin{smallmatrix}
3\\
2
\end{smallmatrix}\hspace{-1mm}\begin{smallmatrix}
3
\end{smallmatrix}}\ar[dl]\\
 &{\tiny\begin{smallmatrix}
2
\end{smallmatrix}}\ar[drrrr]& & & &*+[Fo]{\tiny\begin{smallmatrix}
1
\end{smallmatrix}}\ar[d]& & & &{\tiny\begin{smallmatrix}
3
\end{smallmatrix}}\ar[dllll]& \\
 & & & & &{\tiny\begin{smallmatrix}
\{0\}
\end{smallmatrix}}& & & & & \\
}
$}
\put(210,573){$
\xymatrix @R=3.2mm @C=0mm{
 & & & & &{\tiny\begin{smallmatrix}
1\\
3
\end{smallmatrix}\hspace{-1mm}\begin{smallmatrix}
2\\
1\\
3
\end{smallmatrix}\hspace{-1mm}\begin{smallmatrix}
3\\
2\\
1
\end{smallmatrix}}\ar[dllll]\ar[d]\ar[drrrr]& & & & & \\
 &{\tiny\begin{smallmatrix}
2
\end{smallmatrix}\hspace{-1mm}\begin{smallmatrix}
2\\
1\\
3
\end{smallmatrix}\hspace{-1mm}\begin{smallmatrix}
3\\
2\\
1
\end{smallmatrix}}\ar[dl]\ar[dr]& & & &{\tiny\begin{smallmatrix}
1\\
3
\end{smallmatrix}\hspace{-1mm}\begin{smallmatrix}
2\\
1\\
3
\end{smallmatrix}\hspace{-1mm}\begin{smallmatrix}
1
\end{smallmatrix}}\ar[dl]\ar[ddr]& & & &{\tiny\begin{smallmatrix}
1\\
3
\end{smallmatrix}\hspace{-1mm}\begin{smallmatrix}
3
\end{smallmatrix}\hspace{-1mm}\begin{smallmatrix}
3\\
2\\
1
\end{smallmatrix}}\ar[ddl]\ar[dr]& \\
{\tiny\begin{smallmatrix}
2
\end{smallmatrix}\hspace{-1mm}\begin{smallmatrix}
3\\
2
\end{smallmatrix}\hspace{-1mm}\begin{smallmatrix}
3\\
2\\
1
\end{smallmatrix}}\ar[dd]\ar@/^6mm/[rrrrrrrrrr]& &*+[Fo]{\tiny\begin{smallmatrix}
2
\end{smallmatrix}\hspace{-1mm}\begin{smallmatrix}
2\\
1\\
3
\end{smallmatrix}\hspace{-1mm}\begin{smallmatrix}
2\\
1
\end{smallmatrix}}\ar[dd]& &*+[Fo]{\tiny\begin{smallmatrix}
2\\
1
\end{smallmatrix}\hspace{-1mm}\begin{smallmatrix}
2\\
1\\
3
\end{smallmatrix}\hspace{-1mm}\begin{smallmatrix}
1
\end{smallmatrix}}\ar[dd]\ar[ll]& & & & & &{\tiny\begin{smallmatrix}
3\\
2
\end{smallmatrix}\hspace{-1mm}\begin{smallmatrix}
3
\end{smallmatrix}\hspace{-1mm}\begin{smallmatrix}
3\\
2\\
1
\end{smallmatrix}}\ar[dd] \\
 & & & & & &*+[F]{\tiny\begin{smallmatrix}
1\\
3
\end{smallmatrix}\hspace{-1mm}\begin{smallmatrix}
1
\end{smallmatrix}}\ar[ddl]& &*+[F]{\tiny\begin{smallmatrix}
3
\end{smallmatrix}\hspace{-1mm}\begin{smallmatrix}
1\\
3
\end{smallmatrix}}\ar[ll]\ar[ddr]& & \\
{\tiny\begin{smallmatrix}
2
\end{smallmatrix}\hspace{-1mm}\begin{smallmatrix}
3\\
2
\end{smallmatrix}}\ar[dr]\ar@/^6mm/[rrrrrrrrrr]& &*+[Fo]{\tiny\begin{smallmatrix}
2
\end{smallmatrix}\hspace{-1mm}\begin{smallmatrix}
2\\
1
\end{smallmatrix}}\ar[dl]& &*+[Fo]{\tiny\begin{smallmatrix}
2\\
1
\end{smallmatrix}\hspace{-1mm}\begin{smallmatrix}
1
\end{smallmatrix}}\ar[ll]\ar[dr]& & & & & &{\tiny\begin{smallmatrix}
3\\
2
\end{smallmatrix}\hspace{-1mm}\begin{smallmatrix}
3
\end{smallmatrix}}\ar[dl] \\
 &{\tiny\begin{smallmatrix}
2
\end{smallmatrix}}\ar[drrrr]& & & &{\tiny\begin{smallmatrix}
1
\end{smallmatrix}}\ar[d]& & & &{\tiny\begin{smallmatrix}
3
\end{smallmatrix}}\ar[dllll]& \\
 & & & & &{\tiny\begin{smallmatrix}
\{0\}
\end{smallmatrix}}& & & & & \\
}
$}

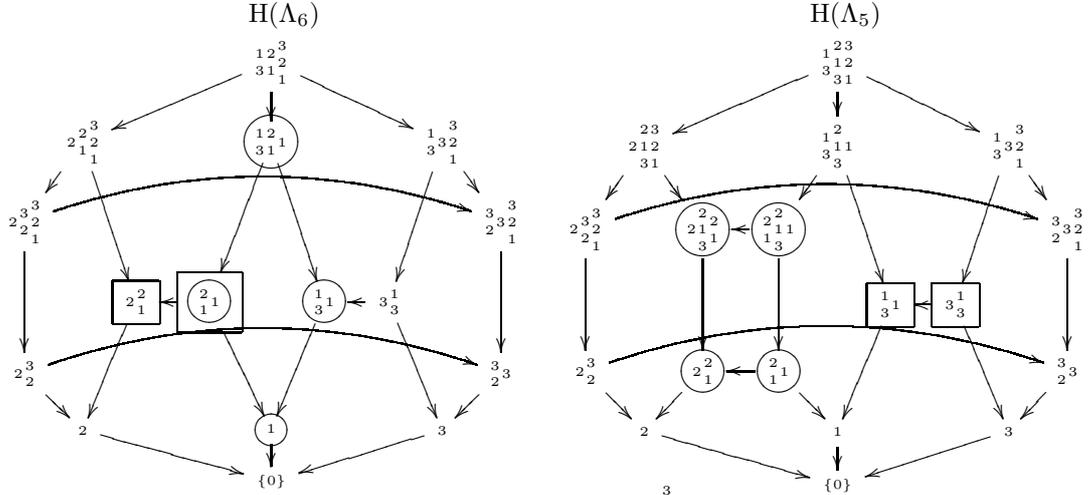
\put(0,402){(7) $\hasse(\Lambda_{4})$ is given below, 
where $\Lambda_{4}=\Lambda_{3}^{3}$ and $Q_{3}=\tiny\begin{smallmatrix}
3\\
2\\
1\\
3
\end{smallmatrix}$:}
\put(0,387){(8) $\hasse(\Lambda_{3})$ is given below, 
where $\Lambda_{3}=K\vec{\Delta}_{3}/\langle \alpha_{3}\alpha_{2}\alpha_{1}, \alpha_{1}\alpha_{3}\alpha_{2}\rangle$ and $Q_{2}=\tiny\begin{smallmatrix}
2\\
1\\
3\\
2
\end{smallmatrix}$:}
\put(95,370){{\small$\hasse(\Lambda_{4})$}}
\put(305,370){{\small$\hasse(\Lambda_{3})$}}
\put(0,352){$
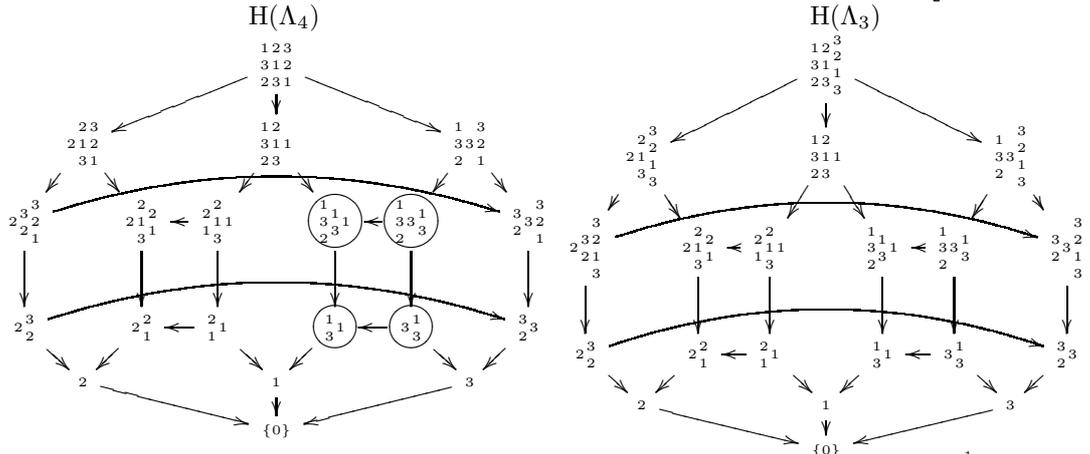
\xymatrix @R=2.5mm @C=0mm{
 & & & & &{\tiny\begin{smallmatrix}
1\\
3\\
2
\end{smallmatrix}\hspace{-1mm}\begin{smallmatrix}
2\\
1\\
3
\end{smallmatrix}\hspace{-1mm}\begin{smallmatrix}
3\\
2\\
1
\end{smallmatrix}}\ar[dllll]\ar[d]\ar[drrrr]& & & & & \\
 &{\tiny\begin{smallmatrix}
2
\end{smallmatrix}\hspace{-1mm}\begin{smallmatrix}
2\\
1\\
3
\end{smallmatrix}\hspace{-1mm}\begin{smallmatrix}
3\\
2\\
1
\end{smallmatrix}}\ar[dl]\ar[dr]& & & &{\tiny\begin{smallmatrix}
1\\
3\\
2
\end{smallmatrix}\hspace{-1mm}\begin{smallmatrix}
2\\
1\\
3
\end{smallmatrix}\hspace{-1mm}\begin{smallmatrix}
1
\end{smallmatrix}}\ar[dl]\ar[dr]& & & &{\tiny\begin{smallmatrix}
1\\
3\\
2
\end{smallmatrix}\hspace{-1mm}\begin{smallmatrix}
3
\end{smallmatrix}\hspace{-1mm}\begin{smallmatrix}
3\\
2\\
1
\end{smallmatrix}}\ar[dl]\ar[dr]& \\
{\tiny\begin{smallmatrix}
2
\end{smallmatrix}\hspace{-1mm}\begin{smallmatrix}
3\\
2
\end{smallmatrix}\hspace{-1mm}\begin{smallmatrix}
3\\
2\\
1
\end{smallmatrix}}\ar[dd]\ar@/^6mm/[rrrrrrrrrr]& &{\tiny\begin{smallmatrix}
2
\end{smallmatrix}\hspace{-1mm}\begin{smallmatrix}
2\\
1\\
3
\end{smallmatrix}\hspace{-1mm}\begin{smallmatrix}
2\\
1
\end{smallmatrix}}\ar[dd]& &{\tiny\begin{smallmatrix}
2\\
1
\end{smallmatrix}\hspace{-1mm}\begin{smallmatrix}
2\\
1\\
3
\end{smallmatrix}\hspace{-1mm}\begin{smallmatrix}
1
\end{smallmatrix}}\ar[dd]\ar[ll]& &*+[Fo]{\tiny\begin{smallmatrix}
1\\
3\\
2
\end{smallmatrix}\hspace{-1mm}\begin{smallmatrix}
1\\
3
\end{smallmatrix}\hspace{-1mm}\begin{smallmatrix}
1
\end{smallmatrix}}\ar[dd]& &*+[Fo]{\tiny\begin{smallmatrix}
1\\
3\\
2
\end{smallmatrix}\hspace{-1mm}\begin{smallmatrix}
3
\end{smallmatrix}\hspace{-1mm}\begin{smallmatrix}
1\\
3
\end{smallmatrix}}\ar[dd]\ar[ll]& &{\tiny\begin{smallmatrix}
3\\
2
\end{smallmatrix}\hspace{-1mm}\begin{smallmatrix}
3
\end{smallmatrix}\hspace{-1mm}\begin{smallmatrix}
3\\
2\\
1
\end{smallmatrix}}\ar[dd] \\
 & & & & & & & & & & \\
{\tiny\begin{smallmatrix}
2
\end{smallmatrix}\hspace{-1mm}\begin{smallmatrix}
3\\
2
\end{smallmatrix}}\ar[dr]\ar@/^6mm/[rrrrrrrrrr]& &{\tiny\begin{smallmatrix}
2
\end{smallmatrix}\hspace{-1mm}\begin{smallmatrix}
2\\
1
\end{smallmatrix}}\ar[dl]& &{\tiny\begin{smallmatrix}
2\\
1
\end{smallmatrix}\hspace{-1mm}\begin{smallmatrix}
1
\end{smallmatrix}}\ar[ll]\ar[dr]& &*+[Fo]{\tiny\begin{smallmatrix}
1\\
3
\end{smallmatrix}\hspace{-1mm}\begin{smallmatrix}
1
\end{smallmatrix}}\ar[dl]& &*+[Fo]{\tiny\begin{smallmatrix}
3
\end{smallmatrix}\hspace{-1mm}\begin{smallmatrix}
1\\
3
\end{smallmatrix}}\ar[ll]\ar[dr]& &{\tiny\begin{smallmatrix}
3\\
2
\end{smallmatrix}\hspace{-1mm}\begin{smallmatrix}
3
\end{smallmatrix}}\ar[dl] \\
 &{\tiny\begin{smallmatrix}
2
\end{smallmatrix}}\ar[drrrr]& & & &{\tiny\begin{smallmatrix}
1
\end{smallmatrix}}\ar[d]& & & &{\tiny\begin{smallmatrix}
3
\end{smallmatrix}}\ar[dllll]& \\
 & & & & &{\tiny\begin{smallmatrix}
\{0\}
\end{smallmatrix}}& & & & & \\
}
$}
\put(210,352){$
\xymatrix @R=2mm @C=-0.3mm{
 & & & & &{\tiny\begin{smallmatrix}
1\\
3\\
2
\end{smallmatrix}\hspace{-1mm}\begin{smallmatrix}
2\\
1\\
3
\end{smallmatrix}\hspace{-1mm}\begin{smallmatrix}
3\\
2\\
1\\
3
\end{smallmatrix}}\ar[dllll]\ar[d]\ar[drrrr]& & & & & \\
 &{\tiny\begin{smallmatrix}
2
\end{smallmatrix}\hspace{-1mm}\begin{smallmatrix}
2\\
1\\
3
\end{smallmatrix}\hspace{-1mm}\begin{smallmatrix}
3\\
2\\
1\\
3
\end{smallmatrix}}\ar[dl]\ar[dr]& & & &{\tiny\begin{smallmatrix}
1\\
3\\
2
\end{smallmatrix}\hspace{-1mm}\begin{smallmatrix}
2\\
1\\
3
\end{smallmatrix}\hspace{-1mm}\begin{smallmatrix}
1
\end{smallmatrix}}\ar[dl]\ar[dr]& & & &{\tiny\begin{smallmatrix}
1\\
3\\
2
\end{smallmatrix}\hspace{-1mm}\begin{smallmatrix}
3
\end{smallmatrix}\hspace{-1mm}\begin{smallmatrix}
3\\
2\\
1\\
3
\end{smallmatrix}}\ar[dl]\ar[dr]& \\
{\tiny\begin{smallmatrix}
2
\end{smallmatrix}\hspace{-1mm}\begin{smallmatrix}
3\\
2
\end{smallmatrix}\hspace{-1mm}\begin{smallmatrix}
3\\
2\\
1\\
3
\end{smallmatrix}}\ar[dd]\ar@/^6mm/[rrrrrrrrrr]& &{\tiny\begin{smallmatrix}
2
\end{smallmatrix}\hspace{-1mm}\begin{smallmatrix}
2\\
1\\
3
\end{smallmatrix}\hspace{-1mm}\begin{smallmatrix}
2\\
1
\end{smallmatrix}}\ar[dd]& &{\tiny\begin{smallmatrix}
2\\
1
\end{smallmatrix}\hspace{-1mm}\begin{smallmatrix}
2\\
1\\
3
\end{smallmatrix}\hspace{-1mm}\begin{smallmatrix}
1
\end{smallmatrix}}\ar[dd]\ar[ll]& &{\tiny\begin{smallmatrix}
1\\
3\\
2
\end{smallmatrix}\hspace{-1mm}\begin{smallmatrix}
1\\
3
\end{smallmatrix}\hspace{-1mm}\begin{smallmatrix}
1
\end{smallmatrix}}\ar[dd]& &{\tiny\begin{smallmatrix}
1\\
3\\
2
\end{smallmatrix}\hspace{-1mm}\begin{smallmatrix}
3
\end{smallmatrix}\hspace{-1mm}\begin{smallmatrix}
1\\
3
\end{smallmatrix}}\ar[dd]\ar[ll]& &{\tiny\begin{smallmatrix}
3\\
2
\end{smallmatrix}\hspace{-1mm}\begin{smallmatrix}
3
\end{smallmatrix}\hspace{-1mm}\begin{smallmatrix}
3\\
2\\
1\\
3
\end{smallmatrix}}\ar[dd] \\
 & & & & & & & & & & \\
{\tiny\begin{smallmatrix}
2
\end{smallmatrix}\hspace{-1mm}\begin{smallmatrix}
3\\
2
\end{smallmatrix}}\ar[dr]\ar@/^6mm/[rrrrrrrrrr]& &{\tiny\begin{smallmatrix}
2
\end{smallmatrix}\hspace{-1mm}\begin{smallmatrix}
2\\
1
\end{smallmatrix}}\ar[dl]& &{\tiny\begin{smallmatrix}
2\\
1
\end{smallmatrix}\hspace{-1mm}\begin{smallmatrix}
1
\end{smallmatrix}}\ar[ll]\ar[dr]& &{\tiny\begin{smallmatrix}
1\\
3
\end{smallmatrix}\hspace{-1mm}\begin{smallmatrix}
1
\end{smallmatrix}}\ar[dl]& &{\tiny\begin{smallmatrix}
3
\end{smallmatrix}\hspace{-1mm}\begin{smallmatrix}
1\\
3
\end{smallmatrix}}\ar[ll]\ar[dr]& &{\tiny\begin{smallmatrix}
3\\
2
\end{smallmatrix}\hspace{-1mm}\begin{smallmatrix}
3
\end{smallmatrix}}\ar[dl] \\
 &{\tiny\begin{smallmatrix}
2
\end{smallmatrix}}\ar[drrrr]& & & &{\tiny\begin{smallmatrix}
1
\end{smallmatrix}}\ar[d]& & & &{\tiny\begin{smallmatrix}
3
\end{smallmatrix}}\ar[dllll]& \\
 & & & & &{\tiny\begin{smallmatrix}
\{0\}
\end{smallmatrix}}& & & & & \\
}
$}

\put(0,195){(9) $\hasse(\Lambda_{2})$ is given below, 
where $\Lambda_{2}=K\vec{\Delta}_{3}/\langle \alpha_{3}\alpha_{2}\alpha_{1}, \alpha_{2}\alpha_{1}\alpha_{3}\alpha_{2}\rangle$ and $Q_{1}=\tiny\begin{smallmatrix}
1\\
3\\
2\\
1
\end{smallmatrix}$:}
\put(0,180){(10) $\hasse(\Lambda_{1})$ is given below, where $\Lambda_{1}=\Lambda_{3}^{4}$:}
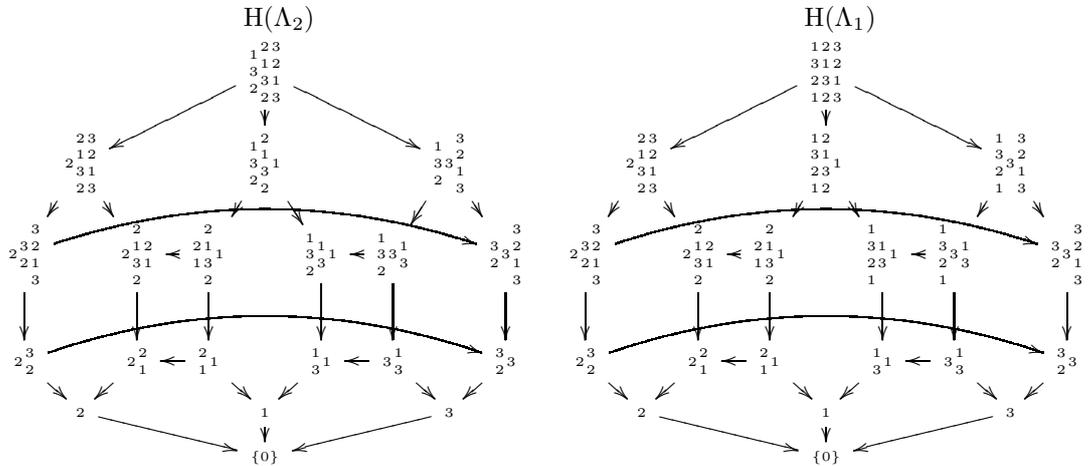
\put(93,165){{\small$\hasse(\Lambda_{2})$}}
\put(303,165){{\small$\hasse(\Lambda_{1})$}}
\put(0,145){$
\xymatrix @R=2mm @C=-0.3mm{
 & & & & &{\tiny\begin{smallmatrix}
1\\
3\\
2
\end{smallmatrix}\hspace{-1mm}\begin{smallmatrix}
2\\
1\\
3\\
2
\end{smallmatrix}\hspace{-1mm}\begin{smallmatrix}
3\\
2\\
1\\
3
\end{smallmatrix}}\ar[dllll]\ar[d]\ar[drrrr]& & & & & \\
 &{\tiny\begin{smallmatrix}
2
\end{smallmatrix}\hspace{-1mm}\begin{smallmatrix}
2\\
1\\
3\\
2
\end{smallmatrix}\hspace{-1mm}\begin{smallmatrix}
3\\
2\\
1\\
3
\end{smallmatrix}}\ar[dl]\ar[dr]& & & &{\tiny\begin{smallmatrix}
1\\
3\\
2
\end{smallmatrix}\hspace{-1mm}\begin{smallmatrix}
2\\
1\\
3\\
2
\end{smallmatrix}\hspace{-1mm}\begin{smallmatrix}
1
\end{smallmatrix}}\ar[dl]\ar[dr]& & & &{\tiny\begin{smallmatrix}
1\\
3\\
2
\end{smallmatrix}\hspace{-1mm}\begin{smallmatrix}
3
\end{smallmatrix}\hspace{-1mm}\begin{smallmatrix}
3\\
2\\
1\\
3
\end{smallmatrix}}\ar[dl]\ar[dr]& \\
{\tiny\begin{smallmatrix}
2
\end{smallmatrix}\hspace{-1mm}\begin{smallmatrix}
3\\
2
\end{smallmatrix}\hspace{-1mm}\begin{smallmatrix}
3\\
2\\
1\\
3
\end{smallmatrix}}\ar[dd]\ar@/^6mm/[rrrrrrrrrr]& &{\tiny\begin{smallmatrix}
2
\end{smallmatrix}\hspace{-1mm}\begin{smallmatrix}
2\\
1\\
3\\
2
\end{smallmatrix}\hspace{-1mm}\begin{smallmatrix}
2\\
1
\end{smallmatrix}}\ar[dd]& &{\tiny\begin{smallmatrix}
2\\
1
\end{smallmatrix}\hspace{-1mm}\begin{smallmatrix}
2\\
1\\
3\\
2
\end{smallmatrix}\hspace{-1mm}\begin{smallmatrix}
1
\end{smallmatrix}}\ar[dd]\ar[ll]& &{\tiny\begin{smallmatrix}
1\\
3\\
2
\end{smallmatrix}\hspace{-1mm}\begin{smallmatrix}
1\\
3
\end{smallmatrix}\hspace{-1mm}\begin{smallmatrix}
1
\end{smallmatrix}}\ar[dd]& &{\tiny\begin{smallmatrix}
1\\
3\\
2
\end{smallmatrix}\hspace{-1mm}\begin{smallmatrix}
3
\end{smallmatrix}\hspace{-1mm}\begin{smallmatrix}
1\\
3
\end{smallmatrix}}\ar[dd]\ar[ll]& &{\tiny\begin{smallmatrix}
3\\
2
\end{smallmatrix}\hspace{-1mm}\begin{smallmatrix}
3
\end{smallmatrix}\hspace{-1mm}\begin{smallmatrix}
3\\
2\\
1\\
3
\end{smallmatrix}}\ar[dd] \\
 & & & & & & & & & & \\
{\tiny\begin{smallmatrix}
2
\end{smallmatrix}\hspace{-1mm}\begin{smallmatrix}
3\\
2
\end{smallmatrix}}\ar[dr]\ar@/^6mm/[rrrrrrrrrr]& &{\tiny\begin{smallmatrix}
2
\end{smallmatrix}\hspace{-1mm}\begin{smallmatrix}
2\\
1
\end{smallmatrix}}\ar[dl]& &{\tiny\begin{smallmatrix}
2\\
1
\end{smallmatrix}\hspace{-1mm}\begin{smallmatrix}
1
\end{smallmatrix}}\ar[ll]\ar[dr]& &{\tiny\begin{smallmatrix}
1\\
3
\end{smallmatrix}\hspace{-1mm}\begin{smallmatrix}
1
\end{smallmatrix}}\ar[dl]& &{\tiny\begin{smallmatrix}
3
\end{smallmatrix}\hspace{-1mm}\begin{smallmatrix}
1\\
3
\end{smallmatrix}}\ar[ll]\ar[dr]& &{\tiny\begin{smallmatrix}
3\\
2
\end{smallmatrix}\hspace{-1mm}\begin{smallmatrix}
3
\end{smallmatrix}}\ar[dl] \\
 &{\tiny\begin{smallmatrix}
2
\end{smallmatrix}}\ar[drrrr]& & & &{\tiny\begin{smallmatrix}
1
\end{smallmatrix}}\ar[d]& & & &{\tiny\begin{smallmatrix}
3
\end{smallmatrix}}\ar[dllll]& \\
 & & & & &{\tiny\begin{smallmatrix}
\{0\}
\end{smallmatrix}}& & & & & \\
}
$}
\put(210,145){$
\xymatrix @R=2mm @C=-0.3mm{
 & & & & &{\tiny\begin{smallmatrix}
1\\
3\\
2\\
1
\end{smallmatrix}\hspace{-1mm}\begin{smallmatrix}
2\\
1\\
3\\
2
\end{smallmatrix}\hspace{-1mm}\begin{smallmatrix}
3\\
2\\
1\\
3
\end{smallmatrix}}\ar[dllll]\ar[d]\ar[drrrr]& & & & & \\
 &{\tiny\begin{smallmatrix}
2
\end{smallmatrix}\hspace{-1mm}\begin{smallmatrix}
2\\
1\\
3\\
2
\end{smallmatrix}\hspace{-1mm}\begin{smallmatrix}
3\\
2\\
1\\
3
\end{smallmatrix}}\ar[dl]\ar[dr]& & & &{\tiny\begin{smallmatrix}
1\\
3\\
2\\
1
\end{smallmatrix}\hspace{-1mm}\begin{smallmatrix}
2\\
1\\
3\\
2
\end{smallmatrix}\hspace{-1mm}\begin{smallmatrix}
1
\end{smallmatrix}}\ar[dl]\ar[dr]& & & &{\tiny\begin{smallmatrix}
1\\
3\\
2\\
1
\end{smallmatrix}\hspace{-1mm}\begin{smallmatrix}
3
\end{smallmatrix}\hspace{-1mm}\begin{smallmatrix}
3\\
2\\
1\\
3
\end{smallmatrix}}\ar[dl]\ar[dr]& \\
{\tiny\begin{smallmatrix}
2
\end{smallmatrix}\hspace{-1mm}\begin{smallmatrix}
3\\
2
\end{smallmatrix}\hspace{-1mm}\begin{smallmatrix}
3\\
2\\
1\\
3
\end{smallmatrix}}\ar[dd]\ar@/^6mm/[rrrrrrrrrr]& &{\tiny\begin{smallmatrix}
2
\end{smallmatrix}\hspace{-1mm}\begin{smallmatrix}
2\\
1\\
3\\
2
\end{smallmatrix}\hspace{-1mm}\begin{smallmatrix}
2\\
1
\end{smallmatrix}}\ar[dd]& &{\tiny\begin{smallmatrix}
2\\
1
\end{smallmatrix}\hspace{-1mm}\begin{smallmatrix}
2\\
1\\
3\\
2
\end{smallmatrix}\hspace{-1mm}\begin{smallmatrix}
1
\end{smallmatrix}}\ar[dd]\ar[ll]& &{\tiny\begin{smallmatrix}
1\\
3\\
2\\
1
\end{smallmatrix}\hspace{-1mm}\begin{smallmatrix}
1\\
3
\end{smallmatrix}\hspace{-1mm}\begin{smallmatrix}
1
\end{smallmatrix}}\ar[dd]& &{\tiny\begin{smallmatrix}
1\\
3\\
2\\
1
\end{smallmatrix}\hspace{-1mm}\begin{smallmatrix}
3
\end{smallmatrix}\hspace{-1mm}\begin{smallmatrix}
1\\
3
\end{smallmatrix}}\ar[dd]\ar[ll]& &{\tiny\begin{smallmatrix}
3\\
2
\end{smallmatrix}\hspace{-1mm}\begin{smallmatrix}
3
\end{smallmatrix}\hspace{-1mm}\begin{smallmatrix}
3\\
2\\
1\\
3
\end{smallmatrix}}\ar[dd] \\
 & & & & & & & & & & \\
{\tiny\begin{smallmatrix}
2
\end{smallmatrix}\hspace{-1mm}\begin{smallmatrix}
3\\
2
\end{smallmatrix}}\ar[dr]\ar@/^6mm/[rrrrrrrrrr]& &{\tiny\begin{smallmatrix}
2
\end{smallmatrix}\hspace{-1mm}\begin{smallmatrix}
2\\
1
\end{smallmatrix}}\ar[dl]& &{\tiny\begin{smallmatrix}
2\\
1
\end{smallmatrix}\hspace{-1mm}\begin{smallmatrix}
1
\end{smallmatrix}}\ar[ll]\ar[dr]& &{\tiny\begin{smallmatrix}
1\\
3
\end{smallmatrix}\hspace{-1mm}\begin{smallmatrix}
1
\end{smallmatrix}}\ar[dl]& &{\tiny\begin{smallmatrix}
3
\end{smallmatrix}\hspace{-1mm}\begin{smallmatrix}
1\\
3
\end{smallmatrix}}\ar[ll]\ar[dr]& &{\tiny\begin{smallmatrix}
3\\
2
\end{smallmatrix}\hspace{-1mm}\begin{smallmatrix}
3
\end{smallmatrix}}\ar[dl] \\
 &{\tiny\begin{smallmatrix}
2
\end{smallmatrix}}\ar[drrrr]& & & &{\tiny\begin{smallmatrix}
1
\end{smallmatrix}}\ar[d]& & & &{\tiny\begin{smallmatrix}
3
\end{smallmatrix}}\ar[dllll]& \\
 & & & & &{\tiny\begin{smallmatrix}
\{0\}
\end{smallmatrix}}& & & & & \\
}
$}
\end{picture}

Note that the Hasse quivers $\Lambda_{i}$ $(i\le 4)$ are the same.
This is a consequence of Theorem \ref{3-10}.
\end{example}

\end{document}